\documentclass[preprint,aos]{imsart}
\RequirePackage{natbib}
\setattribute{journal}{name}{}

\RequirePackage[OT1]{fontenc}
\usepackage[hang,small,bf]{caption}
\RequirePackage{amsthm,amsmath,natbib,graphicx,enumitem}
\RequirePackage[colorlinks,citecolor=blue,urlcolor=blue]{hyperref}
\RequirePackage{hypernat}
\usepackage{amssymb,tabularx,multicol,multirow,booktabs}
\usepackage[top=1in, bottom=1in, left=1in, right=1in]{geometry}
\usepackage{mhequ}
\usepackage[usenames,dvipsnames]{color}
\usepackage{epstopdf}
\RequirePackage{xr}
\usepackage{subcaption}
\startlocaldefs
\numberwithin{equation}{section}
\theoremstyle{plain}
\endlocaldefs
\input binary_macro.tex
\def \bbeta{\boldsymbol{\beta}}
\def \beps{\boldsymbol{\varepsilon}}

\def \by{\mathbf{y}}
\def \bX{\mathbf{X}}
\def \risk{\mathrm{Risk}}

\def \be{\begin{equs}}
	\def \ee{\end{equs}}

\def \E{\mathbb{E}}
\def \P{\mathbb{P}}

\def \fhat{\hat{f}}

\def \sumn{\sum_{i=1}^n}

\def \Pzero {\mathbb{P}_{\mathbf{0}}}

\def \Bin {\mathrm{Bin}}

\def \Pbeta{\P_{\mathbf{\bbeta}}}

\begin{document}
\begin{frontmatter}
\title{On Minimax Exponents  of Sparse Testing}
\runtitle{Testing Rate}

\begin{aug}

\author{\fnms{Rajarshi} \snm{Mukherjee}\ead[label=e1]{rmukherj@hsph.harvard.edu}},
\author{\fnms{Subhabrata} \snm{Sen}\ead[label=e3]{subhabratasen@fas.harvard.edu}}

		\runauthor{R.~Mukherjee and S.~Sen}
\affiliation{Harvard University}

\address{Department of Biostatistics \\ Harvard University \\ 655 Huntington Avenue \\ Boston, MA 02115 \\
\printead{e1}}

\address{Department of Statistics \\
Harvard University\\
1 Oxford Street,\\
Cambridge, MA-02138  \\
\printead{e3}}

\end{aug}

\begin{abstract}
 We consider exact asymptotics of the minimax risk for global testing against sparse alternatives in the context of high dimensional linear regression. Our results characterize the leading order behavior of this minimax risk in several regimes, uncovering new phase transitions in its behavior. This complements a vast literature characterizing asymptotic consistency in this problem, and provides a useful benchmark, against which the performance of   specific tests may be compared. Finally, we provide some preliminary evidence that popular sparsity adaptive procedures might be sub-optimal in terms of the minimax risk.

 \end{abstract}
\end{frontmatter}

\section{Introduction}
Modern technological innovations have ushered in the age of  ``big" data, and large, high-dimensional datasets have become commonplace in applications from genetics, genomics, finance, communications etc. In these applications, it is often believed that the true signals are rare, and the effect sizes are weak--- thus often precluding hope of individual identification of the signal components. One is then faced with a fundamental statistical question--- ``Is it possible to detect the signal in the data?". For a concrete example, consider the setting of modern genetic association studies; given data on sequences across multiple candidate genes or even the whole genome, and a phenotypic response, one naturally wishes to determine whether there is any association between the genetic variants and the response. This is especially relevant when 
genetic effects are weak -- and hence identifying individual genetic variants is statistically harder compared to a global association testing  problem (\cite{visscher2012five}, \cite{lee2014}, \cite{li2008methods}). A statistical error in this setting has serious consequences--- a false negative misses  associations of fundamental scientific importance, while a false positive often prompts hopeless expensive follow up studies to discover the individual effects. This motivates two questions of basic interest

\begin{enumerate}
    \item When is accurate detection possible? 
    \item What is the smallest achievable statistical error for such a detection problem? 
\end{enumerate}

The first question has attracted significant attention in the Statistical literature over the last decade, while very little is understood about the second question. In this article, we initiate a study of this question in the context of linear regression.

Formally, we consider the Gaussian linear regression model
\be
\by=\bX \bbeta +\beps, \label{eqn:gaussian_linreg_model}
\ee
where $\beps\sim \mathcal{N}(0,I_n)$ is a vector of Gaussian white noise, $\bX$ is a $n\times p$ real (random) matrix independent of $\beps$, and $\bbeta =(\beta_1,\ldots,\beta_p)^T\in \mathbb{R}^p$ is an unknown parameter vector of interest. Throughout, we shall work with the high dimensional set up where $p\rightarrow \infty$ and $n:=n(p)\rightarrow \infty$. Further, we assume that the error variance is known throughout --- extending our results to the unknown variance setting will likely require significant new ideas, and is beyond the scope of this paper. 



In this article, we study the \emph{signal detection problem} against sparse alternatives in the setting \eqref{eqn:gaussian_linreg_model}. Specifically, consider the following sequence of hypothesis testing problems 
\be 
H_0: \bbeta =\mathbf{0} \quad \textrm{vs.} \quad H_1: \bbeta \in \Xi(s_p, A_p) \subset \mathbb{R}^p\setminus \{\mathbf{0}\}, \label{eqn:hypo}
\ee
indexed by a pair of sequences $s_p, A_p$, where $\Xi(s,A)$ denotes the parameter space 
\be
\Xi(s,A):=\{{\bbeta}\in \mathbb{R}^p: |\mathrm{supp}(\bbeta)|=s, |\beta_i|\geq A \,\,\,\textrm{if}\,\,\,  i\in \textrm{supp}(\bbeta), \beta_i =0 \,\,\, \textrm{o.w.}\}, \label{eq:parameterspace}
\ee
and $\mathrm{supp}(\bbeta)=\{j:\bbeta_j\neq 0\}$. Throughout, $s$ will be referred to as the sparsity and $A$ as the signal strength of $\bbeta\in \Xi(s,A)$.
Henceforth, whenever the context is clear, we drop the subscript $p$ from $s_p$ and $A_p$. We note that one can also define sparse signals by considering at most $s$ non-zero coordinates of $\bbeta$ and separation from $0$ governed by magnitude of $\|\bbeta\|_2$ (instead of each non-zero coordinate being large in absolute value). However, for the sake of conveying the main ideas we only work with $\Xi(s,A)$ described above.


A natural statistical question in this context concerns the minimum signal strength $A$ (for a given sparsity $s$) which guarantees consistent detection. Research in this direction can be traced back to the seminal results of \cite{burnashev1979minimax, ingster1994minimax,ingster1995minimax,ingster1998minimax,ingster2012nonparametric} on Gaussian white noise models. 
Following tradition, questions of this flavor will be referred to as the ``minimax separation rate" problem (henceforth, we will often refer to this behavior as the first order behavior of the problem). The minimax separation rates for the sparse normal means type problem were subsequently detailed in \cite[Chapter 8]{ingster2012nonparametric} and \cite{donoho2004higher}, \cite{hall2010innovated}, \cite{cai2011optimal}, \cite{cai2014optimal}. Finally, the Gaussian linear regression version of the problem, relevant to this paper, was solved simultaneously (under slightly different assumptions on design distributions and form of alternatives) in \cite{arias2011global} and \cite{ingster2010detection}. For non-asymptotic analogues of the minimax separation problem we refer the interested reader to \cite{baraud2002non} (for Gaussian sequence models) and \cite{carpentier2018minimax} (for Gaussian linear regression), and references therein. 

Although the theory of signal detection essentially emerged from somewhat information theoretic considerations, it has found widespread appeal in diverse modern applications arising from biology, engineering, and the social sciences (we refer the interested reader to \cite{arias2011global} for a discussion of the practical motivations). In turn, several testing procedures have been developed which attain the detection boundary, i.e. are consistent whenever the signal strength is larger than information theoretic minimum. Notable procedures include the Higher Criticism, Generalized Higher Criticism, Minimum p-value, Berk-Jones Test, Averaged Likelihood Ratio and SKAT--- we refer the interested reader to \citep{barnett2017generalized,jin2016invited,walther2013average,cai2011optimal,wu2011rare,jin2003detecting,arias2011global,sun2019powerful,zhong2013tests,fan2015power} for some notable results in this research direction. The popularity and abundance of many such tests have already prompted researchers to compare these procedures \citep{li2015higher,porter2019hc} beyond asymptotic consistency type behavior. We explore a concrete decision theoretic formalization of this perspective. 

\subsection{Formulation}
\label{sec:formulation}

A sequence of tests $T_p(\mathbf{y},\bX)$ is a $[0,1]$-valued measurable function of the data $(\mathbf{y},\bX)$. To introduce our decision theoretic setup, we will need some preliminary notation. Denote the law of $(\mathbf{y},\bX)$ as $\P_{\mathbf{\bbeta}}$, and expectations under this law as $\E_{\mathbf{\bbeta}}[\cdot]$. Finally, we use $\P_{0}$ to refer to $\P_{\mathbf{0}}$ (i.e. when $\bbeta=\mathbf{0}$) and denote the corresponding expectation simply as $\E_{0}[\cdot]$.
 For any sequence of tests $T_p$, we define the maximum risk over $\Xi(s,A)$ as
\be
\mathrm{Risk}(T_p,s,A) = \E_{0}[T_p] + \sup_{\bbeta \in \Xi(s,A)} \E_{\bbeta}[1- T_p]. \label{eq:risk_test}
\ee
This is simply the risk of any test under the $0-1$ loss. As usual, the minimax risk for this problem is defined as 
\be
\mathrm{Risk}(s,A) := \inf_{T_p} \mathrm{Risk}(T_p,s,A), \label{eq:minimax_risk}
\ee
where the infimum is taken over all test sequences $\{T_p\}$. Thus to upper bound the minimax risk, it suffices to analyze the maximum risk of any sequence of tests $T_p$. On the other hand, to lower bound the minimax risk, we will crucially use the Bayes risk under the uniform prior on the ``boundary" of $\Xi(s,A)$, defined as 
\be \label{eqn:bayes_risk}
\mathrm{BRisk}(T_p,s,A)&:=\E_{0}[T_p]+\E_{\mathbf{\bbeta \sim \pi}}\Big[\E_{\bbeta}[1 -T_p ] \Big], \label{eq:bayes_risk}
\ee 
where $\pi$ denotes the uniform distribution on $\tilde\Xi(s,A)$ defined as
\be
\tilde\Xi(s,A):=\{{\bbeta}\in \mathbb{R}^p: |\mathrm{supp}(\bbeta)|=s, |\beta_i|= A \,\,\,\textrm{if}\,\,\,  i\in \textrm{supp}(\bbeta), \beta_i =0 \,\,\, \textrm{o.w.}\}. \label{eq:parameterspace_boundary}
\ee
 In this notation, a sequence of tests $T_p$ is said to be asymptotically powerful (respectively asymptotically powerless) if 
\be
\limsup_{p \to \infty} \mathrm{Risk}(T_p, s,A) = 1 \quad( \textrm{respectively} \liminf_{p\to \infty} \mathrm{Risk}(T_p, s, A) =0). \label{eqn:minimax_first_order}
\ee

To explain our problem in the context of the vast literature mentioned above, it is worth recalling the existing results available in the linear regression set up. 
For sparsities $s \to \infty$ with $\log s/ \log p \to 0$, it is well-known that an asymptotically powerful test sequence exists if and only if $\limsup A/ \sqrt{2 \log(p)/n}\geq 1$. The behavior of the minimax separation rate is more subtle for larger $s$. In the interest of  mathematical tractability, one can consider a polynomial  (in $p$) behavior of $s$, and parametrize $s = p^{1- \alpha}$, $\alpha \in (0,1)$. The minimax separation rates for this problem with orthogonal or subgaussian design (see Section \ref{sec:notations_and_assumptions} for a precise definition), derived in \cite{arias2011global, ingster2010detection}, builds on those for the Gaussian sparse means problem \citep{jin2003detecting, donoho2004higher, ingster2012nonparametric}, and can be described as follows (see Figure \ref{fig:detection_threshold} for an illustration of these thresholds).  For $\alpha \leq \frac{1}{2}$,  an asymptotically powerful test sequence exists if and only if $A^2 \gg p^{- (\frac{1}{2} - \alpha)}/n$. In contrast,  for $\alpha \in (\frac{1}{2} , 1)$, the minimax separation rate is sharp and is given by  $A = \sqrt{2 \rho^*(\alpha){\log{(p)}}/{n}}$, where 
\be
\rho^*(\alpha) = \begin{cases}
	\alpha - \frac{1}{2} & \textrm{if}\,\, \frac{1}{2} < \alpha < \frac{3}{4}, \\ 
	(1 - \sqrt{1 - \alpha})^2 & \textrm{ o.w.} 
\end{cases}\label{eq:detection_threshold}
\ee
\begin{figure}
\begin{subfigure}[b]{0.45\linewidth}
    \centering
    \includegraphics[scale=0.2]{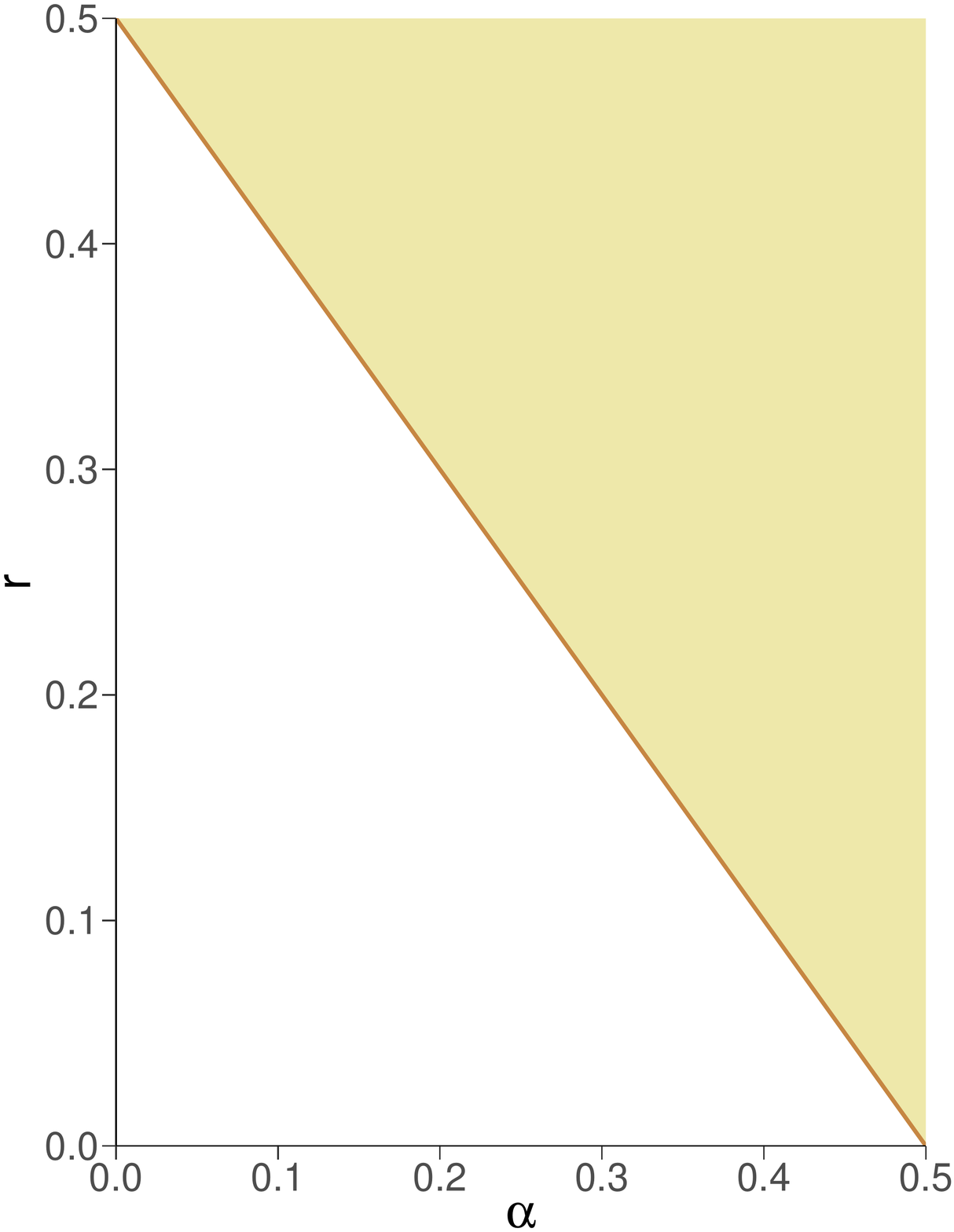}
    \caption{$\alpha \in \big(0, \frac{1}{2} \big)$}
\end{subfigure}
\begin{subfigure}[b]{0.45\linewidth}
    \centering
    \includegraphics[scale=0.2]{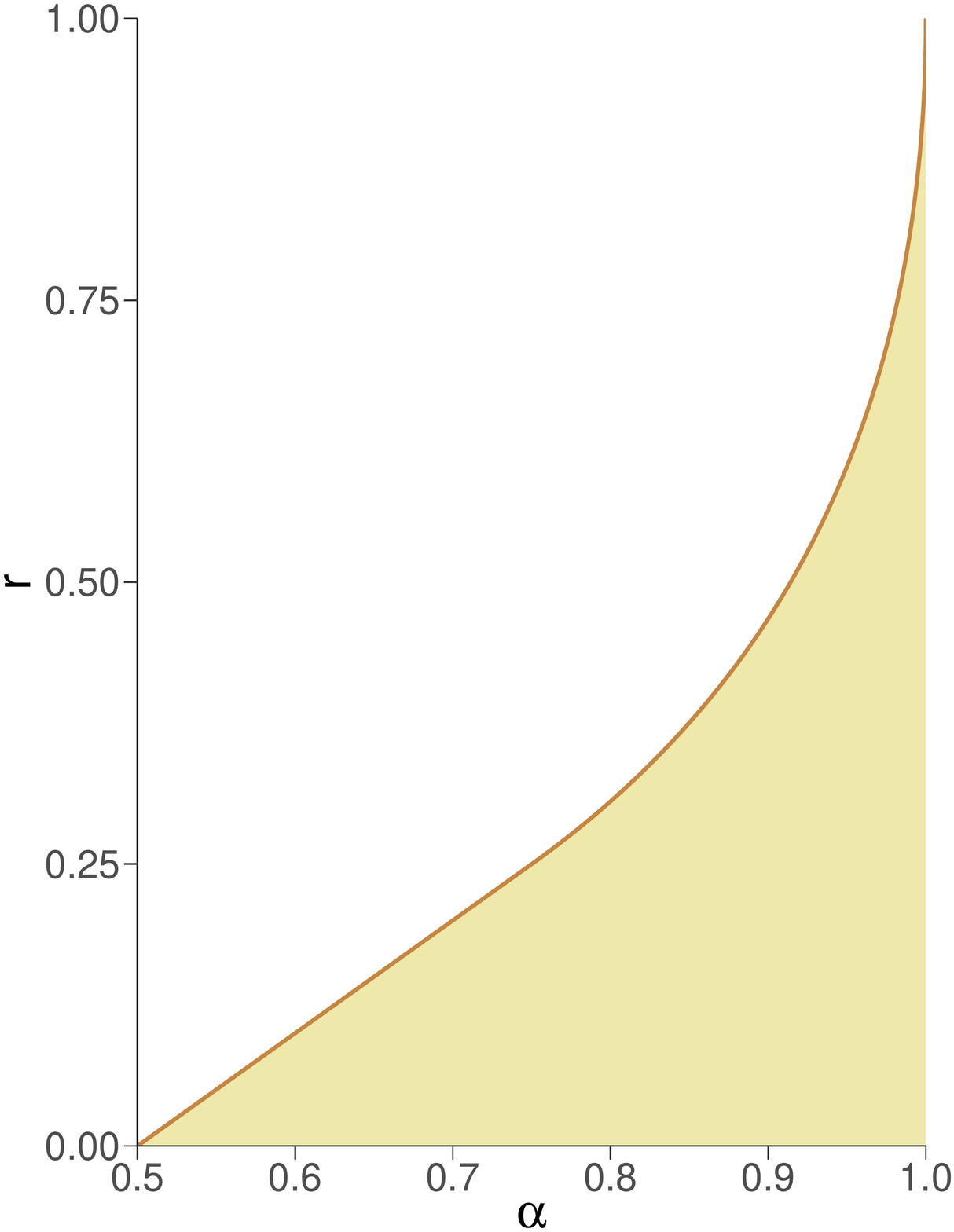}
    \caption{$\alpha \in \big(\frac{1}{2},1\big)$}
\end{subfigure}
    \caption{The detection thresholds for \ref{eqn:hypo}. In (a), we set $A= \sqrt{\frac{p^{-r}}{{n}}}$. In (b), we set $A= \sqrt{\frac{2r \log p}{n}}$. All tests are asymptotically powerless in the shaded region, while an asymptotically powerful test exists in the un-shaded region.}
    \label{fig:detection_threshold}
\end{figure}
\noindent
In this article, we take the natural next step, and study the behavior of $\mathrm{Risk}(s,A)$ for  various $(s,A)$ pairs (henceforth, we refer to this as the second order behavior of the problem). In this endeavor, we will study tests which utilize the knowledge of the sparsity $s$ and signal strength $A$. It is natural to wonder whether there exist sparsity \emph{agnostic} tests which are rate optimal in this setting --- this is a significant challenge, and considerably beyond the scope of this paper. On the other hand, we provide preliminary evidence suggesting that some popular sparsity agnostic tests might not be rate optimal in terms of the worst case risk (see Section \ref{sec:contributions} for an in-depth discussion of this point).


A thorough study of the asymptotics of the minimax risk $\mathrm{Risk}(s,A)$ involves the two distinct paradigms, which display substantially different behaviors. 

\begin{enumerate}
	\item \textbf{Below Boundary Problem:} Here $A$ is below the minimax separation boundary, and one seeks the asymptotic behavior of $1-\mathrm{Risk}(s,A)$. This represents the \textit{slowest rate} (over all test sequences $T_p$) at which the worst case risk \textit{converges to $1$}. 
	
	\item \textbf{Above Boundary Problem:} In this case, $A$ is above the minimax separation boundary, and one seeks the behavior of $\mathrm{Risk}(s,A)$. In this case, this represents the \textit{fastest rate} at which the worst case risk \textit{converges to $0$}. 
\end{enumerate}

\subsection{Background} 
In the setting of high-dimensional linear regression \ref{eqn:gaussian_linreg_model}, several inference problems related to the structure of the regression vector $\bbeta$ have been studied extensively in the prior literature. 
At this point, it is instructive to 
re-visit the widely believed hierarchy among these problems. 
As noted by \cite{wellner2008slides}, the three most common inference problems of this flavor, sorted in increasing order of difficulty, are 
\begin{itemize}
    \item[(i)] the global testing problem against sparse alternatives \eqref{eqn:hypo},
    \item[(ii)] the sparsity estimation problem, where one seeks to estimate $s/p$, and 
    \item[(iii)] the variable selection problem, where the statistician wishes to recover the support of the coefficient vector $\bbeta$. 
\end{itemize}
Each of these problems has attracted significant attention in the past two decades, and has engendered a vast literature. To keep our discussion focused, we will not try to survey the relevant literature for (ii) in depth, but point the curious reader to \cite{meinshausen2006estimating,cai2007estimation,jin2007estimating,jin2008proportion,carpentier2019adaptive} and the references therein, for a discussion of the specific progress attained on these questions.  

Early research on the variable selection problem focussed on exact support recovery (see e.g. \citet{zhao2006model,zhang2010nearly,wasserman2009high,wainwright2009sharp,lounici2008sup,ji2012ups,genovese2012comparison} and references therein). It is intuitively clear that the variable selection problem is intimately related to the marginal testing problem for individual coordinates of $\bbeta$. Indeed, \cite{butucea2018variable,ndaoud2018optimal} show this intuition to be correct. This connection, in turn, leads directly to a fine understanding of the second order behavior in the variable selection problem--- formally, this corresponds to the risk of estimating $\bbeta \in \Xi(s,A)$ in Hamming Loss. This significantly extends the first order understanding of the variable selection problems already explored in prior literature. \cite{butucea2018variable} argue persuasively that obtaining a deep understanding of the second order behavior is of paramount importance in the variable selection context, and that focusing solely on exact asymptotic variable selection necessarily leaves a large void in our understanding of the inherent challenges in the problem.

 Turning back to the global testing problem (i), we note that starting with the seminal works of Ingster \citep{ingster1994minimax,ingster1995minimax,ingster1998minimax,ingster2012nonparametric}, the main emphasis in this research direction has been placed on characterizing the first order behavior. This accomplishes an extremely important, and often technically challenging first step. However,   unlike the variable selection problem, the second order behavior of the global testing problem is completely unexplored. Motivated by these considerations, we take the first steps in filling this gap in the context of global testing for $\bbeta \in \Xi(s,A)$.


\subsection{Our Contributions}
\label{sec:contributions}
We summarize our main contributions under three main themes. 
\begin{itemize}\itemsep3pt
    \item[(i)] \textbf{Formulation and main results:} In this article, we initiate a study of the optimal risk behavior (w.r.t. $0/1$ loss) for the global testing problem against sparse alternatives, in the context of linear regression. 
    We provide tight asymptotics of the minimax risk $\mathrm{Risk}(s,A)$ \eqref{eq:minimax_risk} for various sparsity and signal strength combinations, and study the problem both above and below the detection boundary. To this end, a study of the minimax risk inevitably requires some assumptions on the design distribution $\bX$--- in this article, we consider both orthogonal designs and a class of isotropic sub-gaussian designs with i.i.d. rows (we defer formal definitions to Section \ref{sec:notations_and_assumptions}). Note that in the case of orthogonal designs, this problem is equivalent to the \emph{gaussian sequence model} \cite{ingster2012nonparametric}; however, this problem is unexplored even in this simple setting. For orthogonal designs, our results are valid as soon as $n\geq p$. On the other hand, the problem is significantly more complicated for subgaussian designs. In this case, our arguments require that $n$ grows significantly faster than $p$--- the specific dependence necessary is different for each theorem. While we expect some condition of this flavor to be unavoidable, to communicate our main ideas, we have not tried to optimize this dependence. In summarizing our main results below, we will suppress the specific dependence necessary, and refer the reader to the formal statements of our results in Section \ref{section:main_results} for the explicit conditions. Figure \ref{fig:detection_results} provides a visual summary of our main results. 
    \begin{figure}
\begin{subfigure}[b]{0.45\linewidth}
    \centering
    \includegraphics[scale=0.21]{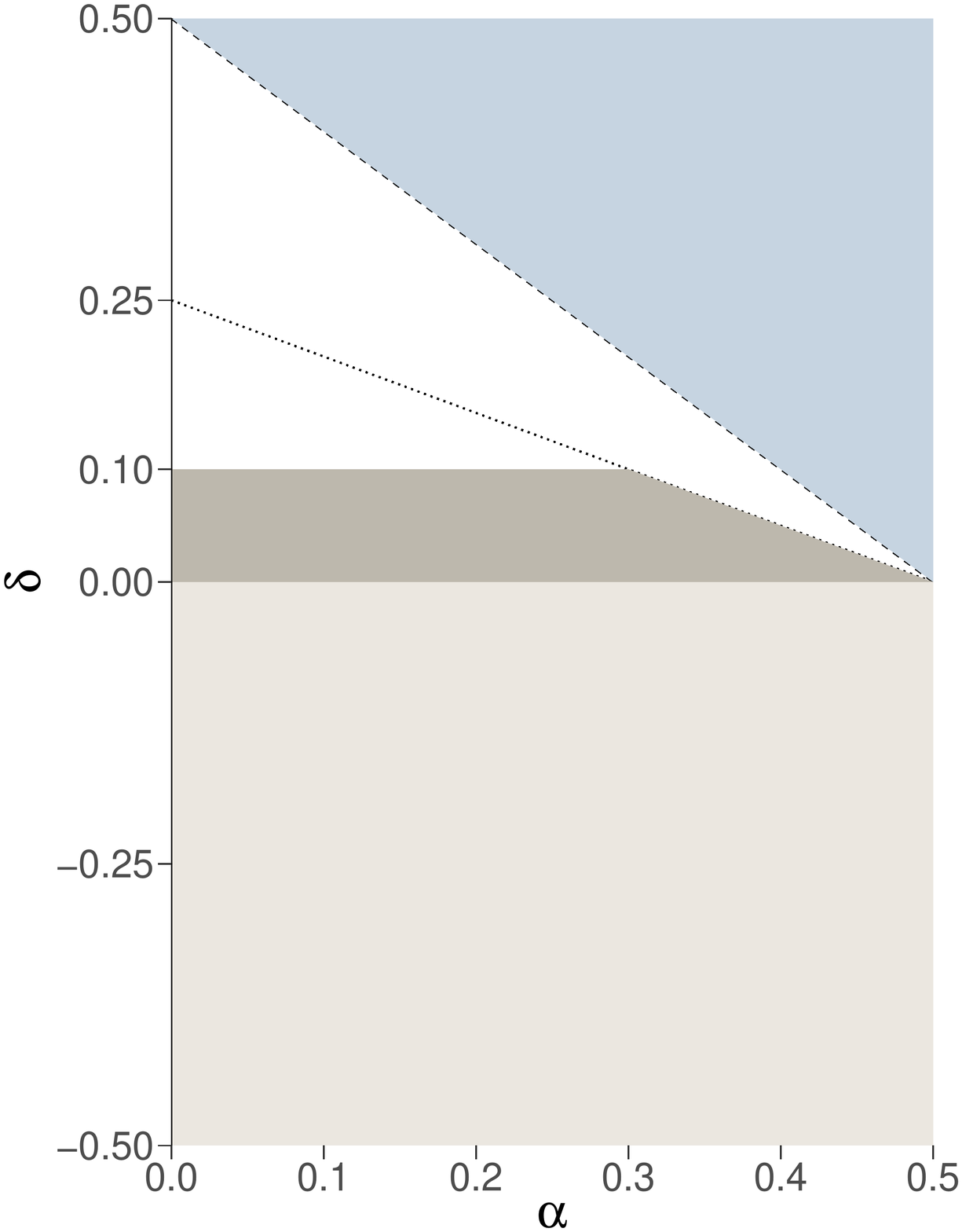}
    \caption{$\alpha \in \big(0, \frac{1}{2} \big)$}
\end{subfigure}
\begin{subfigure}[b]{0.45\linewidth}
    \centering
    \includegraphics[scale=0.21]{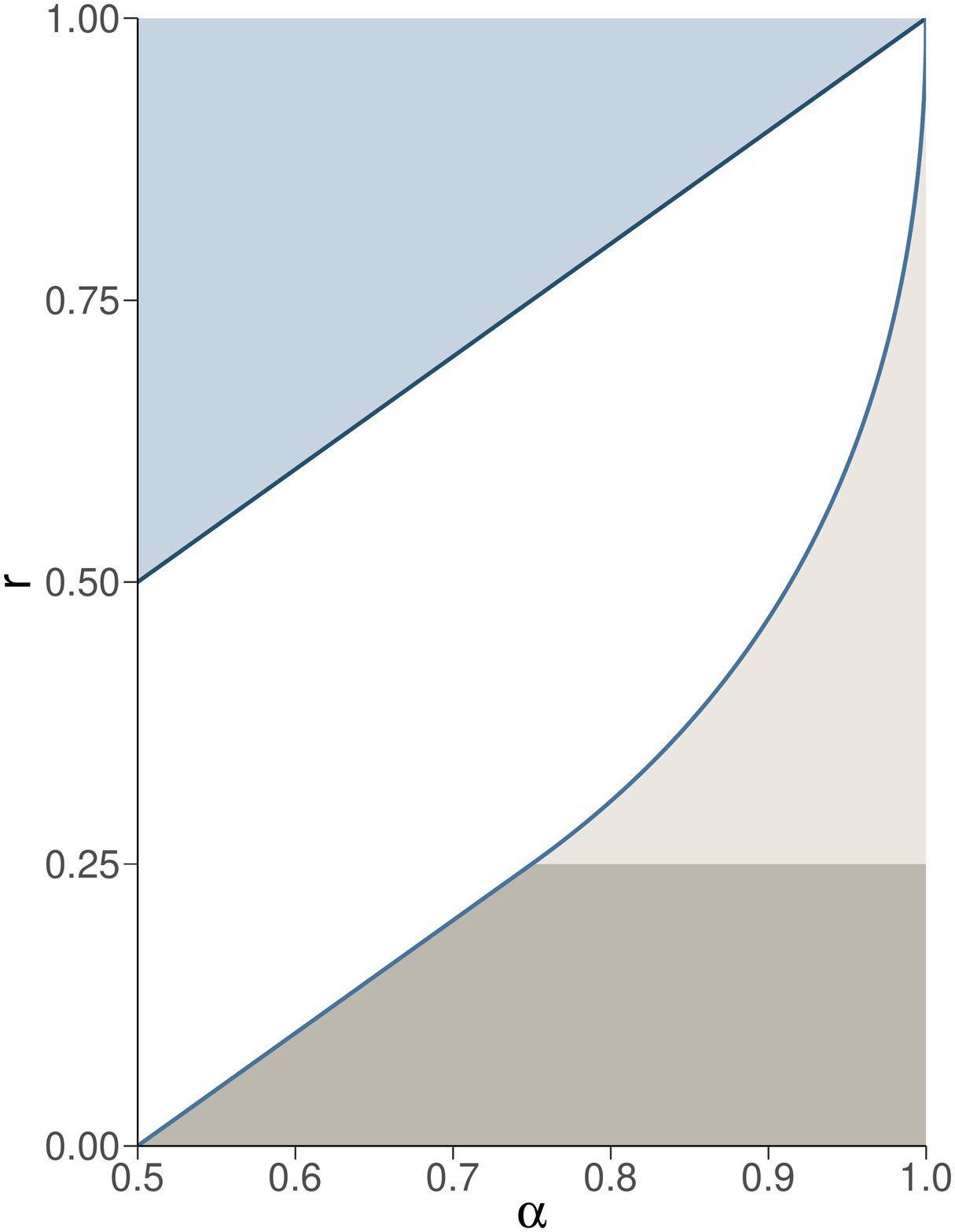}
    \caption{$\alpha \in \big(\frac{1}{2},1\big)$}
\end{subfigure}
    \caption{\scriptsize{A visual depiction of our main results. In (a), we set $A^2 = \frac{p^{\alpha- \frac{1}{2} + \delta}}{n}$. For $-\frac{1}{2}<\delta<0$ (shaded in light grey), Theorem \ref{theorem:below_boundary_dense} establishes that $\mathrm{Risk}(s,A)= 1-p^{-\delta+o(1)}$. For $0<\delta<\frac{1}{10}$, and $\alpha-\frac{1}{2}+ 2\delta<0$(shaded in dark grey), Theorem \ref{theorem:above_boundary_dense} establishes that $\mathrm{Risk}(s,A) = \exp{(-\frac{p^{2\delta}}{16}(1+o(1)))}$. Finally, for $0<\delta<\frac{1}{2}$ and $\frac{1}{2}-\alpha+\delta>0$ (indicated in blue), $\mathrm{Risk}(s,A) = \exp{(- \frac{p^{\frac{1}{2}+\delta}}{8} (1+o(1)))}$. In (b), we set $A= \sqrt{\frac{2r\log p}{n}}$. For $r<\rho^*(\alpha)$, and $4r<1$ (shaded in deep grey), Theorem \ref{theorem:below_boundary_sparse} establishes that $\mathrm{Risk}(s,A)= 1- p^{r- \alpha + \frac{1}{2} + o(1) }$, while for $4r>1$ (indicated in light grey), $\mathrm{Risk}(s,A) = 1- p^{1-\alpha - (1-\sqrt{r})^2+o(1)}$. Above the boundary, for $r>\alpha$ (indicated in blue), Theorem \ref{theorem:above_boundary_sparse} establishes that $\mathrm{Risk}(s,A) = \exp\big(-\frac{(r-\alpha)^2}{4r}s\log p (1+o(1)) \big)$. } }
    \label{fig:detection_results}
\end{figure}
We start with the behavior of the minimax risk below the detection boundary. 

\begin{itemize}
    \item[(a)] For $\alpha \in (0, \frac{1}{2}]$, all tests are asymptotically powerless whenever $n A^2 \ll p^{\alpha - \frac{1}{2}}$. To study the problem below detection boundary, we set $\sqrt{n}A = \sqrt{p^{\alpha-\frac{1}{2}-\delta}}$ for some $\delta>0$. In Theorem~\ref{theorem:below_boundary_dense}, we establish that for $\delta<\frac{1}{2}$, whenever $n$ is sufficiently large compared to $p$, $\mathrm{Risk}(s,A)= 1- p^{-\delta +o(1)}$.  
    
    \item[(b)] For $\alpha \in (\frac{1}{2},1)$, the detection boundary is sharp, and all tests are asymptotically powerless whenever $\sqrt{n}A \leq \sqrt{2\rho^*(\alpha) \log p}$, with $\rho^*(\alpha)$ as in \eqref{eq:detection_threshold}. In Theorem~\ref{theorem:below_boundary_sparse}, we set $\sqrt{n}A = \sqrt{2r \log p}$ for $0<r<\rho^*(\alpha)$, and establish that for $n$ sufficiently large compared to $p$, if $4r\leq 1$, $\mathrm{Risk(s,A)} = 1- p^{r-\alpha+\frac{1}{2}+o(1)}$. On the contrary, if $4r>1$, $\mathrm{Risk}(s,A) = 1- p^{1-\alpha-(1-\sqrt{r})^2+o(1)}$.  
\end{itemize}
Note that while deriving the first order detection boundary of the problem, it is unnecessary to analyze tests below the information theoretic threshold. In contrast, while exploring exact asymptotic behavior of $\mathrm{Risk}(s,A)$, we need to analyze the testing error of appropriate tests below the boundary, and establish that their worst case error grows to 1 at the slowest possible rate. This matches the behavior of the likelihood ratio test with respect to a sequence of least favorable priors.


Next, we turn to the behavior of the problem above the detection boundary. 
\begin{itemize}
    \item[(a)] In the dense signal regime ($\alpha\leq \frac{1}{2}$), we consider alternatives $\bbeta \in \Xi(s,A)$ with $\sqrt{n}A = \sqrt{p^{\alpha -\frac{1}{2}+\delta}}$ with $\delta>0$. First, we consider alternatives such that $\alpha - \frac{1}{2} + 2\delta <0$ and $\delta<\frac{1}{10}$. This corresponds to alternatives which are above the detection boundary, yet very close to it. In this case, we establish that as long as $n$ is significantly larger than $p$, $\mathrm{Risk}(s,A)=\exp{(- \frac{p^{2\delta}}{16}(1+o(1)) )}$. 
    Intriguingly, we establish that for a class of alternatives further away from the detection boundary, the minimax risk undergoes a ``phase transition" phenomenon. Specifically, consider alternatives such that $\alpha-\frac{1}{2}+\delta >0$ and $\delta < \frac{1}{2}$. In this case, we establish that for $n$ sufficiently large compared to $p$, 
    $\mathrm{Risk}(s,A) = \exp{(-\frac{p^{\frac{1}{2}+\delta}}{8}(1+o(1)))}$.
    To the best of our knowledge, this phase transition was not even conjectured in the earlier literature. 
    
    \item[(b)] Finally, we consider the sparse signal regime $\alpha>\frac{1}{2}$. In this case, consider alternatives with $\sqrt{n}A = \sqrt{2r\log p}$ with $r>\alpha$. For $n$ sufficiently larger than $p$, we establish that $\mathrm{Risk}(s,A)=\exp{(-\frac{(r-\alpha)^2}{4r} s\log p(1+o(1)) )}$. 
    Note that $r>\alpha$ is significantly above the detection threshold $r>\rho^*(\alpha)$. 
\end{itemize}  
So far, our results concentrate on the regime of polynomial $s$. These arguments can be adapted in a straight forward manner to study the minimax risk for sub-polynomial $s$ (i.e. $\log s \ll \log p$), where it characterizes the behavior of the minimax risk for all signal strengths above and below the detection threshold $\sqrt{\frac{2\log p}{n}}$. This leaves open the question ``\emph{What happens on the detection boundary?}". To answer this question, we turn to the setting where $s$ is fixed as $n,p \to \infty$. We consider the case where $\sqrt{n} A=\sqrt{2\log p}$. We establish that for $n$ sufficiently large compared to $p$, $\mathrm{Risk}(s,A) \to \Big(\frac{1}{2}\Big)^s$. An analogous result was established for gaussian sequence models in \cite[Theorem 8.1]{ingster2012nonparametric}. Thus, this extends their result to the linear regression model; however, we emphasize that this extension is not straight forward, and requires overcoming significant technical barriers.

    \item[(ii)]\textbf{Statistical price of sparsity adaptive procedures:}
    In the context of the signal detection problem \eqref{eqn:hypo}, the need for \emph{sparsity adaptive} procedures--- i.e., procedures which do not require a knowledge of the alternative $(s,A)$, and yet attain the detection boundary--- has been long recognized. Indeed, this was one of the main motivations behind the introduction of the Higher Criticism Statistic by \citet{donoho2004higher}, and subsequent developments in this line of research (see e.g. \cite{jager2007goodness,hall2008properties,arias2011global,cai2011optimal,ingster2010detection,zhong2013tests,qiu2018detecting}). But this begs the natural question: ``\emph{Do we pay a statistical price for adaptation?}"(in terms of the worst case risk). To the best of our knowledge, this question has not been looked at in prior research on this problem. We provide some preliminary evidence which suggests that there might be a significant difference in the behavior of common sparsity agnostic procedures, and that of the minimax optimal likelihood ratio test, which requires knowledge of the sparsity $s$ and the signal strength $A$. Specifically, consider the regime $\alpha \in \big(\frac{1}{2}, 1\big)$, $A= \sqrt{\frac{2r \log p}{n}}$, with $r>\alpha$. In this setting, the asymptotics of the minimax risk is characterized by Theorem \ref{theorem:above_boundary_sparse}. In Theorem \ref{theorem:hc_vs_scan}, we study an idealized version of the Higher Criticism test, and establish that it is sub-optimal, in terms of maximum risk. We believe it would be intriguing to analyze other established procedures in this light, but leave this for future endeavors. 
    
    \item[(iii)] \textbf{Proof Techniques:}
    Our main results are technically challenging, and require several new ideas. To characterize the asymptotics of the minimax risk $\mathrm{Risk}(s,A)$, we usually derive upper and lower bounds separately. Note that an upper bound on the minimax risk can be derived by analyzing the behavior of specific test sequences. Deriving a matching lower bound is significantly more challenging. At the heart of our arguments is a philosophy originally espoused in \cite{ingster2012nonparametric}--- the behavior of the risk should be governed by a suitably truncated likelihood function. However, while the intuitive idea is extremely natural, implementing it in this setting turns out to be extremely challenging.
    The choice of the truncation event turns out to be extremely subtle, and dependent on the sparsity $s$ and signal strength $A$. Further, even with the choice of the truncation event, a sharp analysis of the truncated likelihood function still requires significant effort. We describe, in detail, the precise challenges, and how we overcome them, in the discussions following the main results in Section \ref{section:main_results}. We hope that these ideas can provide a starting point for analyzing the minimax risk of testing in other problems.

\end{itemize}

\subsection{Connections to Literature:}\label{section:connections_to_literature}
Our framework and subsequent analyses draw inspiration from the substantial Statistical literature on hypothesis testing. We discuss these connections in this section.

\begin{enumerate}\itemsep3pt
\item [(i)] \textbf{Connections to Chernoff Exponents:}
The second order behavior of the problem is analogous to the concept of Chernoff-exponents for testing a simple null versus a simple alternative. 
To this end, consider a setting where we observe $Z_1,\ldots,Z_n\stackrel{i.i.d}{\sim}F$, and wish to test $H_0:F=Q$ vs. $H_1:F=Q'$. We denote the $n$-fold product measures as $Q^{(n)}$ and $Q'^{(n)}$, and introduce the  likelihood ratios $L=\frac{dQ'}{dQ}$ and $L_n=\frac{dQ'^{(n)}}{dQ^{(n)}}$. The error behavior of the Neyman-Pearson test (while minimizing the sum of Type I and Type II error) can be characterized as  (see e.g. \cite[Corollary 12.1]{polyanskiy2015lecture})
\be 
\ & \lim_{n\rightarrow \infty}\frac{\log\left[\inf_T\left\{Q^{(n)}(T(Z_1,\ldots,Z_n)=1)+Q'^{(n)}(T(Z_1,\ldots,Z_n)=0)\right\}\right]}{n}\\
&=\lim_{n\rightarrow \infty}\frac{\log\left[Q^{(n)}(L_n>1)+Q'^{(n)}(L_n\leq 1)\right]}{n}=-\inf_{\lambda\in [0,1]}\psi(\lambda),\quad \psi(\lambda):=\E_{Q}\left(e^{\lambda\log{L}}\right).
\ee
In this context $\inf_{\lambda\in [0,1]}\psi(\lambda)$ is often referred to as the Chernoff-exponent of the testing problem; our goal is intimately connected to the computation of Chernoff-type-exponents for the testing problem \eqref{eqn:hypo}. However, Chernoff-exponents, as introduced above, are only relevant for simple versus simple testing with iid data --- consequently, the null and alternative distributions separate for large sample sizes, leading to an exponential decay in the error probabilities. 
On the other hand, if $Q$, $Q'$ are allowed to change with $n$, the alternatives do not necessarily separate with large sample sizes, and the behavior of the optimal testing error depends crucially on the precise separation of the two alternatives. In particular, if the two distributions are too close, non-trivial hypothesis testing is impossible, and the optimal testing errors will converge to one. This naturally relates to the two distinct regimes arising in the analysis of \eqref{eqn:hypo}, viz.,  the above and below boundary cases. We analyse these cases separately. 


However, one needs to overcome certain conceptual barriers before this analogy can be made precise. First and foremost, the problem under study is not a simple versus simple testing problem, and thus the notion of Chernoff exponents is not directly applicable. It is natural to guess that the minimax risk  \eqref{eq:minimax_risk} should be related to the Chernoff-exponent under a simple versus simple testing problem, obtained by putting an asymptotically least favorable prior on the alternative parameter space. Indeed, this connection arises naturally in our analysis, and the least favorable prior is also intuitive in this case --- under this prior, one selects exactly $s$ coordinates of $\bbeta$ at random and assigns a (possibly random signed) signal $A$ at the chosen locations. 
%
%
The likelihood ratio corresponding to any such prior $\pi$ equals
\be 
L_{\pi}=\int_{\bbeta}\exp\left(\langle \by,\bX\bbeta\rangle-\frac{1}{2}\|\bX\bbeta\|_2^2\right)d\pi(\bbeta), \label{eq:int_likelihood}
\ee
and the performance of any such test provides a lower bound to $\mathrm{Risk}(s,A)$. Under the iid setting described above, the log-likelihood ratio is an iid sum, and its performance may be analyzed using standard Large Deviation Theory for iid sums of random variables. On the contrary, these standard techniques, based on Cramer's method and exponential tilting, are no longer applicable under an integrated likelihood as in \eqref{eq:int_likelihood}. In contrast to the iid setting, the likelihood ratio is a U-statistics of order $s$ (see Section \ref{sec:notations_and_assumptions} for details) --- this makes the analysis of the likelihood ratio test extremely challenging. To facilitate this analysis, we employ a variety of different ideas (based on the below/above regime and the signal sparsity). We comment more on these after the statements of the individual results in Section \ref{section:main_results}.
%

\item [(ii)] \textbf{Connections to Bhattacharya Affinity and Chi-Square Divergence:}
There is a quantity related to the likelihood ratio \eqref{eq:int_likelihood} which is intimately tied to  $\mathrm{Risk}(\pi,s,A):=\P_{0}(L_{\pi}>1)+\E_{\bbeta\sim \pi}\left(\P_{\bbeta}(L_{\pi}\leq 1)\right)$  --- known as the Bhattacharya affinity in statistical literature \citep{bhattacharyya1946measure,devroye2013probabilistic,addario2010combinatorial}, given by $\pmb{\rho}_{\mathrm{BH},\pi}=\E_{0}\sqrt{L_{\pi}}$. In particular, one can show that \citep[Section 3]{addario2010combinatorial} $1-\sqrt{1-4\pmb{\rho}^2_{\mathrm{BH},\pi}}\leq \mathrm{Risk}(\pi,s,A)\leq  2\pmb{\rho}^2_{\mathrm{BH},\pi}$; therefore, the best rate of convergence of  $\mathrm{Risk}(\pi,s,A)$ can be understood by exact asymptotic behavior of $\pmb{\rho}_{\mathrm{BH},\pi}$. However,  $\pmb{\rho}_{\mathrm{BH},\pi}$ is analytically intractable, because of the square root involved. We note that a first order analysis of the problem does not require a fine understanding of $\pmb{\rho}_{\mathrm{BH},\pi}$ --- instead, it suffices to prove that $L_{\pi}\stackrel{\P_{0}}{\rightarrow} 1$ below the conjectured detection threshold. In turn, this is accomplished by bounding an appropriate $\chi^2$ distance, which requires the computation of a (possibly truncated) second moment of the likelihood ratio. This technique has been widely used in prior research, and has now attained considerable maturity.


While we do employ a variant of this idea to analyze the minimax risk below the detection threshold, a considerably finer asymptotic analysis is necessary in this setting. On the other hand, above the detection threshold, it is expected that the $\chi^2$ distance between the null and alternative should diverge, and hence it should be impossible to control the Bhattacharya affinity through an analysis of $\E_{0}(L_{\pi}^2)$ or its variants. 
%
Our proof techniques should be thought of as an indirect way of understanding the Bhattacharya affinity $\pmb{\rho}_{\mathrm{BH},\pi}$ between $\P_{0}$ and $\int \P_{\bbeta}d\pi(\bbeta)$ for an asymptotically least favorable prior $\pi$.

\item [(iii)] \textbf{Connections to Sparse Gaussian Mixture Models:}
At this point it is natural to ask if there is a related setting which sheds light on the problem at hand -- but at the same time involves a more tractable nature of the relevant likelihood ratio. Indeed, the sparse Gaussian mixture model \citep{jin2003detecting,cai2011optimal,cai2014optimal} is a problem that is natural to consider -- at least in the sense that it has similar first order behavior as testing \eqref{eqn:hypo} while having a more tractable likelihood ratio which is an i.i.d sum of random variables (and therefore amenable to careful implementation of Cramer Type analysis).


It turns out that the Chernoff-exponents for the sparse Gaussian mixture problem has been explored recently in \cite{ligo2015detecting,ligo2016rate}.    Formally, in this setup, one observes $\mathbf{y} \in \mathbb{R}^n$, and tests 
\be\label{eq:mixtures}
H_0: y_i \stackrel{\textrm{iid}}{\sim} \mathcal{N}(0,1), \quad
\textrm{v.s.}\quad
H_1: y_i \stackrel{\textrm{iid}}{\sim} (1 - \varepsilon_n) \,\mathcal{N}(0,1) + \varepsilon_n\, \mathcal{N}(\rho_n, 1). 
\ee
Similar to \eqref{eqn:hypo}, one seeks to characterize the minimal separation $\rho_n \geq 0$ required for the existence of asymptotically powerful tests. This is simply a Bayesian analogue of \eqref{eqn:hypo} (with a two point i.i.d. prior on each $\beta_j$ for $j=1,\ldots,p$), and for the purpose of determining the minimax separation rate, equivalent to the problem \eqref{eqn:hypo} introduced above. The seminal paper of \cite{donoho2004higher} derives complete results regarding the minimax separation rates for \eqref{eq:mixtures} (including sharp constants whenever possible) along with the development of a sparsity adaptive test, namely The Higher Criticism Test.  \cite{ligo2015detecting,ligo2016rate} initiated a study into the Type I and Type II errors in this setting, and provide partial answers in this specific case. As we shall see, although the first order behavior of the problem (as captured by the detection boundaries) are the same in this problem compared to the minimax setting we consider -- there are fundamental differences in analyzing the Chernoff-exponents. We devote Section \ref{section:veeravalli_ligo} to discuss this specific connection as well as crucial differences in more detail. 

\end{enumerate}

\subsection{Notations and Assumptions}\label{sec:notations_and_assumptions} 
Throughout, we let $[m] = \{1, \cdots, n\}$ for any $m\in \mathbb{N}$. Also for any $m\in \mathbb{N}$, $S \subset[m]$, and $\mathbf{w} \in \mathbb{R}^m$ we denote $w_{S} = \sum_{i \in S} w_i$ and $\mathrm{supp}(\mathbf{w})=\{i\in [m]: w_i\neq 0\}$. $I_m$ for every $m\in \mathbb{N}$ will stand for the identity matrix in dimension $m$. Throughout $\mathbf{1}(\cdot)$ will stand for the indicator function. Also for any matrix $A\in \mathbb{R}_{m_1\times m_2}$ and $\gamma\in \mathbb{R}$ we let $(A^TA)^{-\gamma}=\sum_{j=1}^{m_2}\delta_j^{-\gamma}\mathbf{1}(\delta_j> 0)v_jv_j^T$ where $v_1,\ldots,v_{m_2}$ are left singular vectors of $A$ and $\delta_j$'s the corresponding singular values. Finally, for a square matrix $A$, we use $\|A\|$ to denote its spectral norm; formally, $\|A\| = \sup_{\|x\|_2=1} \|Ax\|_2$.

The results in this paper are mostly asymptotic (in $n$) in nature and thus requires some standard asymptotic  notations.  If $a_n$ and $b_n$ are two sequences of real numbers then $a_n \gg b_n$ (and $a_n \ll b_n$) implies that ${a_n}/{b_n} \rightarrow \infty$ (and ${a_n}/{b_n} \rightarrow 0$) as $n \rightarrow \infty$, respectively. Similarly $a_n \gtrsim b_n$ (and $a_n \lesssim b_n$) implies that $\liminf_{n \rightarrow \infty} {{a_n}/{b_n}} = C$ for some $C \in (0,\infty]$ (and $\limsup_{n \rightarrow \infty} {{a_n}/{b_n}} =C$ for some $C \in [0,\infty)$). Alternatively, $a_n = o(b_n)$ will also imply $a_n \ll b_n$ and $a_n=O(b_n)$ will imply that $\limsup_{n \rightarrow \infty} \ a_n / b_n = C$ for some $C \in [0,\infty)$).

Throughout $\Bin(m,\theta)$ will stand for a generic binomial random variable with $m\in \mathbb{N}$ trials and success probability $\theta \in [0,1]$. $\Phi$ and $\phi$ denote the CDF and p.d.f respectively of a standard Gaussian. We also use $\bar{\Phi}(x) = 1 - \Phi(x)$. Moreover,  $\mathcal{N}_m(\mathbf{\mu},\Sigma)$ stands for a $m$-dimensional normal random distribution with mean vector $\mathbf{\mu}$ and non-negative definite covariance matrix $\Sigma$. Also, $\chi^2_m(\delta)$ will denote a generic non-central chi-square random variable with degrees of freedom $m\in \mathbb{N}$ and non-centrality $\delta>0$. For central chi-square variables ($\delta=0$), we simply use $\chi^2_m$.

\begin{defn}[Subgaussian random variables]
A random variable $X \in \mathbb{R}$ is subgaussian with parameter $\sigma$ if for all $t \in \mathbb{R}$,
\begin{align}
\log \mathbb{E}[\exp(tX)] \leq \frac{\sigma^2 t^2}{2} \nonumber
\end{align} 
for some $\sigma >0$. We define the subgaussian norm of the random variable $\|X\|_{\psi_2}$ as the smallest $\sigma>0$ such that the inequality holds for all $t \in \mathbb{R}$. 
\end{defn}

\begin{defn}[Subgaussian random vector]
A random vector $\mathbf{X} \in \mathbb{R}^p$ is defined to be subgaussian with parameter $\sigma$ if all its one dimensional projections are subgaussian with parameter $\sigma$, i.e., for all $u \in \mathbb{R}^p$, $\langle u, X \rangle$ is subgaussian with parameter $\sigma^2$. The subgaussian norm of the vector $\mathbf{X}$ is defined as 
\begin{align}
   \|\mathbf{X} \|_{\psi_2}= \sup_{\|u \|_2 = 1} \| \langle u, X \rangle \|_{\psi_2}. \nonumber
\end{align}
\end{defn}

\begin{enumerate}
\item[(A)] We say $\bX\in \rm{SubG}(\sigma^2)$ when the rows of $\bX$ are i.i.d. centered Isotropic subgaussian random vectors with subgaussian parameter $\sigma^2$.

\item [(B)] We say $\bX\in \mathcal{O}_n(p)$ when $\frac{1}{\sqrt{n}}\bX$ is an orthogonal matrix.

\end{enumerate}
Throughout the subsequent discussion, we will assume that the sample covariance matrix is positive definite almost surely. This follows under very weak conditions on the design distribution, e.g. whenever the rows are absolutely continuous with respect to Lebesgue measure on $\mathbb{R}^p$ (see e.g. \cite{eaton1973non}). Under these assumptions, let $\bX^{T}\bX = O D O^{T}$ denote a spectral decomposition of the sample covariance matrix, with $D=\mathrm{diag}(\lambda_1, \cdots, \lambda_p)$. As $\bX^{T}\bX$ is positive definite almost surely, $\lambda_i>0$ with probability one. We define 
\begin{align}
    (\bX^{T}\bX)^{-\frac{1}{2}} = OD^{-\frac{1}{2}}O^{T}, \nonumber 
\end{align}
where $D^{-\frac{1}{2}}= \mathrm{diag}(\lambda_1^{-\frac{1}{2}}, \cdots, \lambda_p^{-\frac{1}{2}})$. Armed with this definition, we introduce
\be
    \mathbf{z}= (\bX^{T}\bX)^{-\frac{1}{2}}\bX^{T}\mathbf{y}. \label{eq:z_defn}
\ee
This quantity will play an extremely crucial role in our subsequent analysis. In particular, using $\bz$ instead of the more natural quantity $(\bX^{T}\bX)^{-1}\bX^{T}\mathbf{y}$ allows substantial simplification of the proofs (and often less stringent mutual dependence of $(n,p)$) owing to the independence of the coordinates of $\bz$ given $\bX$.


\subsection{Tests} We derive upper bounds on $\mathrm{Risk}(s,A)$ by analyzing the performance of several concrete tests. For the convenience of the reader, we introduce some broad classes of tests that we use in the subsequent discussion. 

First, by a \textcolor{black}{chi-squared} type test, we will refer to any test which rejects $H_0$ for large values of $\sum_j z_j^2$, where $\mathbf{z}=(z_1, \cdots, z_p)$ is defined in \eqref{eq:z_defn}. Next, we use Max-test to denote any test that rejects the null for large values of $\max_{j} |z_j|$. We will also use the Scan test, which rejects for large values of $\max_{\{S \subset [p]: |S| = s\}} \sum_{j \in S} |z_j|$. Note that in this case, the scan test is trivial to compute, and is obtained as the sum of the top $s$ order statistics.  Finally, for any $\tau>0$, a Higher-Criticism test (abbreviated henceforth as HC-test) computes $HC(\tau) = \sum_{j =1}^{p} \mathbf{1}(|z_j| > \tau)$ and rejects the null for a large value of $HC(\tau)$. We note that this is simpler than the original $HC$ procedure, which scans over many thresholds $\tau$, and rejects the null for a large value of the maximum. However, we remind the reader that the underlying sparsity $s$ and the signal strength $A$ are known for us, allowing us to choose the optimal sparsity dependent threshold. This corresponds to the notion of an ideal-Higher Criticism statistic, discussed for example in \cite[Section 3.1]{cai2011optimal}, \cite{jin2003detecting}.

It is worth noting that all the tests are kind of standard -- but only through the nature of the underlying test statistics. It is however subtle to choose their rejection regions if one intends to go beyond minimax separation rates and explore exact asymptotic minimax behavior of the problem \eqref{eqn:hypo}. We will discuss the variants of these rejection regions after statement of the individual results.

\subsection{Outline } The rest of the paper is structured as follows. We state our main results formally in Section \ref{section:main_results}. An extremely related problem was recently studied in the context of the Gaussian mixture model in  \cite{ligo2015detecting,ligo2016rate}, and provided a major inspiration for our study --- we compare and contrast our contributions with these existing results in Section \ref{section:veeravalli_ligo}. Section \ref{section:discussion} discusses some unresolved issues in our analysis, and collects some possible directions for future research. In Section \ref{section:technical_lemmas}, we state and prove certain preliminary results which will be used subsequently in our proofs.  Finally, we prove our main results in Section \ref{section:proofs}, and defer some technical proofs to the Appendix.


\section{Main Results}\label{section:main_results}
We state our results formally in this section. We re-emphasize that for orthogonal designs $\bX$, our results are valid whenever $n\geq p$. For designs with iid sub-gaussian rows, the arguments are significantly more involved, and require explicit conditions which guarantee that $n$ grows significantly faster than $p$. In our statements below, we collect these relative growth conditions--- upon encountering such a condition, the reader should immediately ascribe these requirements to the sub-gaussian design setting. To convey our main message clearly, we have not tried to optimize these dependencies; further, for specific sub-gaussian matrices such as gaussian matrices, one can potentially derive stronger results.

We present our results in two subsections, based on whether $A$ is above or below the minimax separation boundary. Our results describe the precise behavior of either $\log{(1-\mathrm{Risk}(s,A))}$ ( below the  boundary) or $\log\mathrm{Risk}(s,A)$  (above the boundary); this characterizes the leading order behavior of the minimax risk, and inspires the title of minimax exponents for this article. However, most of the results can also be written without using the log-scale --- at the cost of long asymptotic expressions. A close inspection of the proofs show that most of the results, when written beyond the log-scale asymptotics presented here, are sharp up to multiplicative constants w.r.t. the rates that drive  $\log{(1-\mathrm{Risk}(s,A))}$  or $\log\mathrm{Risk}(s,A)$.


\subsection{Below Boundary Problem:}
Our first result characterizes the behavior of the minimax risk below the detection boundary in the dense signal regime ($\alpha \leq \frac{1}{2}$). 
\begin{theorem}\label{theorem:below_boundary_dense}
Let $\alpha\leq \frac{1}{2}$, $A=\sqrt{\frac{p^{\alpha-\frac{1}{2}-\delta}}{n}}$ with $\delta>0$. 
\begin{enumerate}
    \item Assume either $\bX\in \mathrm{Sub}(\sigma)$ with $n\gg p$ or $\bX\in \mathcal{O}_n(p)$ with $n \geq p$. If $\delta <\frac{1}{2}$,
\be 
\liminf_{p\rightarrow\infty}\frac{\log{(1-\mathrm{Risk}(s,A))}}{\log{p}}\geq -\delta.
\ee

\item  Assume either $\bX\in \mathrm{Sub}(\sigma)$ with $n \gg p^{1+2\delta}$ or $\bX\in \mathcal{O}_n(p)$ with $n \geq p$. Then we have
\be 
\liminf_{p\rightarrow\infty}\frac{\log{(1-\mathrm{Risk}(s,A))}}{\log{p}}\leq -\delta.
\ee
\end{enumerate} 
\end{theorem}
Note that this result characterizes the behavior of the minimax risk for sub-gaussian designs for $n\gg p^{1+2\delta}$ and $\delta<1/2$ --- thus the worst case dependence required for this result is $n \gg p^2$. 
To derive a lower bound on $(1- \mathrm{Risk}(s,A))$ (note that this corresponds to an upper bound on the risk), we consider a test which rejects the null whenever $\|\mathbf{z}\|_2^2>p$. We note that although the test statistic is chi-square type, the cut-off $p$ is non-standard. This is specifically chosen so that both Type I and Type II error are close to $1/2$ at a desired level. A direct analysis of this test yields the desired lower bound on $(1- \mathrm{Risk}(s,A))$. To derive the matching upper bound on $(1- \mathrm{Risk}(s,A))$ (note that this corresponds to an upper bound on the risk), we consider the classical lower bound approach i.e. proceed by noting that
$1-\mathrm{Risk(s,A)} \leq \frac{1}{2} \sqrt{\mathbb{E}_0(L_{\pi}^2)-1}$
where $L_{\pi}$ denotes the integrated likelihood ratio \eqref{eq:int_likelihood} for a suitable prior $\pi$. The result is established by directly analyzing the second moment $\mathbb{E}_0[L_{\pi}^2]$ -- a staple in literature. 

The next result studies the minimax risk below the detection threshold in the sparse signal regime ($\alpha> \frac{1}{2}$). 
 
\begin{theorem}\label{theorem:below_boundary_sparse}
Suppose $\alpha>\frac{1}{2}$, $A=\sqrt{\frac{2r\log{p}}{n}}$ with $0<r<\rho^*(\alpha)$, and $\bX\in \mathrm{Sub}(\sigma^2)$. Then
\be 
\lim_{p\rightarrow\infty}\frac{\log{(1-\mathrm{Risk}(s,A))}}{\log{p}}=\begin{cases}
r-\left(\alpha-\frac{1}{2}\right),\quad \text{if}\ 4r\leq 1, n\gg p^{\frac{13}{6}}\log p,\\
1-\alpha-(1-\sqrt{r})^2,\quad \text{if}\ 4r>1, n \gg p^{2}(\log p).
\end{cases}
\ee
The same result holds for $\bX\in \mathcal{O}_n(p)$ with any $n\geq p$. 
\end{theorem}

In this case, to derive a lower bound on $(1-\mathrm{Risk}(s,A))$, we consider two separate tests -- when $4r>1$, we consider a max-type test, which rejects the null whenever $\max_{j\in [p]} |z_j| > \sqrt{2\log p}$, where $\bz=(\bX^{T}\bX)^{-\frac{1}{2}} \bX^{T}\by$, as defined in \eqref{eq:z_defn}. Note that unlike the case of below boundary result in the dense signal case presented in Theorem \ref{theorem:below_boundary_dense}, this test does not have both Type I and Type II error close to $1/2$. Instead, the construction ensures Type I error close to $0$ but Type II error close to $1$ at a desired level.  On the other hand, for $4r<1$, the desired lower bound is attained upon analyzing a test which rejects the null whenever $\sum_{j}\mathbf{1}(|z_j|>2A)> 2p \bar{\Phi}(2A) + \tau_p \sqrt{2p \bar{\Phi}(2A) (1- 2\bar{\Phi}(2A))}$ for $\tau_p=O(\log\log{p})$. Note that this is the ideal Higher Criticism test described in Section \ref{sec:notations_and_assumptions}.  The upper bound on $(1-\mathrm{Risk}(s,A))$ uses a truncated second moment approach. More precisely, we consider the inequality $1-\mathrm{Risk(s,A)}\leq \frac{1}{2}\Big[\sqrt{\mathbb{E}_0(\tilde{L}_{\pi}-1)^2} + (1- \mathbb{E}_0[\tilde{L}_{\pi}]) \Big] $ where $\tilde{L}_{\pi}$ is some suitably truncated version of $L_{\pi}$. While several choices of truncation might work for analyzing first order behavior of the problem, we have to be more careful in the choice. To be more precise, consider $\pi$ described by choosing a set $S\subset [p]$ of size $s$ at random and thereafter assigning coordinates $\beta_j=A$ for $j\in S$. In this case, two choices of $\tilde{L_{\pi}}$ are given  $\frac{1}{{p \choose s}}\sum_{S:|S|=s}\exp\Big(\langle \mathbf{y}, \bX \bbeta_S \rangle - \frac{1}{2} \|\bX \bbeta_S\|_2^2 \Big)\mathbf{1}\left(\max_{j\in [p]}|z_j|\leq t_{p}\right)$ and $\frac{1}{{p \choose s}}\sum_{S:|S|=s}\exp\Big(\langle \mathbf{y}, \bX \bbeta_S \rangle - \frac{1}{2} \|\bX \bbeta_S\|_2^2 \Big)\mathbf{1}\left(\max_{j\in S}|z_j|\leq t_{p,s}\right)$ for suitable diverging sequences $t_p$ and $t_{p,s}$. Whereas, both of them yield the same results while considering first order behavior of the problem \citep{arias2011global,ingster2010detection}, for the second order behavior of the problem we can only work with second truncated likelihood ratio. Finally, even after identifying the appropriate truncating event, the subsequent analysis needs to be done with extreme care. For example, when $4r>1$, we crucially show that $\mathbb{E}_0[\tilde{L}_{\pi}-1]^2 \leq  C(1- \mathbb{E}_0[\tilde{L}_{\pi}])^2$ for some absolute constant $C>0$ and thereby demonstrating the desired rate through an exact analysis $(1- \mathbb{E}_0[\tilde{L}_{\pi}])^2$. The analysis for $4r\leq 1$ is even more subtle and the details can be found in Section \ref{section:proofs}. We note that this is contrast to the analysis of the first order behavior of the problem (even while finding sharp constants of detection boundary) since it only involves showing $\mathbb{E}_0[\tilde{L}_{\pi}-1]^2 =o(1)$ and $(1- \mathbb{E}_0[\tilde{L}_{\pi}])^2=o(1)$, and the exact rates of these terms do not matter.

\subsection{Above Boundary Problem:} Our first result for analyzing the second order behavior of the problem above the boundary (i.e. while understanding the fastest one can expect the total error of testing to go to $0$) is for the dense regime i.e. $\alpha\leq \frac{1}{2}$. 

\begin{theorem}\label{theorem:above_boundary_dense}
Let $\alpha\leq \frac{1}{2}$, $A=\sqrt{\frac{p^{\alpha-\frac{1}{2}+\delta}}{n}}$.
\begin{enumerate}[label=\textbf{\roman*}.]
\item
\label{theorem:above_boundary_dense_part1}
Assume either $\bX\in \mathrm{Sub}(\sigma)$ with $n\gg p^{7/5}\log p$ or $\bX\in \mathcal{O}_n(p)$ with $n\geq p$. If  $0<\delta<1/10$ is such that $\alpha-\frac{1}{2}+2\delta<0$, then
\be 
\lim_{p\rightarrow\infty}\frac{\log{\mathrm{Risk}(s,A)}}{p^{2\delta}}=-\frac{1}{16}.
\ee

\item \label{theorem:above_boundary_dense_part2}
Assume either $\bX\in \mathrm{Sub}(\sigma)$ with $n\gg p^2$ or $\bX\in \mathcal{O}_n(p)$ with $n\geq p$. If $0<\delta<\frac{1}{2}$ is such that 
$\alpha-\frac{1}{2}+\delta>0$. Then
\be 
\lim_{p\rightarrow\infty}\frac{\log{\mathrm{Risk}(s,A)}}{p^{\frac{1}{2}+\delta}}=-\frac{1}{8}.
\ee
\end{enumerate}

\end{theorem}
A few comments are in order regarding the proof techniques and phase transitions presented in Theorem \ref{theorem:above_boundary_dense}. In particular, unlike the analysis of the first order behavior of the problem \citep{arias2011global,ingster2010detection} (where analysis of the problem for $\alpha\leq \frac{1}{2}$ is almost trivial), the exact rate analysis is highly non-trivial and involves substantially new proof ideas while proving the lower bounds on the risk. First we note that the test that is optimal for proving an upper bound in Theorem \ref{theorem:above_boundary_dense}\ref{theorem:above_boundary_dense_part1} is the chi-squared test based on rejecting when $\|\mathbf{z}\|^2\geq p+\tau\sqrt{2p}$ with $\tau=nsA^2/2\sqrt{2p}$. In contrast, the test that is optimal for proving an upper bound in Theorem \ref{theorem:above_boundary_dense}\ref{theorem:above_boundary_dense_part2} is based on the scan type test that rejects when $\max_{S:|S|=s}|\mathbf{z}_S|>\sqrt{2(1+\tau)\log{{p\choose s}}}$ for a properly chosen $\tau>0$. The proof of the lower bound on $\mathrm{Risk}(s,A)$ is however extremely subtle and we need new ideas beyond common literature to proceed. In particular, note that we can no longer simply rely on the identity $1-\mathrm{Risk(s,A)} \leq \frac{1}{2} \sqrt{\mathbb{E}_0(L_{\pi}^2)-1}$ since $\mathbb{E}_0(L_{\pi}^2)$ diverges in the above boundary regime. Our proof instead relies on connecting the likelihood ratio $L_{\pi}$ (corresponding to a suitable prior) to the optimal tests described above. We explain the idea here for the first part i.e. Theorem \ref{theorem:above_boundary_dense}\ref{theorem:above_boundary_dense_part1}. Here we note that $\mathrm{Risk}(s,A) \geq \Pzero{[L_{\pi} >1]} + \E_{\bbeta \sim \pi }[\Pbeta{[ L_{\pi} \leq 1]}]\geq \E_{\bbeta \sim \pi }[\Pbeta{[ L_{\pi} \leq 1]}]$ and thereafter consider the lower bound $\Pbeta{[ L_{\pi} \leq 1]}\geq \P_{\bbeta}(\|\mathbf{z}\|^2\leq p+\tau\sqrt{2p})-\P_{\bbeta}[\|\mathbf{z}\|^2\leq p+\tau\sqrt{2p},L_{\pi}>1]$ where $\tau=nsA^2/2\sqrt{2p}$ corresponds to the idea cut-off of the chi-squared test described above. Although, for the upper bound on risk we needed an upper bound on $\P_{\bbeta}(\|\mathbf{z}\|^2\leq p+\tau\sqrt{2p})$, we first need to provide a matching lower bound on the quantity by carefully using Cram\'{e}r Type Moderate Deviation Lower Bound \cite[Theorem 2.13, Part (b)]{pena2008self} (this step required that $0<\delta<1/6$ comes in). Calling this bound $\psi_n$, we further use a change of measure argument to obtain $\E_{\bbeta\sim \pi}(\Pbeta{[ L_{\pi} \leq 1]})\geq \psi_n/2-\E_{0}(L_{\pi}^2\mathbf{1}(\|\mathbf{z}\|^2\leq p+\tau\sqrt{2p}))$. \textcolor{black}{The second term in the difference thereafter needs to be analyzed with extreme care (as a truncated second moment of the likelihood ratio) to show that this term asymptotically less than $\eta\psi_n$ for some fixed $0<\eta<1/2$. This eventually yields that asymptotically $\E_{\bbeta\sim \pi}(\Pbeta{[ L_{\pi} \leq 1]})$ is larger than $\psi_n(1/2-\eta)$ and thereby proving the desired result.} This proof technique however is extremely hard  to carry through for the proof of lower bound for Theorem \ref{theorem:above_boundary_dense}\ref{theorem:above_boundary_dense_part2} -- mainly because of the difficulty in lower bounding the probability upper tail of the scan statistics i.e. the event $\max_{S:|S|=s}|\mathbf{z}_S|>\sqrt{2(1+\tau)\log{{p\choose s}}}$.  We therefore need yet another proof technique, which is similar to the proof of our next theorem and therefore we explain the ideas later. Finally we note that, the phase transition that happens for $\alpha-\frac{1}{2}+\delta>0$ is not present while studying the first order behavior of the problem since as soon as $\delta>0$ all tests are asymptotically powerful. It is only while considering the exact rate of the best power function that this second phase transition appears. 

Our next result for analyzing the second order behavior of the problem above the boundary is for the sparse regime i.e. $\alpha> \frac{1}{2}$. 

\begin{theorem}\label{theorem:above_boundary_sparse}
Suppose $\alpha>\frac{1}{2}$ and $A=\sqrt{\frac{2r\log{p}}{n}}$ with $r>\alpha$. Assume either $\bX\in \mathrm{Sub}(\sigma)$ with $n \gg p^2$ or $\bX\in \mathcal{O}_n(p)$ with $n\geq p$.

\be 
\lim_{p\rightarrow\infty}\frac{\log{\mathrm{Risk}(s,A)}}{s\log{p}}=-\frac{(r-\alpha)^2}{4r}. 
\ee

\end{theorem}
We now discuss the proof techniques for Theorem \ref{theorem:above_boundary_sparse}. In this case, even the analysis of the first order behavior of the problem \citep{arias2011global,ingster2010detection} is relatively subtle. The the exact rate analysis also involves substantially new proof ideas while proving the lower bounds on the risk -- which are different from the proof of the lower bound in Theorem \ref{theorem:above_boundary_dense}\ref{theorem:above_boundary_dense_part1}. First we note that the optimal procedure  for proving an upper bound in Theorem \ref{theorem:above_boundary_sparse} is based on the scan type test that rejects when $\max_{S:|S|=s}|\mathbf{z}_S|>\sqrt{2(1+\tau)\log{{p\choose s}}}$ for properly chosen $\tau>0$. To an expert, this is not surprise since $r>\alpha$ coincides with the regime where strong recovery of the signals are possible in Hamming Loss \citep{butucea2018variable,ndaoud2018optimal,ji2012ups}. It is, however, in no way immediate to obtain the second order behavior of the global testing problem in this regime from the variable selection results. Indeed, since global testing is a information theoretically easier problem, one should not, in principle, be able to borrow ideas from variable selection base methods to explore optimal rates in a global testing problem. This is indeed the case here, and we need new ideas to prove the lower bound on the risk in Theorem \ref{theorem:above_boundary_sparse}. In particular, the proof technique for the similar lower bound presented in Theorem \ref{theorem:above_boundary_dense}\ref{theorem:above_boundary_dense_part1} is extremely hard to implement. This is because it requires a lower bound on $\P_{\bbeta}\left(\max_{S:|S|=s}|\mathbf{z}_S|\leq \sqrt{2(1+\tau)\log{{p\choose s}}}\right)$ which matches the upper bound analysis for the risk. Two main reasons that make this way of analysis hard are, (i) a standard Slepian type argument \citep{li2002normal} is extremely sub-optimal in this regard because of the particular dependence structure among the variables $\{\mathbf{z}_S,|S|=s\}$, and (ii) in general a diverging number of  candidate sets $S$ (beyond just $\mathrm{supp}(\bbeta)$) contributes to the exact asymptotic behavior of $\max_{S:|S|=s}|\mathbf{z}_S|$ whenever $s$ grows as polynomial in $p$. In order to bypass this issue, our proof relies on lower bounding the Type I error of the likelihood ratio test under a suitable prior $\pi$. In particular,  when $\pi$ is described by choosing a set $S\subset [p]$ of size $s$ at random and thereafter assigning coordinates $\beta_j=A$ for $j\in S$, one can appeal to the soft-max inequality (see Lemma \ref{lemma:soft_max}) to relate $L_{\pi}$ to the scan statistics above. Thereafter it remains to lower bound $\P_{0}\left(\max_{S:|S|=s}|\mathbf{z}_S|> \sqrt{2(1+\tau)\log{{p\choose s}}}\right)$ -- once again for which a Slepian Type argument is sub-optimal. We then simply note that the event $\max_{S:|S|=s}|\mathbf{z}_S|> \sqrt{2(1+\tau)\log{{p\choose s}}}$ is implied by $\sum_{j=1}^p\mathbf{1}\left(z_j>\sqrt{2(1+\tau)\log{{p\choose s}}}/s\right)\geq s$. However, under $H_0$ we have that $\sum_{j=1}^p\mathbf{1}\left(z_j>\sqrt{2(1+\tau)\log{{p\choose s}}}/s\right)\sim \mathrm{Bin}\left(p,\bar{\Phi}\left(\sqrt{2(1+\tau)\log{{p\choose s}}}/s\right)\right)$ and therefore we can analyze the probability of this event using Stirling approximations. 

We note that our earlier results, while stated for $s$ growing polynomially in $p$, continue to be valid (sometimes with simpler proofs) in case $s$ is sub-polynomial, i.e. $\log s = o(\log p)$. Recall that in this case, a sequence of asymptotically powerful tests exist if and only if $\limsup A/\sqrt{\frac{2\log p}{n}}\geq 1$. In the sub-polynomial regime, our results precisely characterize the behavior of the minimax risk in all regimes of $A$ ---  the behavior below the boundary is captured  by Theorem \ref{theorem:below_boundary_sparse}, while the behavior above the boundary is characterized by Theorem \ref{theorem:above_boundary_sparse}. Thus it only remains to understand of the minimax risk ``on" the detection boundary. Our next result characterizes this behavior in the special case where $s$ is fixed as $n,p \to \infty$. 

\begin{theorem}
\label{thm:boundary}
Suppose $s=O(1)$,  $A=\sqrt{\frac{2\log{p}}{n}}$. Assume either $\bX\in \mathrm{Sub}(\sigma)$ with $n \gg p (\log p)^2$ or $\bX\in \mathcal{O}_n(p)$ with $n\geq p$. Then
\begin{align}
    \mathrm{Risk}(s,A)\rightarrow \Big(\frac{1}{2} \Big)^s. \nonumber
\end{align}
\end{theorem}

Our result generalizes the corresponding result obtained by \cite[Theorem 8.1 Part 2, 3(a)]{ingster2012nonparametric} in the context of gaussian sequence models. As before, we define $\mathbf{z}=(\bX^{T}\bX)^{-1/2}\bX^{T}\mathbf{y}$, and consider the test which rejects whenever $\max |z_i| > \sqrt{2\log p}$. This provides an upper bound. The proof of the lower bound involves a delicate calculation based on an appropriately truncated likelihood ratio--- it is similar in spirit to that of \cite{ingster2012nonparametric}, but the details are substantially different due to the difference between the sequence model and the linear regression model. 

\subsection{On Possible Sub-optimality of the Higher Criticism Test}\label{section:hc_suboptimal}
One of the fundamental reasons behind the popularity of the class of results \citep{jin2003detecting,hall2010innovated,arias2011global,ingster2010detection,cai2011optimal,cai2014optimal} regarding first order behavior of the testing problem \eqref{eqn:hypo}, is the fact that it is possible to obtain first order optimal results which are agnostic to the sparsity level $s$ (or equivalently $\alpha$). Indeed, the fundamental insights from \cite{jin2003detecting,tukey1976t13} allows one to get the sharp optimality as well as adaptive first order results by using the Higher Criticism Test for $\alpha>\frac{1}{2}$. In contrast, the second order optimal test in Theorem \ref{theorem:above_boundary_sparse} (for $\alpha>1/2$ and $A=\sqrt{2\log{p}/n}$ with $r>\alpha$) is obtained through the scan test (i.e. based on the statistics $\max_{|S|=s}|\mathbf{z}_S|$) -- which crucially depends on the knowledge of $s$. 
This in turn raises the following natural question -- ``is it possible to obtain adaptive second order optimal results with sparsity agnostic methods -- and especially the Higher Criticism Test ?" Here we try to understand this question . 

First we recall \citep{jin2003detecting,cai2011optimal} why the Higher Criticism test based on $\{z_j\}_{j=1}^p$ is expected to succeed in a first order optimal sense without the knowledge of $s=p^{1-\alpha}$ for $\alpha>\frac{1}{2}$. To this end, note that the Higher Criticism test based on $\{z_j\}_{j=1}^p$ looks at the class of statistics $\left\{HC(t):=\frac{\sum_{j=1}^p\mathbf{1}(|z_j|>t)-2\bar{\Phi}(t)}{2p\bar{\Phi}(t)(1-2\bar{\Phi}(t))}\right\}_{t\geq 0}$. Indeed, one can indeed use each $HC(t)$ to perform a test of \eqref{eqn:hypo} -- the error of which is guided by the ratio $\sup_{\bbeta \in \Xi(s,A)}\frac{\E_{\bbeta}^2(HC(t))}{\mathrm{Var}_{\bbeta}(HC(t))}$. When the signal strength is scaled as $\sqrt{\frac{2r\log{p}}{n}}$, it turns out that \citep{jin2003detecting,cai2011optimal} this supremum is attained at $t_{\mathrm{op}}(r):=\sqrt{2c^*\log{p}}$ with $c^*:=\min \{4r,1\}$. Subsequently, the test based on $HC(t_{opt}(r))$ intuitively has the best power among all the tests based on individual $HC(t)$'s. Subsequently, given the knowledge of $r$, the test which has the same first order asymptotic behavior as the Higher Criticism Test (which rejects for large value of $\sup_{t\geq 0}HC(t)$) is the test that rejects for large values of $HC(t_{opt}(r))$.  When $\alpha>\frac{1}{2}$ and $r>\alpha$, we therefore can capture the same first asymptotic behavior of the Higher Criticism test by rejecting using large values of $HC(\sqrt{2\log{p}})$.  Our next theorem compares the test based on rejecting based on any large value of  $HC(\sqrt{2\log{p}})$ with the optimal Scan test from Theorem \ref{theorem:above_boundary_sparse}.

Mathematically, with $\mathbf{z}=(\bX^T\bX)^{-1/2}\bX^T\by$, we
consider the sequence of test statistics given by $T(\tau)=\sum_{j=1}^p\mathbf{1}(|z_j|>\tau)$ and the tests given by 
\be 
\xi(t,\tau)=\mathbf{1}(T(\tau)>t).
\ee
Based on the discussion above, we will work with the ideal Higher Criticism Test (with $\tau=\sqrt{2\log{p}}$) when $A=\sqrt{\frac{2r\log p}{n}}$ with $s=p^{1-\alpha}$ and $r>\alpha>\frac{1}{2}$. Our next result shows that no matter what the cut-off $t$ is, the ideal Higher Criticism test with cut-off $\tau$ is sub-optimal in view of Theorem \ref{theorem:above_boundary_sparse}. 
\begin{theorem}\label{theorem:hc_vs_scan}
Suppose $s=p^{1-\alpha}$ with $\alpha>\frac{1}{2}$ and $A=\sqrt{\frac{2r\log{p}}{n}}$ with $r>\alpha$. Assume either $\bX\in \mathrm{Sub}(\sigma)$ with $n \gg p^{3/2} $ or $\bX\in \mathcal{O}_n(p)$ with $n\geq p$. Then
\be
\lim_{n,p\rightarrow \infty}\frac{\inf_{t}\log\mathrm{Risk}(\xi(t,\sqrt{2\log{p}}),s,A)}{s\log{p}}>-\frac{(r-\alpha)^2}{4r}.
\ee
\end{theorem}
Our proof does not require $r > \alpha$; we include this condition in the statement of the Theorem as it establishes the sub-optimality of the ideal Higher Criticism procedure in this regime. 
The proof of Theorem \ref{theorem:hc_vs_scan} actually shows that  $\inf_{t}\mathrm{Risk}(\xi(t,\sqrt{2\log{p}}),s,A) \gg e^{-Cs\log{p}}$ for any $C>0$ -- and thereby showing acute sub-optimality of the test based on rejecting for large values of $HC(\sqrt{2\log{p}})$ compared to minimax optimal scan test for $r>\alpha$. Although this does not completely show that the HC test is sub-optimal corresponding to the Scan test, the arguments provide a strong evidence in this direction. 

Finally, we note that, for $\alpha>\frac{3}{4}$, the Max test also attains the first order  detection boundary \citep{jin2003detecting,arias2011global} and also functions without the knowledge of $s$. We now show that this test is also sub-optimal compared to the Scan Test. To formalize the statement of this result, we let 
\be 
\xi(t)=\mathbf{1}\left(\max_{j\in [p]}|z_j|>t\right).
\ee

\begin{theorem}\label{theorem:max_vs_scan}
Suppose $s=p^{1-\alpha}$ with $\alpha>\frac{1}{2}$ and $A=\sqrt{\frac{2r\log{p}}{n}}$ with $r>\alpha$. Assume either $\bX\in \mathrm{Sub}(\sigma)$ with $n \gg p^{3/2} $ or $\bX\in \mathcal{O}_n(p)$ with $n\geq p$. Then
\be
\lim_{n,p\rightarrow \infty}\frac{\inf_{t}\log\mathrm{Risk}(\xi(t),s,A)}{s\log{p}}>-\frac{(r-\alpha)^2}{4r}.
\ee
\end{theorem}

    
    

\section{Comparisons with \cite{ligo2015detecting,ligo2016rate}}\label{section:veeravalli_ligo}
As discussed in Section \ref{section:connections_to_literature}, we explore the minimax version of the sparse Gaussian mixtures problem. Indeed, the sparse Gaussian mixture version of the problem was recently analyzed in \cite{ligo2015detecting,ligo2016rate} -- and they serve a major inspiration for our formalization. However,  our results and analyses are not directly comparable to that of  \cite{ligo2015detecting,ligo2016rate} for several reasons. We list them below.

\begin{enumerate}
    \item First, while the models \eqref{eq:parameterspace} and \eqref{eq:mixtures} are equivalent for the purposes of the minimax separation rates (i.e. first order behavior), the testing risks do not immediately have any direct correspondence. This is especially reflected by contrasting our Theorem \ref{theorem:above_boundary_sparse} with \cite[Theorem 3.2, Corollary 3.2,and Corollary 3.8]{ligo2016rate,ligo2015detecting}. This is an important distinction since the differences imply that once cannot even always guess the rate of the risk function for problem \eqref{eq:parameterspace} in \eqref{eqn:gaussian_linreg_model} (even with orthogonal $\bX$ and unidirectional signal $\bbeta$) from \eqref{eq:mixtures}. 
    
    \item  Moreover, being a simple vs simple hypothesis testing problem, the worst risk of any sequence of tests for \eqref{eq:mixtures} can be understood by analyzing the log-likelihood ratio $\log(L)=\sum_{i=1}^n \int\left(1+\varepsilon_n\left[e^{\rho_n y_i-\rho_n^2/2}\right]\right)$ (which is a sum of i.i.d. random variables both under $H_0$ and $H_1$ described by \eqref{eq:mixtures}), and it is classical that the risk is minimized by a test which rejects whenever $\{ L >1 \}$. Consequently, from a technical perspective, in the setting of \eqref{eq:mixtures}, the log-likelihood ratio is a sum of iid variables, facilitating direct analysis using Large-deviation techniques for i.i.d sums. On the contrary, we do not have such special features in our setting, which leads to substantial difficulties, and necessitates fundamentally new ideas.
    
    \item 
    \cite{ligo2015detecting,ligo2016rate} analyze the Type I and Type II errors of the likelihood ratio test separately, and establish that above the detection boundary, these errors converge to zero at the same rate. In contrast, in our setting, an analysis of the minimax risk necessitates an analysis of the sum of Type I and Type II errors. From a technical standpoint, this is often helpful, e.g., to derive a rate optimal lower bound to the minimax risk, we can lower bound either error the Type I or the Type II error. However, the behavior of these terms below the detection boundary are often not symmetric --- specifically, the Type I and worst case Type II error of an asymptotically optimal test sequence often exhibit very different behavior.  
    In particular, the optimal tests might involve either (i) Type I error converging to $0$ and the Type II error converging to $1$ for $\alpha\in (1/2,1)$ (see proof of Theorem \ref{theorem:below_boundary_sparse}) or converging to $(1/2)^s$ for $s=O(1)$ (see proof of Theorem \ref{thm:boundary}); or (ii) both Type I and Type II error converging to $1/2$ (see proof of Theorem \ref{theorem:below_boundary_dense}).
    
    \item Moreover, \cite{ligo2015detecting,ligo2016rate} only considers the ``Above Boundary Problem" whereas we are interested in behavior of the risk function both below and above the minimax separation boundary. We remark more on specific differences following statement of various main results in Section \ref{section:main_results}.
    
    \item \cite{ligo2015detecting,ligo2016rate} only considers \eqref{eq:mixtures}, the likelihood ratio test for their purpose is computable in polynomial time (w.r.t. $n$). This is of course not the case for the problem \eqref{eqn:hypo} for the prior $\pi$ described in Section \ref{sec:notations_and_assumptions} which we show is asymptotically least favorable in our main results.
    
    \item  Finally \cite{ligo2015detecting} only considers the Gaussian sparse mixture model version of the problem whereas our results are in the regression setup with any isotropic subgaussian design matrix -- with the proofs offering verbatim extensions to the Gaussian sequence model and orthogonal design matrix regression case.   
\end{enumerate}

\section{Discussions}
\label{section:discussion} In this section we collect a few thoughts and comments on the results presented in this paper, the challenging gaps that remain, and other problems which we hope might be explored using some of the tools introduced here. 

\begin{enumerate}
    \item \textcolor{black}{\textbf{The Gaps:}} In this paper, we  rigorously explore 
    the second order behavior of a global testing problem against sparse alternatives in high dimensional linear regression. However,  the results in this paper do not characterize the minimax risk in two distinct parameter regimes (these correspond to the un-shaded regions of Figure \ref{fig:detection_results}) --- determining the behavior of $\mathrm{Risk}(s,A)$ in these regimes presents an intriguing mathematical challenge. 
    Specifically, these regions are
    (i) $\alpha\leq \frac{1}{2}$, $A=\sqrt{\frac{p^{\alpha-\frac{1}{2}+\delta}}{n}}$ with $\delta>1/10$ and $\alpha-\frac{1}{2}+2\delta<0$, or $\alpha- \frac{1}{2} + \delta<0< \alpha - \frac{1}{2} + 2\delta$, and (ii) $\alpha>\frac{1}{2}$, $A=\sqrt{\frac{2r\log{p}}{n}}$ and $\rho^*(\alpha)<r\leq \alpha$. In regard to the first regime (i), we believe that a requirement $\delta<\frac{1}{6}$ will be necessary for the postulated rate to hold (this guides the moderate deviation exponent of the chi-square statistic which constitutes the optimal test). Both the requirements $\delta<1/10$ and $\alpha-\frac{1}{2}+2\delta<0$ (instead of $\alpha-\frac{1}{2}+\delta<0$)  arise while performing some detailed asymptotic manipulations with the second moment of the truncated likelihood ratio w.r.t. the least favorable prior. In an unpublished note, we have managed to close this gap completely, once we restrict ourselves to a smaller class of alternatives, consisting of $s$-sparse but one-directional signals (i.e. over the parameter space $\Xi(s,A)\cap (\mathbb{R}_+)^d$), in the setting of orthogonal designs (i.e. $\bX\in \mathcal{O}_n(p)$). Our proof for this special case completely bypasses the truncated second moment type approach considered in this paper. However, this idea does not generalize to the case of bi-directional signals, and more general sub-Gaussian designs. The gap noted in the second regime (ii), however, is much more subtle. Based on some initial calculations we conjecture that there exists at least two more phase transitions for the minimax risk $\mathrm{Risk}(s,A)$ in the regime where $\alpha>\frac{1}{2}$, $A=\sqrt{\frac{2r\log{p}}{n}}$ and $\rho^*(\alpha)<r\leq \alpha$. In particular, we believe that a version of the ideal Higher Criticism Test based on $HC(\tau)$ (see Section \ref{section:hc_suboptimal} for details) should be optimal in this regime -- with the choice of $\tau$ and the resulting minimax risk being different based on whether $4r\leq 1$ or $4r>1$. Although an asymptotic analysis of these tests yield upper bounds on the minimax risk for these parameter regimes -- we have been unable to match it using an appropriate truncated second moment approach. We believe completely new ideas might be necessary to understand the behavior of the minimax risk in these regimes. In addition, understanding sharp dependencies on $n,p$ while maintaining the same rates derived here also remains another highly challenging avenue for future research.

    \item \textcolor{black}{\textbf{Block Signal Detection:}} The class of signals considered here allows arbitrary locations of the $s$ signals among the $p$ components of $\bbeta$. For the case of sparse mean detection problems with independent Gaussian errors (which corresponds to the special case of orthogonal designs), several papers have explored more structured sparsity patterns \citep{arias2005near,arias2011detection,tony2012robust,cai2014rate,chan2015optimal,sharpnack2016exact,arias2018distribution,datta2018optimal}. In particular, when there is a natural organization of the observations over a lattice (e.g. for noisy image data), it becomes relevant to test for signals which are contiguous or form natural shapes. While considering rectangular signals (of certain width and breadth), it is not hard to derive the second order behavior of the problem (extending from the first order behavior explored in \cite{arias2005near,arias2011detection,cai2014rate}) using our proof technique from Theorem \ref{theorem:above_boundary_dense} and \ref{theorem:below_boundary_sparse}. In particular, the lower bound follows upon considering disjoint rectangular signals, and this essentially corresponds to our problem with orthogonal design, known error variance (inversely proportional to the volume of the rectangles), and one signal (i.e. $s=1$).  The arguments can thereafter be seamlessly extended to the case of thick clusters with smooth boundaries as well (using the approximation technique by small sub-cubes presented in \cite{arias2011detection}). It remains to understand, however, how the second order behavior changes depending on a class of combinatorial signals as explored in \cite{addario2010combinatorial}.

   \item  \textcolor{black}{\textbf{Other Related Problems:}} Since the seminal papers of \cite{burnashev1979minimax,ingster1994minimax,ingster1995minimax,ingster1998minimax,ingster1997some,ingster2012nonparametric}, the paradigm of exploring minimax separation rates (i.e. the first order behavior of the problem) for testing problems in high dimensions under structured alternatives, has witnessed tremendous research activity -- see e.g. \cite{addario2010combinatorial,butucea2013detection,berthet2013optimal,arias2014community} and related references. 
   Exploring the exact minimax risk of testing (beyond the separation rates explored in these papers) will naturally constitute future directions, which can be potentially explored using the ideas presented in this paper. 
\end{enumerate}

\section{Technical Lemmas}\label{section:technical_lemmas} 
We collect some preliminary technical results in this section. 

Our first lemma collects an useful restricted isometry property of random matrices with i.i.d. sub-gaussian rows. 
\begin{lemma}[\cite{vershynin2010introduction}, Theorem 5.65]
\label{lemma:rip}
Let $\bX$ be an $n\times p$ matrix with independent i.i.d. sub-gaussian isotropic rows. There exists a constant $C>0$ such that with probability at least $1- 2\exp(-cn)$, 
\begin{align}
   \Big| \|\frac{1}{\sqrt{n}} \bX \bbeta \|_2^2 - \|\bbeta\|_2^2 \Big| \leq C \frac{s \log (p/s)}{n} \|\bbeta\|_2^2 \nonumber
\end{align}
for all $s$-sparse vectors $\bbeta$. The constants $C>0$ and $c>0$ are functions of the sub-gaussian norm of the rows of $A$. 
\end{lemma}
In the subsequent discussion, we will refer to the event introduced above as $\mathcal{G}_1$. Our next lemma introduces another typical event under the covariate distribution, which will be heavily used in the subsequent analysis. 
\begin{lemma}
\label{lem:good_event}
Let $\bX$ be an $n \times p$ matrix with independent isotropic sub-gaussian rows. Consider the event 
\begin{align}
    \mathcal{G}_2 = \left\{ \| \Big(\frac{\bX^T \bX}{n}\Big)^{\frac{1}{2}} - I\| \leq C \sqrt{\frac{p}{n}}\right\}, \nonumber 
\end{align}
Then there exists $C,c>0$, depending on the sub-gaussian norm of a row of the design matrix $\bX$, such that
\be
\mathbb{P}(\mathcal{G}_2) \geq 1-2\exp(-cp).
\ee
\end{lemma} 
\noindent 
Lemma \ref{lem:good_event} follows immediately from \cite[Theorem 5.39]{vershynin2010introduction}. We will  need the exponential moments of a folded normal distribution for the subsequent analysis. 
\begin{lemma}
\label{lemma:folded_normal_expression}
\begin{enumerate}
    \item Let $Z \sim \mathcal{N}(0,1)$ and $\lambda >0$. Then we have, 
    \begin{align}
        \E[\exp{(\lambda |Z|)} ] \leq 2 \exp{\Big( -  \frac{\lambda^2}{2}\Big) }. \nonumber
    \end{align}
    \item Let $Z \sim \mathcal{N}(\mu,1)$ and $\lambda >0$. Then we have, 
    \begin{align}
    \mathbb{E}[\exp{(-\lambda |Z|)}] =   \exp{ \Big(\frac{\lambda^2}{2} - \mu \lambda  \Big)  } \Big( 1- \Phi\Big(\lambda - \mu \Big) \Big) + \exp{\Big(\frac{ \lambda^2}{2} + \mu \lambda \Big)} \Big(1- \Phi\Big(\mu + \lambda\Big)\Big). \nonumber
    \end{align}
\end{enumerate}
\end{lemma}

We will use the following sharp concentration inequalities for the tails of chi-square random variables. 
\begin{lemma}[\cite{laurent2000adaptive}]
\label{lemma:chisq_conc}
Let $X \sim \chi^2_k$. Then we have, for all $x>0$, 
\begin{align}
    \mathbb{P}[X- k > 2 \sqrt{k x} + 2x] &\leq \exp(-x), \nonumber \\
     \mathbb{P}[k - X \geq 2\sqrt{k x} ] &\leq \exp(-x). \nonumber 
\end{align}
\end{lemma}
\begin{lemma}[\cite{birge2001alternative}]\label{lemma:noncchisq_conc}
Let $X \sim \chi^2_k(\nu)$ then for all $x>0$, 
\begin{align}
\P[ X > (k+\nu) + 2\sqrt{(k+ 2\nu) x} +  2x ] \leq \exp(-x). \nonumber \\
\P[ X < (k+ \nu) - 2\sqrt{(k+ 2\nu) x}  ] \leq \exp(-x). \nonumber 
\end{align}
\end{lemma}

\begin{lemma}
\label{lemma:hyp_bounds} 
Let $X \sim \mathrm{Hyp}(p,s,s)$. If $s= O(\sqrt{p})$, then there exists a universal constant $C>0$ such that 
\begin{align}
    \mathbb{P}[W=k] \leq C \frac{\Big(\frac{s^2}{p} \Big)^k}{k!}. \nonumber 
\end{align}
\end{lemma}
\begin{proof}
Observe that 
\begin{align}
    \mathbb{P}[W=k] &= \frac{{s \choose k} {p-s \choose s-k}}{{p\choose k}}= \frac{1}{k!} \Big( \frac{s!}{(s-k)!} \Big)^2 \frac{\{(p-s)!\}^2}{p! (p-2s +k)!} \nonumber\\
    &\leq \frac{1}{k!} \frac{s^{2k} p^{s-k} }{(p-s+1)^s}.  \nonumber 
\end{align}
The thesis follows whenever $s = O(\sqrt{p})$. 
\end{proof}


\section{Proofs}
\label{section:proofs}
In this section, we collect the proofs of the main results. We only proof the results for the isotropic sub-Gaussian designs. The proofs go through verbatim for exactly orthogonal designs for $n\geq p$.
\subsection{Proof of Theorem \ref{theorem:below_boundary_dense}}
We establish the upper and lower bounds on the minimax risk separately.

\textbf{Proof of Upper Bound:}
Consider a test $T$ that rejects when $\|(\bX^T\bX)^{-\frac{1}{2}}\bX^T\by\|_2^2>p$. Since $\by\sim \mathcal{N}(\bX\bbeta,I_p)$, we have that under $H_0$, $\|(\bX^T\bX)^{-\frac{1}{2}}\bX^T\by\|_2^2\sim \chi^2_p$. Therefore, the Type I error of $T$ equals 
\be 
\P_{0}\left(T=1\right)&=\P(\chi^2_p>p).
\ee
For the Type II error of $T$ note that $(\bX^T\bX)^{-\frac{1}{2}}\bX^T\by=(\bX^T\bX)^{\frac{1}{2}}\bbeta+\boldsymbol{\eta}$ where $\boldsymbol{\eta}=(\bX^T\bX)^{-\frac{1}{2}}\bX^T\beps\sim \mathcal{N}(\mathbf{0},I_p)$. Consequently,
\be
\|(\bX^T\bX)^{-\frac{1}{2}}\bX^T\by\|_2^2=\|\boldsymbol{\eta}\|_2^2+\|(\bX^T\bX)^{\frac{1}{2}}\bbeta\|_2^2+2\boldsymbol{\eta}^T(\bX^T\bX)^{\frac{1}{2}}\bbeta.
\ee
Therefore, the Type II error of $T$ under any $\bbeta \in \Xi(s,A)$ equals 
\be 
\P_{\bbeta}\left(T=0\right)&=\P_{\bbeta}\left(\|\boldsymbol{\eta}\|_2^2+\|(\bX^T\bX)^{\frac{1}{2}}\bbeta\|_2^2+2\boldsymbol{\eta}^T(\bX^T\bX)^{\frac{1}{2}}\bbeta\leq p\right)\\
&\leq \P_{\bbeta}\left(\|\boldsymbol{\eta}\|_2^2+\|(\bX^T\bX)^{\frac{1}{2}}\bbeta\|_2^2+2\boldsymbol{\eta}^T(\bX^T\bX)^{\frac{1}{2}}\bbeta\leq p,\mathcal{A}_{\kappa}\right)+\P_{\bbeta}(\mathcal{A}_{\kappa}^c),
\ee
where for any $\kappa>0$ we let $\mathcal{A}_{\kappa}=\left\{|\boldsymbol{\eta}^T(\bX^T\bX)^{\frac{1}{2}}\bbeta|\leq \kappa\|(\bX^T\bX)^{\frac{1}{2}}\bbeta\|_2, \|(\bX^T\bX)^{\frac{1}{2}}\bbeta\|_2^2\geq n\|\bbeta\|_2^2/4\right\}$. Since given $\bX$, $\bbeta^T(\bX^T\bX)^{1/2}\boldsymbol{\eta}\sim \mathcal{N}(0,\|(\bX^T\bX)^{\frac{1}{2}}\bbeta\|_2^2)$, we have
\be
\P_{\bbeta}\left(|\boldsymbol{\eta}^T(\bX^T\bX)^{\frac{1}{2}}\bbeta|> \kappa\|(\bX^T\bX)^{\frac{1}{2}}\bbeta\|_2\right)\leq 2e^{-\frac{\kappa^2}{2}}.
\ee
Moreover, by Lemma \ref{lemma:rip}, using $\frac{s\log(p/s)}{n}\to 0$, there exists $c,C>0$ (functions of subgaussian norms of the rows of $\bX$) such that for all sufficiently large $n$ (depending on $C>0$) 
\be 
\P_{\bbeta}\left(\|(\bX^T\bX)^{\frac{1}{2}}\bbeta\|_2^2< n\|\bbeta\|_2^2/4\right)\leq 2e^{-cn}.
\ee
Therefore, we have
\be 
\P_{\bbeta}(\mathcal{A}_{\kappa}^c)&\leq 2e^{-\frac{\kappa^2}{2}}+2e^{-cn}.
\ee
Moreover, 
\be 
\ & \P_{\bbeta}\left(\|\boldsymbol{\eta}\|_2^2+\|(\bX^T\bX)^{\frac{1}{2}}\bbeta\|_2^2+2\boldsymbol{\eta}^T(\bX^T\bX)^{\frac{1}{2}}\bbeta\leq p,\mathcal{A}_{\kappa}\right)\\
&\leq \P_{\bbeta}\left(\|\boldsymbol{\eta}\|_2^2\leq p-\frac{n\|\bbeta\|_2^2}{4}\left(1-2\frac{\kappa}{\sqrt{n}\|\bbeta\|_2/2}\right) \right)\\
&=\P\left(\chi^2_p\leq p\right)-\P\left(p-\frac{n\|\bbeta\|_2^2}{4}\left(1-\frac{4\kappa}{\sqrt{n}\|\bbeta\|_2}\right)< \chi^2_p\leq p\right)\\
\ee 
Therefore, for any $\kappa>0$ we have
\be 
\ & \P_{0}(T=1)+\P_{\bbeta}(T=0)\\ &\leq \P(\chi^2_p>p)+\P\left(\chi^2_p\leq p\right)-\P\left(p-\frac{n\|\bbeta\|_2^2}{4}\left(1-\frac{4\kappa}{\sqrt{n}\|\bbeta\|_2}\right)< \chi^2_p\leq p\right)+2e^{-\frac{\kappa^2}{2}}+2e^{-cn}\\
&=1-\P\left(p-\frac{n\|\bbeta\|_2^2}{4}\left(1-\frac{4\kappa}{\sqrt{n}\|\bbeta\|_2}\right)< \chi^2_p\leq p\right)+2e^{-\frac{\kappa^2}{2}}+2e^{-cn}.
\ee
Choose $\kappa=\sqrt{n}\|\bbeta\|_2/8$. This implies
\be 
\ & \P_{0}(T=1)+\P_{\bbeta}(T=0)\\ &\leq 
1-\P\left(p-\frac{n\|\bbeta\|_2^2}{8}< \chi^2_p\leq p\right)+2e^{-\frac{n\|\bbeta\|_2^2}{128}}+2e^{-cn}.
\ee 

Consequently,
\be 
\rm{Risk}(T,s,A)&\leq \sup_{\bbeta\in \Xi(s,A)}\left\{1-\P\left(p-\frac{n\|\bbeta\|_2^2}{8}< \chi^2_p\leq p\right)+2e^{-\frac{n\|\bbeta\|_2^2}{128}}+2e^{-cn}\right\}\\
&\leq 1-\P\left(p-\frac{nsA^2}{8}< \chi^2_p\leq p\right)+2e^{-\frac{nsA^2}{128}}+2e^{-cn}.
\ee
This also implies the same upper bound on $\rm{Risk}(s,A)$ and therefore
\be 
1-\rm{Risk}(s,A)&\geq \P\left(p-\frac{nsA^2}{8}< \chi^2_p\leq p\right)-2e^{-\frac{nsA^2}{128}}-2e^{-cn}.
\ee 
Now note that $p>\mathrm{Mode}(\chi^2_p)= p-2$ and for sufficiently large $n,p$, using $\delta<1/2$, we have $p-\frac{nsA^2}{8}<\mathrm{Mode}(\chi^2_p)= p-2$. Therefore, with $f_{ \chi^2_p}(x)$ denoting the density of $\chi^2_p$ random variable, we have
\be 
\P\left(p-\frac{nsA^2}{8}< \chi^2_p\leq p\right)\geq \frac{nsA^2}{8}\times\frac{f_{ \chi^2_p}\left(p-\frac{nsA^2}{8}\right)+f_{ \chi^2_p}\left(p\right)}{2}\geq \frac{nsA^2}{16}f_{ \chi^2_p}\left(p\right).
\ee
Now note that 
\be 
\frac{nsA^2}{16}f_{ \chi^2_p}\left(p\right)=\frac{(1+o(1))}{16}p^{\frac{1}{2}-\delta}\frac{e^{-\frac{p}{2}}p^{\frac{p}{2}-1}}{2^{p/2}\sqrt{2\pi}\left(\frac{p}{2}-1\right)^{\frac{p}{2}-1+\frac{1}{2}}e^{-\frac{p}{2}+1}}\geq \frac{1}{100}p^{-\delta}.
\ee
Therefore, 
\be 
1-\rm{Risk}(s,A)&\geq \frac{1}{100}p^{-\delta}-2e^{-\frac{p^{\frac{1}{2}-\delta}}{128}}-2e^{-cn}.
\ee 
Under the assumption that $\delta<\frac{1}{2}$, this implies 
\be 
\liminf_{p\rightarrow\infty}\frac{\log{(1-\mathrm{Risk}(T,s,A))}}{\log{p}}\geq-\delta,
\ee
as required.

\textbf{Proof of Lower Bound:}
%
In this case, we wish to prove that 
\begin{align}
    \limsup_{p \to \infty} \frac{\log (1- \mathrm{Risk}(s,A))}{\log p} \leq - \delta. \nonumber 
\end{align}
Recall the parameter space $\tilde\Xi(s,A)$ \eqref{eq:parameterspace_boundary}; we start with the likelihood ratio 
\begin{align}
    L(\bbeta) = \exp{\Big( \langle \by, \bX\bbeta \rangle - \frac{1}{2} \| \bX \bbeta \|_2^2 \Big) },  \nonumber 
\end{align}
and consider the uniform prior $\pi$ on $\tilde\Xi(s,A)$ which chooses $s$ locations uniformly at random and sets these coordinates to be $\pm A$ at random. All other coordinates are set at zero. We define the integrated likelihood ratio $L_{\pi} = \E_{\bbeta\sim \pi} [L(\bbeta)]$, and  note the classical lower bound
\be 
1-\mathrm{Risk}(s,A)\leq \frac{1}{2}\sqrt{\E_{0}(L_{\pi}^2)-1}.
\ee
Now 
\be 
\E_{0}(L_{\pi}^2)=\int_{\bbeta,\bbeta'}\E\left[\exp\left(\langle\bX\bbeta,\bX\bbeta'\rangle\right)\right]d\pi(\bbeta)d\pi(\bbeta)
\ee
To proceed first note that for $\bbeta\in \mathrm{supp}(\pi)$ we have $\|\bbeta\|=\sqrt{p^{1/2-\delta}/n}=un^{-1/4}$ where $u=\sqrt{p^{1/2-\delta}/\sqrt{n}}$. Further, $p/n\rightarrow 0$ implies $u\rightarrow 0$ and so $u\in (0,u_0)$ for any fixed $u_0>0$. Therefore using \cite[Lemma 11]{carpentier2018minimax} we have
\be 
\E\left[\exp\left(\langle\bX\bbeta,\bX\bbeta'\rangle\right)\right]&\leq \exp\left(n\langle\bbeta,\bbeta'\rangle\right)(1+C_0u^2)
\ee
for some $C_0>0$ only depending on the subgaussian constant $\sigma^2$. Subsequently,
\be 
\E_{0}(L_{\pi}^2)&\leq (1+C_0u^2)\int_{\bbeta,\bbeta'}\exp\left(n\langle\bbeta,\bbeta'\rangle\right)d\pi(\bbeta)d(\bbeta')
\ee
Now, using \cite[Theorem 4]{hoeffding1994probability}, we have for for any $\epsilon>0$
 and sufficiently large $n,p$ 
\be 
\E_{\bbeta,\bbeta'\sim \pi}\left[e^{n\langle\bbeta,\bbeta'\rangle}\right]&\leq \exp\left(\frac{s^2}{p}(\cosh(nA^2)-1)\right)\leq \exp\left(\frac{s^2}{p}\frac{n^2A^4}{2}(1+\epsilon)\right)\leq 1+\frac{p^{-2\delta}}{2}(1+\epsilon)^2.
\ee
Therefore for any $\epsilon>0$ one has for sufficiently large $n,p$  that
\be 
\E_{0}(L_{\pi}^2)-1\leq 2(p^{-2\delta}(1+\epsilon)^2+u^2) 
\ee
Now $u^2\ll p^{-2\delta}$ whenever $p^{1+2\delta}\ll n$ and we get the desired result in that case.

\subsection{Proof of Theorem \ref{theorem:below_boundary_sparse}}
We start with a proof of the upper bound.

\textbf{Proof of Upper Bound:} 
First, we derive an upper bound to the minimax risk via the analysis of specific tests. Consider first the case $4r>1$, and recall $\mathbf{z}$ \eqref{eq:z_defn}. 
%
Consider the test sequence $T_{1p}$ which rejects $\mathbf{H}_0$ whenever 
$\{\max |z_i| > \sqrt{2 \log p}\}$. 
Note that under $\mathbf{H}_0$, $\mathbf{z}\sim \mathcal{N}(0, I)$, and thus we have, 
\begin{align}
    \P_0\Big[\max_{1\leq i \leq p} |z_i| > \sqrt{2\log p}\Big] = 1 - (2 {\Phi}(\sqrt{2\log p})-1)^p. \nonumber 
\end{align}

Next, we control the Type II error of this procedure. To this end, recall the event $\mathcal{G}_2$ introduced in Lemma \ref{lem:good_event}.
Note that under $\Pbeta{}$, $\mathbf{z} = (\bX^{T} \bX)^{-1/2} \bX^{T}\mathbf{y} = (\bX^{T} \bX)^{1/2}\bbeta + \mathbf{\eta}$, where $\eta|\bX \sim \mathcal{N}(0, I)$. Thus we have, using Lemma \ref{lem:good_event}, 
\begin{align}
    \Pbeta{\Big[ \max_{1\leq i \leq p} |z_i| \leq \sqrt{2\log p}\Big]} &\leq \Pbeta{\Big[ \max_{1\leq i \leq p} |z_i| \leq  \sqrt{2\log p}, \mathcal{G}_2\Big]} + 2 \exp{(-c p)} \nonumber \\
    &= \E_{\bbeta}\Big[ \mathbf{1}_{\mathcal{G}_2} \Pbeta{\Big[ \max_{i \in S(\bbeta)} |z_i | \leq  \sqrt{2\log p} , \max_{i \notin S(\bbeta)} |z_i| \leq  \sqrt{2\log p} | \bX\Big]} \Big]+ 2 \exp{(-c p)}. \nonumber
    \end{align}
Observe that given $\bX$, $\{z_i : i \notin S(\bbeta)\}$ are independent Gaussian random variables. In this case, $|z_i | \succeq |\eta_i|$, where $\{\eta_i : i \notin S(\bbeta)\}$ are independent $\mathcal{N}(0,1)$ random variables. Thus we have, 
\begin{align}
     \Pbeta{\Big[ \max_{1\leq i \leq p} |z_i| \leq \sqrt{2\log p}\Big]} \leq (2 \Phi(\sqrt{2\log p}) -1)^{p-s} \E_{\bbeta}\Big[ \mathbf{1}_{\mathcal{G}_2} \Pbeta{\Big[ \max_{i \in S(\bbeta)} |z_i | \leq  \sqrt{2\log p}  | \bX\Big]} \Big]+ 2 \exp{(-c p)}. \nonumber 
\end{align}
To control the probability corresponding to $i \in S(\bbeta)$, note that we have, 
\begin{align}
    z_i = \sqrt{n} \beta_i +  e_i^{T}\Big( \Big( \frac{\bX^{T} \bX}{n}  \Big)^{\frac{1}{2}}-I \Big)\sqrt{n} \bbeta  + \eta_i. \nonumber 
\end{align}
Further, on the good event $\mathcal{G}_2$, we have, 
\begin{align}
    | e_i^{T}\Big( \Big( \frac{\bX^{T} \bX}{n}  \Big)^{\frac{1}{2}}-I \Big)\sqrt{n} \bbeta | \leq \sqrt{n}\| \Big(\frac{\bX^T \bX}{n}\Big)^{\frac{1}{2}} - I\|  \|\bbeta \|_2\lesssim \sqrt{\frac{sp \log p}{n} }, \nonumber 
\end{align}
\textcolor{black}{provided $\max_{j}|\beta_j|\leq \sqrt{\frac{C^*\log{p}}{n}}$ for any constant $C^*>0$.}
Thus we have, 
\begin{align}
    &\Pbeta{\Big[ \max_{1\leq i \leq p} |z_i|\leq \sqrt{2\log p}\Big]} \nonumber \\ 
    &\leq  (2\Phi(\sqrt{2\log p}) -1)^{p-s} \Big(\mathbb{P}\Big[|\sqrt{2r\log p} - O\Big(\sqrt{\frac{sp\log p}{n}}\Big) + \eta |\leq \sqrt{2\log p}\Big]\Big)^s + 2 \exp(-cp). \nonumber
\end{align}
Hence we have the following bound on the risk of this sequence of tests (for a constant $C^*$ to be decided later). 
\be
    \ & \P_{0}(T_{1p}=1)+\sup_{\bbeta\in \Xi(s,A): \|\bbeta\|_{\infty}\leq \sqrt{\frac{C^*\log{p}}{n}}}\P_{\bbeta}(T_{1p}=0)\\
    &\leq  1- (2\Phi(\sqrt{2\log p})-1)^p \nonumber \\
    &+ (\Phi(\sqrt{2\log p}) -1)^{p-s} \Big(\mathbb{P}\Big[|\sqrt{2r\log p} - O\Big(\sqrt{\frac{sp\log p}{n}}\Big) + \eta |\leq \sqrt{2\log p}\Big]\Big)^s + 2 \exp(-cp).
\ee
Usual Mills ratio bounds imply 
\begin{align}
(2\Phi(\sqrt{2 \log p})-1)^{p-s} = 1 + o(1), \,\,\,\,
(2\Phi(\sqrt{2 \log p})-1)^s  = 1 - p^{- \alpha + o(1)},\nonumber  \\
\Big(\mathbb{P}\Big[|\sqrt{2r\log p} - O\Big(\sqrt{\frac{sp\log p}{n}}\Big) + \eta |\leq \sqrt{2\log p}\Big]\Big)^s  = 1 - p^{ 1 - \alpha - (1 - \sqrt{r})^2 + o(1)}, \nonumber
\end{align}
\textcolor{black}{provided $sp\ll n$ which is attained when $p^{3/2}\ll n$ since $\alpha>1/2$.}
Thus we obtain the following bounds on the minimax risk.
\be
  \ & 1-\left(\P_{0}(T_{1p}=1)+\sup_{\bbeta\in \Xi(s,A): \|\bbeta\|_{\infty}\leq \sqrt{\frac{C^*\log{p}}{n}}}\P_{\bbeta}(T_{1p}=0)\right)\\ &\geq (2\Phi(\sqrt{2\log p})-1)^{p-s}\Big( (2\Phi(\sqrt{2\log p})-1)^s - \Big(\mathbb{P}\Big[|\sqrt{2r\log p} - O\Big(\sqrt{\frac{sp\log p}{n}}\Big) + \eta |\leq \sqrt{2\log p}\Big]\Big)^s \Big)\nonumber \\
    &- 2\exp(-cp). \label{eqn:sparse_below_boundary_small_beta_maxtest}
\ee
On the other hand if $\|\bbeta\|_{\infty}>\sqrt{\frac{C^*\log{p}}{n}}$ for some large enough $C^*$ then the signal becomes easily detectable by a Max type test which has the desired level of convergence rate in terms of its risk. We however need to perform a Max type test based on the coordinates of $(\bX^T\bX)^{-1}\bX^T\by$ instead of $(\bX^T\bX)^{-1/2}\bX^T\by$ to obtain a better dependence of $p$ on $n$. The result is collected in our next lemma.  
\begin{lemma}\label{lemma:sparse_below_boundary_bonferroni}
Assume $p/n\rightarrow 0$. Then for any $C>0$ there exists $C^*>0,c>0$ and a sequence of tests $T_{3p}$ such that
\be 
\P_{0}(T_{3p}=1)+\sup_{\bbeta\in \Xi(s,A): \|\bbeta\|_{\infty}>\sqrt{\frac{C^*\log{p}}{n}}}\P_{\bbeta}(T_{3p}=0)&\leq 3e^{-C\log{p}}.
\ee
\end{lemma}

We now combine \eqref{eqn:sparse_below_boundary_small_beta_maxtest} and Lemma \ref{lemma:sparse_below_boundary_bonferroni} with $C>1-\alpha - (1-\sqrt{r})^2$ to get that for $4r>1$ (by considering a Bonferroni correction between $T_{1p}$ and $T_{3p}$ i.e. considering the test $T_p=\max\{T_{1p},T_{3p}\}$) 
\begin{align}
    \liminf_{p\to \infty} \frac{\log(1- \mathrm{Risk}(s,A))}{\log p}\geq 1-\alpha - (1-\sqrt{r})^2. \nonumber
\end{align}
Next, we derive the upper bound in the setting $4r\leq 1$. Consider a sequence of tests $T_{2p}(\tau_p)$ which rejects the null whenever $\sum_{i=1}^{p} \mathbf{1}(|z_i| > 2A) > 2p \bar{\Phi}(2A) + \tau_p \sqrt{2p \bar{\Phi}(2A) (1- 2\bar{\Phi}(2A))}$, for some sequence $\tau_p$ to be chosen appropriately. The following lemma posits the existence of an appropriate sequence $\tau_p$ such that the test sequence $T_{2p}(\tau_p)$ attains the optimal minimax risk. 
\begin{lemma}\label{lemma:sparsepoly_below_hc}
Assume that $A\geq \sqrt{\frac{2r\log p}{n}}$. Then there exists a sequence $\tau_p \to \infty$ such that for all $0< r < \rho^*(\alpha)$ with $ 4 r \leq 1$ and all constant $C^*>0$, one has
\be
\lim_{p \to \infty} \frac{\log \left(1 -\left(\P_{0}(T_{2p}(\tau_p)=1)+\sup_{\bbeta\in \Xi(s,A): \|\bbeta\|_{\infty}\leq \sqrt{\frac{C^*\log{p}}{n}}}\P_{\bbeta}(T_{2p}(\tau_p)=0)\right)\right)}{\log p} \geq r - \Big( \alpha  - \frac{1}{2} \Big). \nonumber
\ee
provided $p^{3/2}\ll n$.
\end{lemma}
\noindent
The proof of Lemma \ref{lemma:sparsepoly_below_hc}, while conceptually straightforward, is computationally involved, and thus is deferred to the Appendix. 
The requisite lower bound now follows by appealing to Lemma \ref{lemma:sparsepoly_below_hc} and Lemma \ref{lemma:sparse_below_boundary_bonferroni} with $C>r - \Big( \alpha - \frac{1}{2} \Big)$---  upon considering a Bonferroni correction between $T_{2p}(\tau_p)$ and $T_{3p}$,  we obtain the lower bound for $4r \leq 1$ as 
\be
\liminf_{p \to \infty} \frac{\log (1 - \mathrm{Risk}(s,A) ) }{\log p} \geq r - \Big( \alpha - \frac{1}{2} \Big). \nonumber
\ee
This completes the proof of the upper bound to the minimax risk. 

\textbf{Proof of Lower Bound:} 
To lower bound the minimax risk, we consider a prior $\pi$ which selects $s= p^{1-\alpha}$ locations at random, and sets $\beta_i=  A$ at all the selected locations. For each $S \subset [p]$ with $|S|=s$, we define the $\bbeta$ constructed as above as $\bbeta_S$. We define 
\begin{align}
    L_{\pi} = \frac{1}{{p \choose s}} \sum_{|S|=s} \exp\Big(\langle \mathbf{y}, \bX \bbeta_S \rangle - \frac{1}{2} \|\bX \bbeta_S\|_2^2 \Big). \nonumber 
\end{align}
This implies the following lower bound on the minimax risk
\begin{align}
    \mathrm{Risk}(s,A) \geq 1 - \frac{1}{2} \mathbb{E}_0[| L_{\pi} - 1|]. \nonumber 
\end{align}
Recalling $\mathbf{z}$ \eqref{eq:z_defn}, we define
\begin{align}
    \tilde{L}_{\pi} = \frac{1}{{p \choose s}} \sum_{|S| =s} \exp\Big(\langle \mathbf{y}, \bX \bbeta_S \rangle - \frac{1}{2} \|\bX \bbeta_S\|_2^2 \Big) \mathbf{1}(\max_{i \in S} z_i < \sqrt{2\log p}). \nonumber  
\end{align}
An application of triangle inequality, coupled with the observation that $\mathbb{E}_0(\tilde{L}_{\pi})\leq 1$, immediately implies that 
\be
    1- \mathrm{Risk}(s,A) \leq \frac{1}{2} \mathbb{E}_0[|L_{\pi}-1|] \leq \frac{1}{2}\Big[\mathbb{E}_0[|\tilde{L}_{\pi}-1|] + (1- \mathbb{E}_0[\tilde{L}_{\pi}]) \Big]. \label{eq:below_int1}
\ee
We will analyze each term in turn. To this end, recall the good event $\mathcal{G}_2$ introduced in Lemma \ref{lem:good_event}. 
First, note that we have, 
\begin{align}
    1- \mathbb{E}_0[\tilde{L}_{\pi}] \leq 1- \mathbb{E}_0[\tilde{L}_{\pi} \mathbf{1}_{\mathcal{G}_2}]. \nonumber 
\end{align}
Now, observe that 
\begin{align}
    \mathbb{E}_0[\tilde{L}_{\pi} \mathbf{1}_{\mathcal{G}_2}] = \frac{1}{{p \choose s}} \sum_{|S|=s} \mathbb{P}_{\bbeta_{S}} \Big[ \max_{i \in S} z_i < \sqrt{2\log p}, \mathcal{G}_2 \Big]. \nonumber 
\end{align}
Under $\mathbb{P}_{\bbeta_S}$, $\mathbf{z} | \bX \sim \mathcal{N}\Big( (\bX^T \bX/n)^{1/2} \sqrt{n} \bbeta    , I\Big)$. On the event $\mathcal{G}_2$, for $i \in S$,  $\mathbb{E}_{\bbeta_S}[z_i |X] \leq \sqrt{2r \log p}+ C' \sqrt{\frac{sp\log p}{n}}$, with $C'= C \sqrt{2r}$. Therefore, we have, 
\be
    \mathbb{E}_0[ \tilde{L}_{\pi} \mathbf{1}_{\mathcal{G}_2}] \geq  \Phi\Big(\sqrt{2\log p} - \sqrt{2r \log p} - C' \sqrt{\frac{p^{3/2}\log p}{n}} \Big)^s \mathbb{P}(\mathcal{G}_2). \label{eq:below_int_exp}
\ee
Combining the above with the bound on $\mathbb{P}(\mathcal{G}_2^c)$, we obtain an upper bound the second term in \eqref{eq:below_int1}. we have
\be
    1 - \mathbb{E}_0[\tilde{L}_{\pi}] \leq p^{1-\alpha - (1-\sqrt{r})^2 + o(1)}.  \label{eq:below_int2} 
\ee
Next, we turn to the first term in \eqref{eq:below_int1}. We have, 
\begin{align}
    \mathbb{E}_0[|\tilde{L}_{\pi} - 1|] \leq \mathbb{E}_0[|\tilde{L}_{\pi} -1| \mathbf{1}_{\mathcal{G}_2}] + \mathbb{E}_0[(\tilde{L}_{\pi} +1) \mathbf{1}_{\mathcal{G}_2^c}]. \nonumber
\end{align}
We note $\mathbb{E}_0[(\tilde{L}_{\pi} +1) \mathbf{1}_{\mathcal{G}_2^c}]\leq \mathbb{E}_0[\mathbf{1}_{\mathcal{G}_2^c}\mathbb{E}_0[(L_{\pi}+1)| \bX]] \leq 4\exp(-cp)$.  Finally, by an application of Cauchy-Schwarz inequality, we have, 
\be
    \mathbb{E}_0[|\tilde{L}_{\pi} -1| \mathbf{1}_{\mathcal{G}_2}] \leq \sqrt{\mathbb{E}_0\Big[ (\tilde{L}_{\pi} -1)^2 \mathbf{1}_{\mathcal{G}_2} \Big]} = \sqrt{\mathbb{E}_0[\tilde{L}_{\pi}^2\mathbf{1}_{\mathcal{G}_2}] - 2\mathbb{E}_0[\tilde{L}_{\pi}\mathbf{1}_{\mathcal{G}_2}] + \mathbb{P}(\mathcal{G}_2)}. \label{eq:lower_chebychev}
\ee
We note that 
\begin{align}
    \mathbb{E}_0[\tilde{L}_{\pi}^2 \mathbf{1}_{\mathcal{G}_2}] = \frac{1}{{p \choose s}^2} \sum_{|S| = |S'| =s } \mathbb{E}_0\Big[ \exp\Big( \langle \mathbf{y}, \bX (\bbeta_S + \bbeta_{S'}) \rangle - \frac{1}{2}(\| \bX \bbeta_S\|^2 + \|\bX \bbeta_{S'}\|^2 ) \Big) \mathbf{1}(\max_{i \in S\cup S'} z_i < \sqrt{2\log p} , \mathcal{G}_2) \Big].\nonumber 
\end{align}
For any fixed $S,S' \subset [p]$ with $|S|=|S'|=s$, 
\begin{align}
    &\mathbb{E}_0\Big[ \exp\Big( \langle \mathbf{y}, \bX (\bbeta_S + \bbeta_{S'}) \rangle - \frac{1}{2}(\| \bX \bbeta_S\|^2 + \|\bX \bbeta_{S'}\|^2 ) \Big) \mathbf{1}(\max_{i \in S\cup S'} z_i < \sqrt{2\log p} , \mathcal{G}_2) \Big] \nonumber \\
    &\leq \mathbb{E}_0\Big[ \exp(\langle \bX \bbeta_S, \bX \bbeta_{S'} \rangle) \mathbf{1}_{\mathcal{G}_2} \mathbb{P}_{\bbeta_S + \bbeta_{S'}}\Big( \max_{i \in S \cup S'} z_i < \sqrt{2\log p} | \bX \Big) \Big]. \nonumber 
\end{align}
On the event $\mathcal{G}_2$, we have, 
\begin{align}
    &\mathbb{P}_{\bbeta_S + \bbeta_{S'}} \Big( \max_{i \in S \cup S'} z_i < \sqrt{2\log p} | \bX \Big)  \nonumber \\
    &\leq \prod_{i \in S \cup S'} \Phi\Big( \sqrt{2\log p} - \sqrt{n}(\bbeta_{S,i} + \bbeta_{S', i}) + C'\sqrt{\frac{p^{3/2} \log p}{n}} \Big) \mathbb{E}_0\Big[\exp(\langle \bX \bbeta_S, \bX \bbeta_{S'} \rangle )\Big] \nonumber \\
    &\leq  \Phi\Big(\sqrt{2 \log p} - \sqrt{2r \log p}  + C'\sqrt{\frac{p^{3/2} \log p}{n}} \Big)^{2(s - |S \cap S'|)} \Phi\Big(\sqrt{2 \log p} - 2\sqrt{2r \log p}  + C'\sqrt{\frac{p^{3/2} \log p}{n}} \Big)^{|S \cap S'|} \times \nonumber \\
    & \mathbb{E}_0\Big[\exp(\langle \bX \bbeta_S, \bX \bbeta_{S'} \rangle ) \Big]. \label{eq:lower_fixedss'} 
\end{align}
Next, note that \cite[Lemma 11]{carpentier2018minimax}, implies that up to universal constants depending just on the sub-gaussian norm of the design, we have, 
\begin{align}
    \mathbb{E}_0[\exp(\langle \bX \bbeta_S, \bX \bbeta_{S'} \rangle )] &\leq \exp(n \langle \bbeta_S, \bbeta_{S'} \rangle) \Big( 1+ O\Big(  \frac{p (\log p)^2}{n}\Big) \Big) \nonumber \\
    &\leq \exp(2r\log p |S \cap S'|) \Big( 1 + O\Big( \frac{p (\log p)^2}{n} \Big) \Big). \nonumber 
\end{align}
Plugging this back into \eqref{eq:lower_fixedss'}, we obtain the upper bound, 
\begin{align}
   & \mathbb{E}_0\Big[\tilde{L}_{\pi}^2 \mathbf{1}_{\mathcal{G}_2} \Big]= \Big( 1 + O\Big( \frac{p (\log p)^2}{n} \Big) \Big) \mathbb{E}_{S,S'}\Big[ \Phi\Big(\sqrt{2\log p} - \sqrt{2r \log p} + C' \sqrt{\frac{p^{3/2}\log p}{n}}\Big)^{2(s- |S \cap S'|)} \nonumber \\ 
    &\Phi\Big(\sqrt{2\log p} - 2 \sqrt{2r\log p} + C' \sqrt{\frac{p^{3/2}\log p}{n}} \Big)^{|S\cap S'|}  \exp( 2r \log p |S \cap S'| ) \Big], \nonumber 
\end{align}
where $\mathbb{E}_{S,S'}[\cdot]$ denotes the joint expectation under the random independent sampling of $S,S'$ under the prior $\pi$. For notational convenience, set $b_{j} = \mathbb{P}[j\sqrt{2r \log p} + Z \leq \sqrt{2\log p} + C' \sqrt{\frac{p^{3/2}\log p}{n}}]$, for $j = 1,2$, and $Z\sim \mathcal{N}(0,1)$. Also, note that for iid samples $S, S'$ from the prior $\pi$, $W:=|S\cap S'| \sim \mathrm{Hyp}(p,s,s)$. Armed with this notation, we can simplify the second moment upper bound--- specifically, we obtain, 
\begin{align}
    \mathbb{E}_0[\tilde{L}_{\pi}^2\mathbf{1}_{\mathcal{G}_2}] &= \Big( 1 + O\Big( \frac{p (\log p)^2}{n} \Big)\Big) b_1^{2s}\sum_{k=0}^{s} \Big[\frac{b_2}{b_1^2} e^{A^2} \Big]^k \mathbb{P}[W=k]. \nonumber \\
    &= \Big( 1 + O\Big( \frac{p (\log p)^2}{n} \Big)\Big) \Big(b_1^{2s} \mathbb{P}[W=0] + b_1^{2s} \sum_{k=1}^{s} \Big[\frac{b_2}{b_1^2} e^{A^2} \Big]^k \mathbb{P}[W=k]\Big) \nonumber \\
    &\leq \Big( 1 + O\Big( \frac{p (\log p)^2}{n} \Big)\Big) \Big(b_1^{2s} \mathbb{P}[W=0] + \sum_{k=1}^{s} \Big[\frac{b_2}{b_1^2} e^{A^2} \Big]^k \mathbb{P}[W=k]\Big). \nonumber 
\end{align}
Direct computation reveals that $\mathbb{P}[W=0] = 1- \frac{s^2}{p}(1+o(1))$. We have $n\gg p^2 (\log p)^2$, and thus 
\begin{align}
     \mathbb{E}_0[\tilde{L}_{\pi}^2\mathbf{1}_{\mathcal{G}_2}] \leq b_1^{2s} + (1+o(1)) \sum_{k=1}^{s} \Big[\frac{b_2}{b_1^2} e^{A^2} \Big]^k \mathbb{P}[W=k]. \nonumber
\end{align}
Note that if $4r \leq 1$, 
\begin{align}
    \sum_{k=1}^s \Big[ \frac{b_2}{b_1^2} e^{A^2} \Big]^k \mathbb{P}[W=k] \leq C \Big( \exp\Big((1+o(1) e^{A^2} \frac{s^2}{n} \Big) -1\Big)  \leq p^{2r- 2\alpha +1 +o(1)}. \nonumber
\end{align}
On the other hand, for $4r>1$, $b_2 = p^{-(2\sqrt{r}-1)^2}$
\be
    \frac{s^2}{n} e^{A^2} \frac{b_2}{b_1^2} = p^{1-\alpha - (1-\sqrt{r})^2 +o(1)}
\ee

\noindent 
Plugging these bounds into \eqref{eq:lower_chebychev}, we have, 
\begin{align}
    \mathbb{E}_0[| \tilde{L}_{\pi} -1| \mathbf{1}_{\mathcal{G}_2}] &\leq \sqrt{\mathbb{E}_0[\tilde{L}_{\pi}^2 \mathbf{1}_{\mathcal{G}_2}]- 2 \mathbf{E}_0[\tilde{L}_{\pi} \mathbf{1}_{\mathcal{G}_2}] + \mathbb{P}(\mathcal{G}_2)} 
    \leq \sqrt{\mathbb{E}_0[\tilde{L}_{\pi}^2\mathbf{1}_{\mathcal{G}_2}] - \mathbb{E}_0^2[\tilde{L}_{\pi} \mathbf{1}_{\mathcal{G}_2}] + (1 - \mathbb{E}_0[\tilde{L}_{\pi} \mathbf{1}_{\mathcal{G}_2} ] )^2 }. \nonumber 
\end{align}
This implies
\begin{align}
   & \mathbb{E}_0[\tilde{L}_{\pi}^2\mathbf{1}_{\mathcal{G}_2}] - \mathbb{E}_0^2[\tilde{L}_{\pi} \mathbf{1}_{\mathcal{G}_2}] + (1 - \mathbb{E}_0[\tilde{L}_{\pi} \mathbf{1}_{\mathcal{G}_2} ] )^2  \nonumber \\
    &= b_1^{2s} - \mathbb{E}_0^2[\tilde{L}_{\pi} \mathbf{1}_{\mathcal{G}_2}] + (1+o(1)) \sum_{k=1}^{s} \Big[\frac{b_2}{b_1^2} e^{A^2} \Big]^k \mathbb{P}[W=k] + (1 - \mathbb{E}_0[\tilde{L}_{\pi} \mathbf{1}_{\mathcal{G}_2} ] )^2. \nonumber 
\end{align}
Further, using \eqref{eq:below_int_exp}, we have, 
\begin{align}
    &b_1^{2s} - \mathbb{E}_0^2[\tilde{L} \mathbf{1}_{\mathcal{G}_2}] \leq b_1^{2s} - \Phi\Big(\sqrt{2\log p} - \sqrt{2r \log p} - C' \sqrt{\frac{p^{3/2}\log p}{n}} \Big)^{2s} + O(e^{-cp}) \nonumber \\
    &= \Phi\Big(\sqrt{2\log p} - \sqrt{2r \log p} + C' \sqrt{\frac{p^{3/2}\log p}{n}} \Big)^{2s} - \Phi\Big(\sqrt{2\log p} - \sqrt{2r \log p} - C' \sqrt{\frac{p^{3/2}\log p}{n}} \Big)^{2s} + O(e^{-cp}). \nonumber \\
    &= 2sC' \Phi(\xi_1)^{2s-1} \phi(\xi_1) \sqrt{\frac{p^{3/2}\log p}{n}} + 2s C' \Phi(\xi_2)^{2s-1} \phi(\xi_2) \sqrt{\frac{p^{3/2}\log p}{n}} + O(e^{-cp}), \label{eq:below_error} 
\end{align}
where the last equality follows using the Mean Value Theorem for $\xi_1 \in \Big[\sqrt{2\log p}- \sqrt{2r\log p}, \sqrt{2\log p} - \sqrt{2r\log p} + C' \sqrt{\frac{p^{3/2}\log p}{n}} \Big]$, $\xi_2 \in \Big[\sqrt{2\log p}- \sqrt{2r\log p}-  C' \sqrt{\frac{p^{3/2}\log p}{n}}, \sqrt{2\log p} - \sqrt{2r\log p} \Big]$. As $n \gg p^2 (\log p)^2$, for $i=1,2$,
\begin{align}
    s\phi(\xi_i) \sqrt{\frac{p^{3/2}\log p}{n}} = (1+o(1)) p^{1-\alpha- (1-\sqrt{r})^2} \sqrt{\frac{p^{3/2}\log p}{n}}. \nonumber 
\end{align}
Recall that in this case, $\alpha \in (\frac{1}{2},1)$. If $4r >1$, note that $p^{1-\alpha- (1-\sqrt{r})^2} \sqrt{\frac{p^{3/2}\log p}{n}}= o(p^{2(1-\alpha - (1-\sqrt{r})^2)})$ provided $n\gg p^{2}\log p$. On the other hand, if $4r \leq 1$, 
$p^{1-\alpha- (1-\sqrt{r})^2} \sqrt{\frac{p^{3/2}\log p}{n}}= o(p^{2(r- (\alpha - \frac{1}{2} ) )} )$ provided $n \gg p^{13/6}\log p$. 

Under these assumptions, 
\begin{align}
    \mathbb{E}_0[| \tilde{L}_{\pi} -1| \mathbf{1}_{\mathcal{G}_2}] \leq \Big[b_1^{2s} - \mathbb{E}_0^2[\tilde{L}_{\pi} \mathbf{1}_{\mathcal{G}}] + (1+o(1)) \sum_{k=1}^{s} \Big[\frac{b_2}{b_1^2} e^{A^2} \Big]^k \mathbb{P}[W=k] + (1 - \mathbb{E}_0[\tilde{L}_{\pi} \mathbf{1}_{\mathcal{G}_2} ] )^2 \Big]^{\frac{1}{2}}. \nonumber 
\end{align}
Plugging in the prior estimates \eqref{eq:below_int2}, \eqref{eq:below_error} the proof follows. 


\subsection{Proof of Theorem \ref{theorem:above_boundary_dense} Part (a)}
\textbf{Proof of Upper Bound:} 
Recall $\bz$ from \eqref{eq:z_defn}, and 
for any $\tau>0$, consider first a test $T_{\tau}$ that rejects when $\frac{\|\bz\|_2^2-p}{\sqrt{2p}}>\tau$. Since $\by|\bX\sim \mathcal{N}(\bX\bbeta,I_p)$, we have that under $H_0$, $\|\bz\|_2^2\sim \chi^2_p$. Therefore, using Lemma \ref{lemma:chisq_conc}, we have that the Type I error of $T_{\tau}$ equals 
\be 
\P_{0}\left(T_{\tau}=1\right)&=\P\left(\frac{\chi^2_p-p}{\sqrt{2p}}>\tau\right)&\leq e^{-\frac{\tau^2}{2}(1+o(1))}, \quad \text{if} \quad \tau=o(\sqrt{p}).
\ee
For the Type II error of $T$ note that $\bz=(\bX^T\bX)^{\frac{1}{2}}\bbeta+\boldsymbol{\eta}$ where $\boldsymbol{\eta}=(\bX^T\bX)^{-\frac{1}{2}}\bX^T\beps\sim \mathcal{N}(\mathbf{0},I_p)$. 
Therefore, the Type II error of $T$ under any $\bbeta \in \Xi(s,A)$ equals 
\be 
\P_{\bbeta}\left(T_{\tau}=0\right)&=\P_{\bbeta}\left(\frac{\chi^2_p\left(\|(\bX^T\bX)^{\frac{1}{2}}\bbeta\|_2^2\right)-p}{\sqrt{2p}}\leq \tau\right)\\
&\leq \P_{\bbeta}\left(\frac{\chi^2_p\left(\|(\bX^T\bX)^{\frac{1}{2}}\bbeta\|_2^2\right)-p}{\sqrt{2p}}\leq \tau,\mathcal{A}_{\kappa}\right)+\P_{\bbeta}(\mathcal{A}_{\kappa}^c),
\ee 
where for any $\kappa \in (0,1)$ we define
\be 
\mathcal{A}_{\kappa}=\left\{\|(\bX^T\bX)^{\frac{1}{2}}\bbeta\|_2^2\geq (1-\kappa)n\|\bbeta\|_2^2\right\}.
\ee
 Now
 , by Lemma \ref{lemma:rip}, there exists  $C,c>0$ (depending only on the subgaussian norm of the rows of $\bX$) such that for any $\kappa \geq C\frac{s\log{(p/s)}}{n}$
 \be 
 \P_{\bbeta}(\mathcal{A}_{\kappa}^c)&\leq 2e^{-cn},
 \ee
 for a $c>0$ (once again only depending only on the subgaussian norm of the rows of $\bX$). Moreover, by stochastic monotonicity of non-central chi-squares random variables in terms of the non-centrality parameter, we have for any $\bbeta\in \Xi(s,A)$
 \be 
 \P_{\bbeta}\left(\frac{\chi^2_p\left(\|(\bX^T\bX)^{\frac{1}{2}}\bbeta\|_2^2\right)-p}{\sqrt{2p}}\leq \tau,\mathcal{A}_{\kappa}\right)&\leq \P_{\bbeta}\left(\frac{\chi^2_p\left((1-\kappa)n\|\bbeta\|_2^2\right)-p}{\sqrt{2p}}\leq \tau\right)\\ &\leq \P_{\bbeta}\left(\frac{\chi^2_p\left((1-\kappa)nsA^2\right)-p}{\sqrt{2p}}\leq \tau\right).
 \ee 
Now choosing $\tau\sqrt{2p}=\frac{nsA^2}{2}=O(p^{\frac{1}{2}+\delta})$, we have that $\tau=o(\sqrt{p})$ since $\delta<\frac{1}{2}$ and also $(1-\kappa)nsA^2>\tau\sqrt{2p}$ whenever $\kappa<\frac{1}{2}$. Consequently, we have for $\kappa<\frac{1}{2}$, by putting $x=\frac{p}{p+2(1-\kappa)nsA^2}\frac{\tau^2}{2}(1-2\kappa)^2$ (the solution of $2\tau(1/2-\kappa)\sqrt{2p}=2\sqrt{(p+(1-\kappa)nsA^2)x}$), the following holds by Lemma \ref{lemma:noncchisq_conc}
\be 
\P_{\bbeta}\left(\frac{\chi^2_p\left((1-\kappa)nsA^2\right)-p}{\sqrt{2p}}\leq \tau\right)&\leq \exp\left(-\frac{p}{p+2(1-\kappa)nsA^2}\frac{\tau^2}{2}(1-2\kappa)^2\right)
\ee
The required upper bound follows upon combining Type I and Type II errors of this test, with $\tau = \frac{p^{\delta}}{2\sqrt{2}}$, and sending $\kappa \to 0$, at a rate slower than $s\log(p/s)/n$ (this is possible in this setting as $p\ll n$ and $s=o(\sqrt{p})$).  

\textbf{Proof of Lower Bound:} 
As usual, we start with the likelihood ratio 
\begin{align}
    L(\bbeta) = \exp{\Big( \langle \by, \bX\bbeta \rangle - \frac{1}{2} \| \bX \bbeta \|_2^2 \Big) }.  \nonumber 
\end{align}
and consider the uniform prior $\pi$ on the parameter space $\tilde\Xi(s,A)$ which chooses $s$ locations uniformly at random and sets these coordinates to be $\pm A$ at random. All other coordinates are set at zero. We define the integrated likelihood ratio $L_{\pi} = \E_{\bbeta\sim \pi} [L(\bbeta)]$, and start with the classical lower bound (see e.g. \cite{ingster2012nonparametric})
\be
    \mathrm{Risk}(s,A) \geq \Pzero{[L_{\pi} >1]} + \E_{\bbeta \sim \pi }[\Pbeta{[ L_{\pi} \leq 1]}]\geq \E_{\bbeta \sim \pi }[\Pbeta{[ L_{\pi} \leq 1]}].
\ee
Recall the event $\mathcal{G}_1$ introduced in Lemma \ref{lemma:rip}. Fix any $\bbeta\in \tilde{\Xi}(s,A)$ and note that for any $\tau>0$
\be 
\Pbeta{[ L_{\pi} \leq 1]}
&\geq \P_{\bbeta}\left(\frac{\|(\bX^T\bX)^{-\frac{1}{2}}\bX^T\by\|_2^2-p}{\sqrt{2p}}\leq\tau,L_{\pi}\leq 1,\mathcal{G}_1\right)\\
&= \P_{\bbeta}\left(\frac{\|(\bX^T\bX)^{-\frac{1}{2}}\bX^T\by\|_2^2-p}{\sqrt{2p}}\leq\tau,\mathcal{G}_1\right)-\P_{\bbeta}\left(\frac{\|(\bX^T\bX)^{-\frac{1}{2}}\bX^T\by\|_2^2-p}{\sqrt{2p}}\leq\tau,L_{\pi}>1,\mathcal{G}_1\right)\\
&=I-II.
\ee 

The following lemmas collect a lower bound on $I$ and and upper bound on $II$. We defer their proof to the end of the Section, and first complete the proof of the Theorem, given these bounds. 

\begin{lemma}
\label{theorem:above_boundary_dense_1upper}
In the setting introduced above, whenever $n,p \to \infty$ with $n \gg p \log p$, we have, 
\begin{align}
I\geq \frac{1}{2}(1+o(1))\frac{\exp\Big(-\frac{\left(\tau\sqrt{\frac{2p}{2(p+\lambda)}}+\frac{\kappa nsA^2}{\sqrt{2(p+\lambda)}}\right)^2}{2}\Big)}{\sqrt{2\pi}\Big(\tau\sqrt{\frac{2p}{2(p+\lambda)}}+\frac{\kappa nsA^2}{\sqrt{2(p+\lambda)}}\Big)},  \nonumber \end{align}
with $\kappa = C \frac{s\log(p/s)}{n}$ for some universal constant $C>0$, depending only on the sub-gaussian norm of the design $\bX$.
\end{lemma}

\begin{lemma}
\label{theorem:above_boundary_dense_2lower}
For $0<\delta<\frac{1}{6}$ with $\alpha - \frac{1}{2} + 2\delta<0$, and $n,p \to \infty$ with $n \gg p^{1/2+ \delta} \log p$, we have, 
\begin{align}
    II \leq \frac{1}{\sqrt{2\pi}}\cdot\frac{75}{219\tau}\cdot \exp(-\frac{\tau^2}{2}) +e^{-C^*p^{1-\alpha}/4}, \nonumber 
\end{align}
where $C^*>0$ is a universal constant. 
\end{lemma}
\noindent 
The proof of the Theorem follows upon combining Lemma \ref{theorem:above_boundary_dense_1upper} and \ref{theorem:above_boundary_dense_2lower}. \qed

It remains to prove Lemma \ref{theorem:above_boundary_dense_1upper} and Lemma \ref{theorem:above_boundary_dense_1upper}. We turn to Lemma \ref{theorem:above_boundary_dense_1upper} first. 
\begin{proof}[Proof of Lemma \ref{theorem:above_boundary_dense_1upper}]
We first provide a lower bound for $I$ as follows (with $\kappa=C \frac{s\log (p/s)}{n}$).
\be 
\P_{\bbeta}\left(\frac{\|(\bX^T\bX)^{-\frac{1}{2}}\bX^T\by\|_2^2-p}{\sqrt{2p}}\leq\tau,\mathcal{G}_1\right)&=\P_{\bbeta}\left(\frac{\chi^2_p\left(\|(\bX^T\bX)^{\frac{1}{2}}\bbeta\|_2^2\right)-p}{\sqrt{2p}}\leq\tau,\mathcal{G}_1\right)\\
&\geq \P_{\bbeta}\left(\frac{\chi^2_p\left((1+\kappa)n\|\bbeta\|_2^2\right)-p}{\sqrt{2p}}\leq \tau,\mathcal{G}_1\right)\\
&=\P_{\bbeta}\left(\frac{\chi^2_p\left((1+\kappa)n\|\bbeta\|_2^2\right)-p}{\sqrt{2p}}\leq \tau\right)\P_{\bbeta}(\mathcal{G}_1).
\ee

Note that under $\bbeta \sim \pi$ one has $\|\bbeta\|_2^2=sA^2$ and that $\chi_p^2(\lambda)\stackrel{d}{=} \chi_{k+2J}^2(0)$ with $J\sim \mathrm{Poisson}(\lambda/2)$. Hence letting $\lambda=(1+\kappa)n\|\bbeta\|_2^2=(1+\kappa)nsA^2$ and $\mathcal{G}_1'$ denote the event that $\{J\leq \lambda/2\}$ we have 
\be 
\P_{\bbeta}\left(\frac{\chi^2_p\left((1+\kappa)n\|\bbeta\|_2^2\right)-p}{\sqrt{2p}}\leq \tau\right)&\geq\P\left(\frac{\chi_{p+2J}^2(0)-(p+\lambda)}{\sqrt{2(p+\lambda)}}\leq \tau\sqrt{\frac{2p}{2(p+\lambda)}}-\frac{\lambda}{\sqrt{2(p+\lambda)}},\mathcal{G}_1'\right)\\
&\geq \P\left(\frac{\chi_{p+\lambda}^2(0)-(p+\lambda)}{\sqrt{2(p+\lambda)}}\leq \tau\sqrt{\frac{2p}{2(p+\lambda)}}-\frac{\lambda}{\sqrt{2(p+\lambda)}},\mathcal{G}_1'\right)\\
&=\frac{1}{2}(1+o(1))\P\left(\frac{\chi_{p+\lambda}^2(0)-(p+\lambda)}{\sqrt{2(p+\lambda)}}\leq \tau\sqrt{\frac{2p}{2(p+\lambda)}}-\frac{\lambda}{\sqrt{2(p+\lambda)}}\right).
\ee
The last line in the display above uses $\mathbb{P}(\mathcal{G}_1')\to 1/2$, which follows from the Central Limit Theorem, by noting that $\lambda\rightarrow \infty$ as $n,p\rightarrow \infty$. Now note that $\lambda=O(p^{1/2+\delta})\ll p$ whenever $\delta<\frac{1}{2}$. Now we let $\tau=nsA^2/2\sqrt{2p}$ and therefore $\lambda=2\sqrt{2p}\tau+\kappa nsA^2$. As a result, $\frac{\lambda}{\sqrt{2(p+\lambda)}}=(2\tau\sqrt{2p}+\kappa nsA^2)/\sqrt{2(p+\lambda)}$. Consequently, we need to bound from below
\be
\P\left(\frac{\chi_{p+\lambda}^2(0)-(p+\lambda)}{\sqrt{2(p+\lambda)}}\leq -\tau\sqrt{\frac{2p}{2(p+\lambda)}}-\frac{\kappa nsA^2}{\sqrt{2(p+\lambda)}}\right)
\ee
Now note that $\frac{\kappa nsA^2}{\sqrt{2(p+\lambda)}}=O(\frac{p^{1+\delta}\sqrt{\log{p}}}{n})\ll \tau$ whenever $p\log{p}\ll n$. For such $(n,p)$ pair and $\tau=O(p^{\delta})$ with $\delta<1/6$, we have  by Cram\'{e}r Type Moderate Deviation Lower Bound \cite[Theorem 2.13, Part (b)]{pena2008self} that the following holds
\be 
I\geq \frac{1}{2}(1+o(1))\frac{\exp\Big(-\frac{\left(\tau\sqrt{\frac{2p}{2(p+\lambda)}}+\frac{\kappa nsA^2}{\sqrt{2(p+\lambda)}}\right)^2}{2}\Big)}{\sqrt{2\pi}\Big(\tau\sqrt{\frac{2p}{2(p+\lambda)}}+\frac{\kappa nsA^2}{\sqrt{2(p+\lambda)}}\Big)}.
\ee
\end{proof}

\noindent 
Finally, we turn to the proof of Lemma \ref{theorem:above_boundary_dense_2lower}.

\begin{proof}[Proof of Lemma \ref{theorem:above_boundary_dense_2lower}]
Next we provide an upper bound on $II$ as follows.  To this end note that 
\be 
\ & \E_{\bbeta\sim \pi}\mathbb{E}_0\Big[L_{\bbeta}\,\mathbf{1}\left(\frac{\|(\bX^T\bX)^{-\frac{1}{2}}\bX^T\by\|_2^2-p}{\sqrt{2p}}\leq\tau,L_{\pi}>1,\mathcal{G}_1\right)\Big]\\
&=\E_{\mathbf{0}}\left[L_{\pi}\,\mathbf{1}\left(\frac{\|(\bX^T\bX)^{-\frac{1}{2}}\bX^T\by\|_2^2-p}{\sqrt{2p}}\leq\tau,L_{\pi}>1\right)\mathbf{1}(\mathcal{G}_1)\right]\\
&\leq \E_{\mathbf{0}}\left[L_{\pi}^2\,  \mathbf{1}\left(\frac{\|(\bX^T\bX)^{-\frac{1}{2}}\bX^T\by\|_2^2-p}{\sqrt{2p}}\leq\tau\right)\mathbf{1}(\mathcal{G}_1)\right]\\
&=\int_{\bbeta,\bbeta'}\E_{\bX}\left[\mathbf{1}(\mathcal{G}_1)e^{\langle \bX\bbeta,\bX\bbeta'\rangle}\P_{\bbeta+\bbeta'}\left(\frac{\|(\bX^T\bX)^{-\frac{1}{2}}\bX^T\by\|_2^2-p}{\sqrt{2p}}\leq\tau\right)\right]d\pi(\bbeta)d\pi(\bbeta')
\ee
Now note, on the event $\mathcal{G}_1$ one has that uniformly for all $\bbeta\in \mathrm{supp}(\pi)$ that $\langle \bX\bbeta,\bX\bbeta'\rangle \leq n \langle\bbeta,\bbeta'\rangle+C\log(p/s)s^2A^2$. Consequently,
\be
\ &  \E_{\bbeta\sim \pi}\P_{\bbeta}\left(L_{\bbeta}\frac{\|(\bX^T\bX)^{-\frac{1}{2}}\bX^T\by\|_2^2-p}{\sqrt{2p}}\leq\tau,L_{\pi}>1,\mathcal{G}_1\right)\\
&\leq e^{Cs^2A^2\log{(p/s)}}\int_{\bbeta,\bbeta'}\E_{\bX}\left[\mathbf{1}(\mathcal{G}_1)e^{n\langle \bbeta,\bbeta'\rangle}\P_{\bbeta+\bbeta'}\left(\frac{\|(\bX^T\bX)^{-\frac{1}{2}}\bX^T\by\|_2^2-p}{\sqrt{2p}}\leq\tau\right)\right]d\pi(\bbeta)d\pi(\bbeta')\\
&= e^{Cs^2A^2\log{(p/s)}}\int_{\bbeta,\bbeta'}\E_{\bX}\left[\mathbf{1}(\mathcal{G}_1)e^{n\langle \bbeta,\bbeta'\rangle}\P_{\bbeta+\bbeta'}\left(\frac{\chi^2_p\left(\|(\bX^T\bX)^{1/2}(\bbeta+\bbeta')\|_2^2\right)-p}{\sqrt{2p}}\leq\tau\right)\right]d\pi(\bbeta)d\pi(\bbeta')\\
&\leq e^{Cs^2A^2\log{(p/s)}}\int_{\bbeta,\bbeta'}e^{n\langle \bbeta,\bbeta'\rangle}\P\left(\frac{\chi^2_p\left(n(1-\kappa)\|(\bbeta+\bbeta')\|_2^2\right)-p}{\sqrt{2p}}\leq\tau\right)d\pi(\bbeta)d\pi(\bbeta').\\
\ee
Let $\mathcal{G}(\bbeta,\bbeta')=\{|\mathrm{supp}(\bbeta)\cap \mathrm{supp}(\bbeta')|\leq s/200\}$. Then
\be 
\ & e^{Cs^2A^2\log{(p/s)}}\int_{\bbeta,\bbeta'}e^{n\langle \bbeta,\bbeta'\rangle}\P\left(\frac{\chi^2_p\left(n(1-\kappa)\|(\bbeta+\bbeta')\|_2^2\right)-p}{\sqrt{2p}}\leq\tau\right)d\pi(\bbeta)d\pi(\bbeta')\\
&\leq e^{Cs^2A^2\log{(p/s)}}\int_{\bbeta,\bbeta'}e^{n\langle \bbeta,\bbeta'\rangle}\P\left(\frac{\chi^2_p\left(n(1-\kappa)\|(\bbeta+\bbeta')\|_2^2\right)-p}{\sqrt{2p}}\leq\tau\right)\mathbf{1}(\mathcal{G}(\bbeta,\bbeta'))d\pi(\bbeta)d\pi(\bbeta')\\
&+e^{Cs^2A^2\log{(p/s)}}\E_{\bbeta,\bbeta'\sim \pi}\left(e^{n\langle \bbeta,\bbeta'\rangle}\mathbf{1}\left(|\mathrm{supp}(\bbeta)\cap \mathrm{supp}(\bbeta')|> s/200\right)\right)
\ee
Now note that since $|\mathrm{supp}(\bbeta)\cap \mathrm{supp}(\bbeta')|\sim \mathrm{Hypergeometric}(p,s,s)$, \cite[Theorem 4]{janson2016large} implies that there exists a constant $C^*$ such that 
\be 
P_{\bbeta,\bbeta'\sim \pi}\left(|\mathrm{supp}(\bbeta)\cap \mathrm{supp}(\bbeta')|> s/200\right)\leq e^{-C^*s}.
\ee
Therefore by Cauchy-Schwarz Inequality
\be 
e^{Cs^2A^2\log{(p/s)}}\E_{\bbeta,\bbeta'\sim \pi}\left(e^{n\langle \bbeta,\bbeta'\rangle}\mathbf{1}\left(|\mathrm{supp}(\bbeta)\cap \mathrm{supp}(\bbeta')|> s/200\right)\right) \\
\leq e^{Cs^2A^2\log ( p/s) }e^{-C^* s/2}\E^{1/2}_{\bbeta,\bbeta'\sim \pi}\left(e^{2n\langle \bbeta,\bbeta'\rangle}\right).
\ee
Now direct calculation shows that there exists a constant $C',C''$ such that 
\be 
\E_{\bbeta,\bbeta'\sim \pi}\left(e^{2n\langle \bbeta,\bbeta'\rangle}\right)\leq e^{C'n^2s^2A^4/p}\leq e^{C''p^{2\delta}}.
\ee
Consequently
\be 
\ & e^{Cs^2A^2\log{(p/s)}}\E_{\bbeta,\bbeta'\sim \pi}\left(e^{n\langle \bbeta,\bbeta'\rangle}\mathbf{1}\left(|\mathrm{supp}(\bbeta)\cap \mathrm{supp}(\bbeta')|> s/200\right)\right)\\
&\leq e^{Cs^2A^2\log{(p/s)}-C^* s/2+ C''p^{2\delta}}\\
&\leq e^{C\log{(p/s)}p^{3/2-\alpha+\delta}/n+C''p^{2\delta}/2-C^*p^{1-\alpha}/2}\\
&\ll e^{-C^*p^{1-\alpha}/4}\ll \exp(-p^{2\delta})
\ee
\textcolor{black}{provided $1-\alpha>2\delta$ and $\log{(p/s)}p^{3/2-\alpha+\delta}/n\ll p^{1-\alpha}$ (which happens whenever $p^{1/2+\delta}\log{(p/s)}\ll n$ and is true here since $\delta<1/10$).}

Note that under $\bbeta \sim \pi$ one has $\|\bbeta\|_2^2=sA^2$ and that $\chi_p^2(\lambda)\stackrel{d}{=} \chi_{k+2J}^2(0)$ with $J\sim \mathrm{Poisson}(\lambda/2)$. Now on the event $\mathcal{G}(\bbeta,\bbeta')$ note that $\lambda=(1-\kappa)n\|\bbeta+\bbeta'\|_2^2=\Omega(nsA^2)$. Let also $\mathcal{G}_{a_n}$ denote the event that $\{J\geq \lambda/2-\sqrt{a_n\lambda/2}\}$. Let now $\lambda_n=\lambda_n(\bbeta,\bbeta')=2(\lambda/2-\sqrt{a_n\lambda/2})$. Then 
\be 
\ & \P\left(\frac{\chi^2_p\left(n(1-\kappa)\|(\bbeta+\bbeta')\|_2^2\right)-p}{\sqrt{2p}}\leq\tau\right)\mathbf{1}(\mathcal{G}(\bbeta,\bbeta'))\\
&=\P\left(\frac{\chi^2_{p+2J}\left(0\right)-p}{\sqrt{2p}}\leq\tau\right)\mathbf{1}(\mathcal{G}(\bbeta,\bbeta'))\\
&\leq\P\left(\frac{\chi^2_{p+2J}\left(0\right)-p}{\sqrt{2p}}\leq\tau,\mathcal{G}_{a_n}\right)\mathbf{1}(\mathcal{G}(\bbeta,\bbeta'))+\P(\mathcal{G}_{a_n}^c)\mathbf{1}(\mathcal{G}(\bbeta,\bbeta'))
\ee
\textcolor{black}{Using Moderate Deviation bounds \cite[Lemma 2]{arias2015sparse} we note  that as long as $1\ll a_n\ll \lambda=\Omega(nsA^2)$ we have for any $\epsilon>0$ and $n,p$ large enough}
\be 
\P(\mathcal{G}_{a_n}^c)\mathbf{1}(\mathcal{G}(\bbeta,\bbeta'))&\leq \exp\left(-(1-\epsilon)a_n/2\right).
\ee
Also, 
\be 
\ & \P\left(\frac{\chi^2_{p+2J}\left(0\right)-p}{\sqrt{2p}}\leq\tau,\mathcal{G}_{a_n}\right)\mathbf{1}(\mathcal{G}(\bbeta,\bbeta'))\\
&=\P\left(\frac{\chi_{p+2J}^2(0)-(p+\lambda_n)}{\sqrt{2(p+\lambda_n)}}\leq \tau\sqrt{\frac{2p}{2(p+\lambda_n)}}-\frac{\lambda_n}{\sqrt{2(p+\lambda_n)}},\mathcal{G}_{a_n}\right)\mathbf{1}(\mathcal{G}(\bbeta,\bbeta'))
\\
&\leq \P\left(\frac{\chi_{p+\lambda_n}^2(0)-(p+\lambda_n)}{\sqrt{2(p+\lambda_n)}}\leq \tau\sqrt{\frac{2p}{2(p+\lambda_n)}}-\frac{\lambda_n}{\sqrt{2(p+\lambda_n)}},\mathcal{G}_{a_n}\right)\mathbf{1}(\mathcal{G}(\bbeta,\bbeta'))
\ee
Recall that $\tau = \frac{nsA^2}{2 \sqrt{2p}}$ and $\lambda = n(1-\kappa) \| \bbeta + \bbeta'\|_2^2$. Thus we have, 
\be 
\ & \tau\sqrt{\frac{2p}{2(p+\lambda_n)}}-\frac{\lambda_n}{\sqrt{2(p+\lambda_n)}}\\&\leq -3\tau\sqrt{\frac{2p}{2(p+\lambda_n)}}+4\kappa\tau\sqrt{\frac{2p}{2(p+\lambda_n)}}-\frac{2(1-\kappa)}{\sqrt{2(p+\lambda_n)}}n\langle\bbeta,\bbeta'\rangle+\frac{\sqrt{2a_n\lambda_n}}{\sqrt{2(p+\lambda_n)}}\\
&\leq -3\tau\sqrt{\frac{2p}{2(p+\lambda_n)}}-\frac{2(1-\kappa)}{\sqrt{2(p+\lambda_n)}}n\langle\bbeta,\bbeta'\rangle+4\kappa\tau\sqrt{\frac{2p}{2(p+\lambda_n)}}+
\frac{\sqrt{2a_n\lambda_n}}{\sqrt{2(p+\lambda_n)}}
\ee
Now note that on the event $\mathcal{G}(\bbeta,\bbeta')$ we have for large enough $n,p$ $|\frac{2(1-\kappa)}{\sqrt{2(p+\lambda_n)}}n\langle\bbeta,\bbeta'\rangle|\leq \frac{2(1-\kappa)nsA^2}{200\sqrt{2p}}= (1-\kappa)\tau/50\leq \tau/25$ -- since $\kappa=C \frac{s\log (p/s)}{n}$ converges to $0$ as $n,p\rightarrow \infty$. 

Also $\lambda_n=\lambda_n(\bbeta,\bbeta')=2(\lambda/2-\sqrt{a_n\lambda/2})\leq \lambda=(1-\kappa)n\|\bbeta+\bbeta'\|_2^2=O(nsA^2)$. Note that $nsA^2=O(p^{1/2+\delta})\ll p$ for \textcolor{black}{$\delta<\frac{1}{2}$}. 

Combining all the above we have for \textcolor{black}{$0<\delta<1/6$} the following holds by Cramer Type Moderate Deviation \cite[Theorem 2.13, Part (b)]{pena2008self} (with the notation $\gamma_n=\frac{4\kappa}{3}+\frac{\sqrt{2a_n\lambda}}{3\tau\sqrt{2p}}$)
\be 
 \ & \P\left(\frac{\chi_{p+\lambda_n}^2(0)-(p+\lambda_n)}{\sqrt{2(p+\lambda_n)}}\leq \tau\sqrt{\frac{2p}{2(p+\lambda_n)}}-\frac{\lambda_n}{\sqrt{2(p+\lambda_n)}},\mathcal{G}_{a_n}\right)\mathbf{1}(\mathcal{G}(\bbeta,\bbeta'))\\
 &\leq \frac{1}{\sqrt{2\pi}}\frac{\exp\left(-\frac{9\tau^2}{2}\left(\sqrt{\frac{2p}{2(p+\lambda_n)}}(1-\gamma_n)+\frac{2(1-\kappa)}{3\tau\sqrt{2(p+\lambda_n)}}n\langle\bbeta,\bbeta'\rangle\right)^2\right)\mathbf{1}(\mathcal{G}(\bbeta,\bbeta'))}{3\tau\left(\sqrt{\frac{2p}{2(p+\lambda_n)}}(1-\gamma_n)+\frac{2(1-\kappa)}{3\tau\sqrt{2(p+\lambda_n)}}n\langle\bbeta,\bbeta'\rangle\right)}\\
 &\leq  \frac{1}{\sqrt{2\pi}}\frac{\exp\left(-\frac{9\tau^2}{2}\left(\sqrt{\frac{2p}{2(p+\lambda_n)}}(1-\gamma_n)+\frac{2(1-\kappa)}{3\tau\sqrt{2(p+\lambda_n)}}n\langle\bbeta,\bbeta'\rangle\right)^2\right)\mathbf{1}(\mathcal{G}(\bbeta,\bbeta'))}{3\tau\left(\sqrt{\frac{2p}{2(p+\lambda_n)}}(1-\gamma_n)-\frac{1}{3\tau}\frac{\tau}{25}\right)}\\
 &= \frac{1}{\sqrt{2\pi}}\frac{\exp\left(-\frac{9\tau^2}{2}\left(\sqrt{\frac{2p}{2(p+\lambda_n)}}(1-\gamma_n)+\frac{2(1-\kappa)}{3\tau\sqrt{2(p+\lambda_n)}}n\langle\bbeta,\bbeta'\rangle\right)^2\right)\mathbf{1}(\mathcal{G}(\bbeta,\bbeta'))}{3\tau\left(\sqrt{\frac{2p}{2(p+\lambda_n)}}(1-\gamma_n)-\frac{1}{75}\right)}\\
 &\leq \frac{1}{\sqrt{2\pi}}\frac{\exp\left(-\frac{9\tau^2}{2}\frac{2p}{2(p+\lambda_n)}(1-\gamma_n)^2-\frac{9\tau^2}{2}\times 2 \times \sqrt{\frac{2p}{2(p+\lambda_n)}}(1-\gamma_n)\times \frac{2(1-\kappa)}{3\tau\sqrt{2(p+\lambda_n)}}n\langle\bbeta,\bbeta'\rangle \right)\mathbf{1}(\mathcal{G}(\bbeta,\bbeta'))}{3\tau\left(\sqrt{\frac{2p}{2(p+\lambda_n)}}(1-\gamma_n)-\frac{1}{75}\right)}
\ee
Therefore \textcolor{black}{as long as $1\ll a_n\ll \lambda=\Omega(nsA^2)$ and $0<\delta<\frac{1}{6}$}
\be 
\ & 
\int_{\bbeta,\bbeta'}e^{n\langle \bbeta,\bbeta'\rangle}\P\left(\frac{\chi^2_p\left(n(1-\kappa)\|(\bbeta+\bbeta')\|_2^2\right)-p}{\sqrt{2p}}\leq\tau\right)\mathbf{1}(\mathcal{G}(\bbeta,\bbeta'))d\pi(\bbeta)d\pi(\bbeta')\\
&\leq \int_{\bbeta,\bbeta'}e^{n\langle \bbeta,\bbeta'\rangle} \frac{1}{\sqrt{2\pi}}\frac{\exp\left(\begin{array}{c}-\frac{9\tau^2}{2}\frac{2p}{2(p+\lambda_n)}(1-\gamma_n)^2\\-\frac{9\tau^2}{2}\times 2 \times \sqrt{\frac{2p}{2(p+\lambda_n)}}(1-\gamma_n)\times \frac{2(1+\kappa)}{3\tau\sqrt{2(p+\lambda_n)}}n\langle\bbeta,\bbeta'\rangle \end{array}\right)\mathbf{1}(\mathcal{G}(\bbeta,\bbeta'))}{3\tau\left(\sqrt{\frac{2p}{2(p+\lambda_n)}}(1-\gamma_n)-\frac{1}{75}\right)}d\pi(\bbeta)d\pi(\bbeta')\\
&+\int_{\bbeta,\bbeta'}e^{n\langle \bbeta,\bbeta'\rangle}\exp\left(-(1-\epsilon)a_n/2\right)\mathbf{1}(\mathcal{G}(\bbeta,\bbeta'))d\pi(\bbeta)d\pi(\bbeta')
\ee
Now as noted before,  direct calculation shows that there exists a constant $C',C''$ such that 
\be 
\E_{\bbeta,\bbeta'\sim \pi}\left(e^{n\langle \bbeta,\bbeta'\rangle}\right)\leq e^{C'n^2s^2A^4/p}\leq e^{C''p^{2\delta}}.
\ee
Therefore,
\be 
\int_{\bbeta,\bbeta'}e^{n\langle \bbeta,\bbeta'\rangle}\exp\left(-(1-\epsilon)a_n/2\right)\mathbf{1}(\mathcal{G}(\bbeta,\bbeta'))d\pi(\bbeta)d\pi(\bbeta')&\leq e^{C''p^{2\delta}-(1-\epsilon)a_n/2}
\ee
\textcolor{black}{Choosing $a_n=C_{large}p^{2\delta}$ for some large constant $C_{large}>C''+10$, we have satisfies $1\ll a_n\ll \lambda=\Omega(nsA^2)$} and 
\be 
\int_{\bbeta,\bbeta'}e^{n\langle \bbeta,\bbeta'\rangle}\exp\left(-(1-\epsilon)a_n/2\right)\mathbf{1}(\mathcal{G}(\bbeta,\bbeta'))d\pi(\bbeta)d\pi(\bbeta')&\leq e^{-(C_{large}-C'')p^{2\delta}/4}\leq e^{-10p^{2\delta}}
\ee
for sufficiently large $n,p$ by choosing $C_{large}$ large enough as above.
Moreover, \textcolor{black}{using the fact that $\lambda_n\leq 4nsA^2$} we have for large enough $n,p$
\be 
\ & \int_{\bbeta,\bbeta'}e^{n\langle \bbeta,\bbeta'\rangle} \frac{1}{\sqrt{2\pi}}\frac{\exp\left(\begin{array}{c}-\frac{9\tau^2}{2}\frac{2p}{2(p+\lambda_n)}(1-\gamma_n)^2\\-\frac{9\tau^2}{2}\times 2 \times \sqrt{\frac{2p}{2(p+\lambda_n)}}(1-\gamma_n)\times \frac{2(1+\kappa)}{3\tau\sqrt{2(p+\lambda_n)}}n\langle\bbeta,\bbeta'\rangle \end{array}\right)\mathbf{1}(\mathcal{G}(\bbeta,\bbeta'))}{3\tau\left(\sqrt{\frac{2p}{2(p+\lambda_n)}}(1-\gamma_n)-\frac{1}{75}\right)}d\pi(\bbeta)d\pi(\bbeta')\\
&\leq \frac{1}{\sqrt{2\pi}}\frac{\exp\left(-\frac{9\tau^2}{2}\frac{2p}{2(p+4nsA^2)}(1-\gamma_n)^2 \right)}{3\tau\left(\sqrt{\frac{2p}{2(p+4nsA^2)}}(1-\gamma_n)-\frac{1}{75}\right)}\E_{\bbeta,\bbeta'\sim \pi}\left[\mathbf{1}(\mathcal{G}(\bbeta,\bbeta'))e^{n\langle \bbeta,\bbeta'\rangle-\frac{9\tau^2}{2}\times 2 \times \sqrt{\frac{2p}{2(p+\lambda_n)}}(1-\gamma_n)\times \frac{2(1+\kappa)}{3\tau\sqrt{2(p+\lambda_n)}}n\langle\bbeta,\bbeta'\rangle}\right]\\
&\leq \frac{1}{\sqrt{2\pi}}\times \frac{1}{3\tau\times \frac{73}{75}}\exp\left(-9\tau^2/2\right)\exp\left(\frac{9\tau^2}{2}\times\frac{8nsA^2}{2p}+\frac{9\tau^2}{2}\times \frac{2p}{2(p+\lambda_n)}\times 2\gamma_n\right)\times \E_{\bbeta,\bbeta'\sim \pi}\left[e^{n\langle\bbeta,\bbeta'\rangle\left(1-\theta_n\right)}\right]
\ee
where \textcolor{black}{$\theta_n=\frac{9\tau^2}{2}\times 2 \times \sqrt{\frac{2p}{2(p+\lambda_n)}}(1-\gamma_n)\times \frac{2(1+\kappa)}{3\tau\sqrt{2(p+\lambda_n)}}=O(\tau/\sqrt{p})$ whenever $\delta<\frac{1}{2}$}. Also note that \textcolor{black}{whenever $\delta<1/10$, $a_n=O(p^{2\delta})$, $p^{7/5}\log{p}\ll n$} one has
\be
\frac{9\tau^2}{2}\times\frac{8nsA^2}{2p}+\frac{9\tau^2}{2}\times \frac{2p}{2(p+\lambda_n)}\times 2\gamma_n\rightarrow 0.
\ee
Now by direct calculations, we have \cite[Theorem 4]{hoeffding1994probability} 
\be 
\E_{\bbeta,\bbeta'\sim \pi}\left[e^{n\langle\bbeta,\bbeta'\rangle\left(1-\theta_n\right)}\right]&\leq \exp\left(\frac{s^2}{p}[\cosh(n(1-\theta_n)A^2)-1]\right)\\
&\leq \exp\left(4\tau^2(1-\theta_n)+(1-\theta_n)\frac{s^2}{p}n^4A^8\right)
\ee
\textcolor{black}{provided $\alpha-\frac{1}{2}+\delta<0$. Also note that when $\alpha-\frac{1}{2}+2\delta<0$ is needed to guarantee $\frac{s^2}{p}n^4A^8\rightarrow 0$.}
\end{proof}

\subsection{Proof of Theorem \ref{theorem:above_boundary_dense} Part (b)}

The proof can be completed as a combination of the following two lemmas.

\begin{lemma}\label{lemma:dense_above_upperbound}
There exists tests $T_1$ and $T_2$ and a sequence $\xi_p\rightarrow 0$ such that the following hold for $s=p^{1-\alpha}$ with $\alpha\leq 1/2$, $A=\sqrt{p^{\alpha-1/2+\delta}/n}$ with $\alpha-1/2+\delta>0$ \textcolor{black}{and $0<\delta<1/2$}, $p^2\ll n$ and $r=p^{\alpha-1/2+\delta}/2\log{p}$.
\be 
\P_{0}(T_1=1)+\sup_{\bbeta\in \Xi(s,A):\|\bbeta\|^2\leq \frac{\xi_p r\log{p}}{p}}\P_{\bbeta}(T_1=0)\leq e^{-\frac{p^{\frac{1}{2}+\delta}}{8}(1+o(1))},\\
\P_{0}(T_2=1)+\sup_{\bbeta\in \Xi(s,A):\|\bbeta\|^2> \frac{\xi_p r\log{p}}{p}}\P_{\bbeta}(T_2=0)\ll e^{-\frac{p^{\frac{1}{2}+\delta}}{8}(1+o(1))}.
\ee
\end{lemma}

\begin{lemma}\label{lemma:dense_above_lowerbound}
The following hold for $s=p^{1-\alpha}$ with $\alpha\leq 1/2$, $A=\sqrt{p^{\alpha-1/2+\delta}/n}$ with $\alpha-1/2+\delta>0$ \textcolor{black}{and $0<\delta<1/2$}, $p^2\ll n$ and $r=p^{\alpha-1/2+\delta}/2\log{p}$.
\be 
\lim_{p\rightarrow \infty}\frac{\log{\mathrm{Risk}(s,A)}}{\frac{(r-\alpha)^2}{4r}s\log{p}}\geq-1.
\ee
\end{lemma}
The proof of the theorem follows from these two lemmas by noting that $\frac{(r-\alpha)^2}{4r}s\log{p}=\frac{1}{8}p^{\frac{1}{2}+\delta}(1+o(1))$ for $r=p^{\alpha-1/2+\delta}/2\log{p}$.




\subsection{Proof of Theorem \ref{theorem:above_boundary_sparse}}

In this section, we establish Theorem \ref{theorem:above_boundary_sparse}. To this end, we require the following two lemmas. 

\begin{lemma}
\label{lemma:above_boundary_sparse_upper}
As $n,p \to \infty$ with $n \gg p^{7/4}\sqrt{\log p}$, we have, 
\begin{align}
    \limsup_{p \to \infty} \frac{\log \risk{(s,A)}}{s \log p } \leq - \frac{(r- \alpha)^2}{4r}. \nonumber 
\end{align}
\end{lemma}

\begin{lemma}
\label{lemma:above_boundary_sparse_lower}
As $n,p \to \infty$ with $n\gg p^2$, we have, 
\begin{align}
    \liminf_{p \to \infty} \frac{\log \risk{(s,A)}}{s \log p }  \geq - \frac{(r- \alpha)^2}{4r}. \nonumber
\end{align}
\end{lemma}

\noindent
The proof of Theorem \ref{theorem:above_boundary_sparse} follows immediately upon combining Lemma \ref{lemma:above_boundary_sparse_upper} and Lemma \ref{lemma:above_boundary_sparse_lower}.


It remains to establish Lemma \ref{lemma:above_boundary_sparse_upper} and Lemma \ref{lemma:above_boundary_sparse_lower}. To this end, we will require the following gaussian deviation bounds. We defer its proof to the Appendix. 
\begin{lemma}
\label{lemma:folded_normal_exp}
\begin{enumerate}
    \item Let $Z_1, Z_2, \cdots, Z_s$ be i.i.d.  $\mathcal{N}(0,1)$ random variables. For $\tau>0$, we have,
    \begin{align}
    \P\Big[ \sum_{i=1}^s |Z_i| > s \sqrt{2 \tau \log p} \Big]\leq 3^s \exp{[-  \tau s \log p]}. \nonumber
    \end{align}
    \item Let $Z_i \sim \mathcal{N}(\mu_i, 1)$, $1\leq i \leq s$ be independent random variables, with $|\mu_i|> \sqrt{2r \log p}$. Then we have, for $\tau < r$, and $p$ sufficiently large,  
    \begin{align}
        \P\Big[\sum_{i=1}^s |Z_i | \leq s \sqrt{2 \tau \log p} \Big]\leq 2\exp{\Big( - (\sqrt{r} - \sqrt{\tau})^2 s \log p  \Big)}. \nonumber
    \end{align}
\end{enumerate}
\end{lemma}
\noindent
Armed with Lemma \ref{lemma:folded_normal_exp}, we prove Lemma \ref{lemma:above_boundary_sparse_upper}. 
\begin{proof}[Proof of Lemma \ref{lemma:above_boundary_sparse_upper}]
Recall $\mathbf{z} = (\bX^{\mathrm{T}}\bX)^{-\frac{1}{2}} \bX^{\mathrm{T}} \mathbf{y}$ \eqref{eq:z_defn}, and note that under $\mathbb{P}_0$, $\mathbf{z} \sim \mathcal{N}(0, I_p)$. For any subset $S\subset [p]$ and vector $v \in \mathbb{R}^p$, recall $|v|_S = \sum_{i \in S} |v_i|$. Consider a test $T(\tau^*)$ which rejects $H_0$ when $\max_{|S|= s} |\mathbf{z}|_S>  s\sqrt{2 \tau^* \log p}$ for $\tau^* = (r+ \alpha)^2/(4r)$. We first control the Type I error. Using union bound, we have, 
\be
   \Pzero{\Big[\max_{|S|=s} |\mathbf{z}|_S > s \sqrt{2\tau^* \log p} \Big]}\leq {p \choose s}  \mathbb{P}\Big[ \sum_{i=1}^s |Z_i| > s \sqrt{2 \tau^* \log p}\Big], \label{eq:above_sparse_type1}
\ee 
where $Z_1, \cdots, Z_s$ are i.i.d. $\mathcal{N}(0,1)$. Lemma \ref{lemma:folded_normal_exp} implies that 
\be
    \Pzero{\Big[\max_{|S|=s} |\bz|_S > s \sqrt{2\tau^* \log p} \Big]} \leq 3^s {p \choose s} \exp{\Big( - \tau^* s \log p \Big)}.
    \label{eq:type1_final}
\ee

Next, we derive an upper bound to the Type II error. To this end, note that under $\mathbb{P}_{\bbeta}$, 
\begin{align}
    \mathbf{z} = \Big( \frac{\bX^{{\sf T}}\bX}{n} \Big)^{\frac{1}{2}} \sqrt{n} \bbeta + \tilde{\varepsilon}, \nonumber 
\end{align}
where $\tilde{\varepsilon}\sim \mathcal{N}_p(0,I_p)$. Let $S(\bbeta) =\mathrm{supp}(\bbeta)$, and note that 
\be
    \Pbeta{\Big[\max_{|S|= s} |\bz|_S < s \sqrt{2 \tau^* \log p} \Big] } \leq \Pbeta{\Big[ |\bz|_{S(\bbeta)} < s \sqrt{2 \tau^* \log p}  \Big]}.\label{eq:type2_bound}
\ee 
\noindent
Recall the event $\mathcal{G}_2$ introduced in Lemma \ref{lem:good_event}. Thus we have, 
\begin{align}
    |\mathbb{E}_{\bbeta}[z_i |\bX]- \sqrt{n}\beta_i| \leq C \sqrt{p} \|\bbeta\|_2. \nonumber 
\end{align}
Thus as long as $\|\bbeta\|_2 = o\Big(\sqrt{\frac{\log p}{p}}\Big)$, we have, 
\be
    \Pbeta{\Big[ |\bz|_{S(\bbeta)} < s \sqrt{2 \tau^* \log p}  \Big]} &\leq \Pbeta{\Big[ \sum_{i \in S(\bbeta)}|\sqrt{n}\beta_i + \tilde{\varepsilon}_i | < s \sqrt{2 \tau^* \log p}\,\, (1 +o(1))   \Big]} + 2 \exp{(-cp)} \nonumber \\
 &\leq \P\Big[ \sum_{i=1}^{s} |\sqrt{2r \log p} + Z_i| < s \sqrt{2 \tau^* \log p} \Big]  + 2 \exp{(-cp)}, \nonumber 
\ee
where $Z_1, \cdots, Z_s$ are i.i.d. $\mathcal{N}(0,1)$. Finally, Lemma \ref{lemma:folded_normal_exp} and \eqref{eq:type2_bound} implies that 
\begin{align}
    \Pbeta{\Big[\max_{|S|= s} |\mathbf{z}|_S < s \sqrt{2 \tau^* \log p} \Big]} \leq 2 \exp(- (\sqrt{r} - \sqrt{\tau^*})^2 s\log p) + 2 \exp{(-cp)}. \label{eq:type2_final}
\end{align}
Finally, combining \eqref{eq:type1_final} and \eqref{eq:type2_final} yields the required upper bound in case $\|\bbeta\|_2 = o\Big(\sqrt{\frac{\log p}{p}}\Big)$ i.e. for any sequence $\xi_p\rightarrow 0$
\be 
\ & \P_{0}(T(\tau^*)=1)+\sup_{\bbeta\in \Xi(s,A): \|\bbeta\|^2\leq \xi_p\frac{\log{p}}{p}}\P_{\bbeta}(T(\tau^*)=0)\\
&\leq 3^s {p \choose s} \exp{\Big( - \tau^* s \log p \Big)}+2 \exp(- (\sqrt{r} - \sqrt{\tau^*})^2 s\log p) + 2 \exp{(-cp)}.
\ee

\begin{lemma}\label{lemma:sparse_above_boundary_bonferroni}
Assume $p^{7/4}\sqrt{\log{p}}/n\rightarrow 0$. Then for any constant $C>0$, there exists a sequence $\xi_p\rightarrow 0$ and a sequence of tests $T_{p}$ such that
\be 
\P_{0}(T_p=1)+\sup_{\bbeta\in \Xi(s,A): \|\bbeta\|^2\geq \xi_p\frac{\log{p}}{p}}\P_{\bbeta}(T_p=0)\leq e^{-Cs\log{p}}.
\ee
\end{lemma}

The requisite upper bound can now be completed by appealing to  Lemma \ref{lemma:sparse_above_boundary_bonferroni} with $C>(\sqrt{r}-\sqrt{\tau^*})^2$ (by considering a Bonferroni correction between $T(\tau^*)$ and $T_{p}$).

\end{proof}
\noindent
Next we prove Lemma \ref{lemma:above_boundary_sparse_lower}. 

\begin{proof}[Proof of Lemma \ref{lemma:above_boundary_sparse_lower}] 
We define the conditional likelihood 
\begin{align}
    L(\bbeta) = \exp{\Big( \langle \by, \bX\bbeta \rangle - \frac{1}{2} \| \bX \bbeta \|_2^2 \Big) }.  \nonumber 
\end{align}
Consider a prior $\pi$ on the parameter space $\tilde\Xi(s,A)$ which chooses $s$ locations uniformly at random and sets these coordinates to be $A$. All other coordinates are set at zero. We define the integrated likelihood ratio $L_{\pi} = \E_{\bbeta\sim \pi} [L(\bbeta)]$, and  note the classical lower bound \cite{ingster2012nonparametric} 
\begin{align}
    \mathrm{Risk}(s,A) \geq \Pzero{[L_{\pi} >1]} + \E_{\bbeta \sim \pi }[\Pbeta{[ L_{\pi} \leq 1]}]. \nonumber  
\end{align}
Recall the event $\mathcal{G}_1$ introduced in Lemma \ref{lemma:rip}. Let $S(\bbeta)= \mathrm{supp}(\bbeta)$, and set $e_{S(\bbeta)}$ to be the indicator vector of this set, i.e., $[e_{S(\bbeta)}]_j = \mathbf{1}(j \in S(\bbeta))$. Armed with this notation, we can lower bound the minimax risk as 
\begin{align}
     \mathrm{Risk}(s,A) \geq \Pzero{[L_{\pi} >1, \mathcal{G}_1]}\geq \Pzero{\Big(\E_{\bbeta \sim \pi } \Big[ L(\bbeta) \mathbf{1}\Big( e_{S(\bbeta)}^{{\sf T }} \frac{\bX^{{ \sf T}} \mathbf{y}}{\sqrt{n}} > s \sqrt{2\tau^*\log p}\Big) \Big]>1  , \mathcal{G}_1 \Big)},\nonumber 
\end{align}
where $\tau^* = (r+\alpha)^2/4r$. Next, we define, 
\begin{align}
    \tilde{L}(\bbeta) = \exp{\Big( \langle \bbeta, \bX^{{\sf T}} \mathbf{y}\rangle - \frac{n}{2} \|\bbeta\|_2^2 \Big)}. \nonumber 
\end{align}
We observe that on the event $\mathcal{G}_1$, 
\begin{align*}
    \Big| \bbeta^{{\sf T }} \Big( \frac{\bX^{{\sf T }} \bX}{n} - I  \Big) \bbeta \Big|\leq C \frac{s \log (p/s)}{n} \|\bbeta\|_2^2 \leq C' \frac{s^2 (\log p)^2}{n^2}. \nonumber 
\end{align*}
Therefore, on the event $\mathcal{G}_1$, $L(\bbeta) = \tilde{L} (\bbeta) (1+ o(1))$, which yields the lower bound 
\be
    \mathrm{Risk}(s,A) &\geq \Pzero{\Big( \mathbb{E}_{\bbeta \sim \pi} \Big[ \tilde{L}(\bbeta) \mathbf{1}\Big( e_{S(\bbeta)}^{{\sf T }} \frac{\bX^{{ \sf T}} \mathbf{y}}{\sqrt{n}} > s \sqrt{2\tau^*\log p}\Big) \Big]>(1+o(1))   \, ,\,  \mathcal{G}_1  \Big)}\nonumber\\
    &\geq \Pzero{\Big( \sum_{|S|= s} \mathbf{1}\Big( e_S^{{\sf T}}\frac{\bX^{{ \sf T}} \mathbf{y}}{\sqrt{n}} > s \sqrt{2 \tau^* \log p}  \Big) \geq 1, \mathcal{G}_1 \Big)}. \label{eq:int_sparse_above_lowerbd}
\ee
Next, we observe that given $\bX$, under $\Pzero$,
\begin{align}
   \frac{ \bX^{{\sf T}}\mathbf{y}}{\sqrt{n}} \stackrel{d}{=} \Big( \frac{\bX^{{\sf T}} \bX}{n}\Big)^{\frac{1}{2}}\bZ, \nonumber 
\end{align}
where $\bZ \sim \mathcal{N}(0, I_p)$ is independent of $\bX$. We have, for any $S \subset [p]$ with $|S|=s$,  
\begin{align}
    \Big| e_S^{{\sf T}}\Big( \Big(\frac{\bX^{{\sf T}} \bX}{n}\Big)^{\frac{1}{2}} - I \Big) \bZ \Big|  \leq \sqrt{s} \Big\|\Big( \frac{\bX^{{\sf T}} \bX}{n}\Big)^{\frac{1}{2}} - I \Big\| \, \| \bZ\|_2. \nonumber  
\end{align}
For $C_1>0$ let  $\mathcal{E}_1$ denote the event that $\{\|\bZ\|_2 \leq C_1 \sqrt{p}\}$ and recall the event $\mathcal{G}_2$ introduced in Lemma \ref{lem:good_event}.
%
Standard tail bounds imply that for $C_1>0$ sufficiently large, 
\begin{align}
    \P(\mathcal{E}_1^c ) \leq \exp{(-c_1 p)}, \nonumber
\end{align}
for some constant $c_1 >0$. 
%
The bounds derived above imply that on the event $\mathcal{E}_1 \cap \mathcal{E}_2$, 
\begin{align}
    \Big| e_S^{{\sf T}}\Big( \Big(\frac{\bX^{{\sf T}} \bX}{n}\Big)^{\frac{1}{2}} - I \Big) \bZ \Big| \leq C \sqrt{s}\, \sqrt{\frac{p}{n}} \,\sqrt{p} =o(s)\nonumber 
\end{align}
whenever $n \gg p^2$. This in turn, implies 
\begin{align}
    &\Pzero{\Big( \sum_{|S|= s} \mathbf{1}\Big( e_S^{{\sf T}}\frac{\bX^{{ \sf T}} \mathbf{y}}{\sqrt{n}} > s \sqrt{2 \tau^* \log p}  \Big) \geq 1, \mathcal{G}_1 \Big)} \geq \P\Big(\sum_{|S|= s} \mathbf{1} \Big( e_S^{{ \sf T}} \Big(\frac{\bX^{{\sf T}} \bX}{n}\Big)^{\frac{1}{2}}\bZ > s \sqrt{2\tau^* \log p} \Big) \geq 1, \mathcal{G}_1, \mathcal{E}_1 \cap \mathcal{G}_2 \Big)\nonumber \\
    &\geq \P\Big(\sum_{|S|= s} \mathbf{1} \Big( e_S^{{ \sf T}}\bZ > s \sqrt{2\tau^* \log p} (1+ o(1)) \Big) \geq 1, \mathcal{G}_1, \mathcal{E}_1 \cap \mathcal{G}_2 \Big)\nonumber \\
    &\geq \P\Big(\sum_{|S|= s} \mathbf{1} \Big( e_S^{{ \sf T}}\bZ > s \sqrt{2\tau^* \log p} (1+ o(1)) \Big) \geq 1 \Big) - 2\exp{(-cn)} - \exp{(-c_2 p)} - \exp{(-c_3 p)} \nonumber \\
    &\geq \P\Big( \sum_{i=1}^{p} \mathbf{1}(Z_i > \sqrt{2\tau^* \log p} (1+o(1)) = s  \Big) - 4 \exp(-c p). \nonumber 
\end{align}
Finally, we note that $\sum_{i=1}^{p} \mathbf{1}(Z_i > \sqrt{2\tau^* \log p}(1+o(1)) ) \sim \mathrm{Bin}(p, \bar{\Phi}(\sqrt{2\tau^* \log p}(1+o(1)))$. The proof is now complete using direct computation.

\end{proof}

\subsection{Proof of Theorem \ref{thm:boundary}}
Recall $\mathbf{z}$ from \eqref{eq:z_defn},
and consider a test $T_p$ which rejects the null whenever $\max_i |z_i| > \sqrt{2\log p}$. Under $\mathbb{P}_0$, $\mathbf{z} \sim \mathcal{N}(0, I)$, and thus, using a Mills Ratio bound \citep{williams1991probability} 
\begin{align}
    \mathbb{P}_0[T_p =1]= \mathbb{P}_0[ \max_i |z_i| > \sqrt{2\log p}] \leq 2p \bar{\Phi}(\sqrt{2\log p}) \to 0 \nonumber 
\end{align}
as $n,p \to \infty$. Next, we turn to the Type II error. 
Recall the event $\mathcal{G}_2$, introduced in Lemma \ref{lem:good_event}. We have, 
\begin{align}
    \mathbb{P}_{\bbeta}[\max_i |z_i| < \sqrt{2\log p}] \leq \mathbb{P}_{\bbeta}\Big[\max_{i \in \mathrm{supp}(\bbeta)} |z_i| < \sqrt{2\log p}, \mathcal{G}_2\Big] + 2 \exp(-cp). \nonumber 
\end{align}
We proceed exactly as in the proof of the upper bound in Theorem \ref{theorem:below_boundary_sparse}, and conclude that 
\begin{align}
    \mathbb{P}_{\bbeta}[\max_i |z_i| < \sqrt{2\log p}] \leq \mathbb{P}_{\bbeta}\Big[\max_{i \in \mathrm{supp}(\bbeta)} | \sqrt{n} \beta_i + O\Big(\sqrt{\frac{p\log p}{n}} \Big) + \eta_i | \leq \sqrt{2\log p} \Big] + 2 \exp(-cp),  \nonumber 
\end{align}
where we crucially use  $p\log p =o(n)$. In the display above, 
 $\eta_i \sim \mathcal{N}(0,1)$ are iid. This immediately implies 
\begin{align}
    \limsup \mathrm{Risk}(s,A)\leq \limsup \mathrm{Risk}(T_p, s, A) = \Big(\frac{1}{2} \Big)^s. \nonumber 
\end{align}
Next, we turn to the lower bound on the minimax risk. To this end, let $\pi$ be a prior which selects $s$ locations at random, and sets the selected $\beta_i$ as $\sqrt{2\log p}$. Then we have, 
\begin{align}
    \mathrm{Risk}(s,A) \geq \mathbb{P}_0[L_{\pi} >1] + \mathbb{E}_{\bbeta \sim \pi} \Big[\mathbb{P}_{\bbeta}[L_{\pi} < 1 ] \Big], \nonumber 
\end{align}
where 
\begin{align}
    L(\bbeta) = \exp\Big(\langle \mathbf{y}, \bX \bbeta \rangle - \frac{1}{2} \| \bX \bbeta \|_2^2 \Big),\,\,\,\, L_{\pi} = \mathbb{E}_{\bbeta \sim \pi}[L(\bbeta)].  \nonumber
\end{align}
The proof will be completed by invoking the following lemmas. 

\begin{lemma}
\label{lem:boundary_null}
As $n,p \to \infty$ with $n\gg p(\log p)^2$, $\mathbb{P}_0[L_{\pi} >1 ] \to 0$. 
\end{lemma}

\begin{lemma}
\label{lem:boundary_alternative}
As $n,p \to \infty$, with $n \gg p(\log p)^2$, $\mathbb{E}_{\bbeta}[\mathbb{P}_{\bbeta} [L_{\pi} \leq 1 ] ] \to \Big(\frac{1}{2}\Big)^s$. 
\end{lemma}
\noindent 
The proofs of these lemmas are technically involved, and are deferred to the appendix. This concludes the proof.

\subsection{Proof of Theorem \ref{theorem:hc_vs_scan}}
Note that under $H_0$ we have $\mathbf{z}\sim N(\mathbf{0},I_p)$ and therefore $\sum_{j=1}^p\mathbf{1}(|z_j|>\tau)\sim \mathrm{Bin}\left(p,2\bar{\Phi}(\tau)\right)$. Therefore
\be 
\P_{0}\left(T(\tau)\geq t\right)\geq \P\left(\mathrm{Bin}\left(p,2\bar{\Phi}(\tau)\right)=t\right)={p \choose t}(2\bar{\Phi}(\tau))^{t}(1-2\bar{\Phi}(\tau))^{p-t}
\ee
Now note that $(p-t)\bar{\Phi}(\sqrt{2\log{p}})\rightarrow 0$ for any $t\geq 0$ and consequently $(1-2\bar{\Phi}(\sqrt{2\log{p}}))^{p-t}=1+o(1)$. Further, ${p \choose t} \geq (p/t)^t$; thus, setting  $\gamma_p=2\bar{\Phi}(\sqrt{2\log{p}})$, we have 
\be 
\P_{0}\left(T(\tau)\geq t\right)\geq (1+o(1))\exp\left(t\log(p\gamma_p/t)\right)=(1+o(1))\exp\left(-t\log(t/p\gamma_p)\right)
\ee
Now note that $t\log(t/p\gamma_p)=\Theta\left(t\max\{\log t,\log\log p\}\right)\leq t\log{p}$ as $t\leq p$. So in order to achieve optimal risk one necessarily has $t\geq cs$ for some $c>0$. The rest of the proof shows, that for any $t\geq cs$ with some $c>0$, given any $\delta>0$, one has for sufficiently large $n,p$ that  $\P_{\bbeta}(T(\sqrt{2\log{p}})\leq
t)\geq 1-\delta$.  \textcolor{black}{Also, since the risk function involves a supremum over all $\bbeta\in \Xi(s,A)$, it is enough to prove the results for $\|\bbeta\|=O(\sqrt{s\log{p}/n})$ -- which we shall assume for our subsequent analyses.}

Now let $S$ be such that under $H_1$ one has $\mathrm{supp}(\bbeta)=S$ with $|S|=s$. Note that we can write as usual that $\mathbf{z}=(\bX^T\bX)^{1/2}\bbeta+\mathbf{\eta}$ where $\mathbf{\eta}\sim N(0,I_p)$.  Hence, 
\begin{align}
    z_i = \sqrt{n} \beta_i +  e_i^{T}\Big( \Big( \frac{\bX^{T} \bX}{n}  \Big)^{\frac{1}{2}}-I \Big)\sqrt{n} \bbeta  + \eta_i. \nonumber 
\end{align}
Recall the event $\mathcal{G}_2$, introduced in Lemma \ref{lem:good_event}. On the event $\mathcal{G}_2$, we have, 
\begin{align}
    | e_i^{T}\Big( \Big( \frac{\bX^{T} \bX}{n}  \Big)^{\frac{1}{2}}-I \Big)\sqrt{n} \bbeta | \leq \sqrt{n}\| \Big(\frac{\bX^T \bX}{n}\Big)^{\frac{1}{2}} - I\|  \|\bbeta \|_2\lesssim \sqrt{\frac{sp \log p}{n} }, \nonumber 
\end{align}
whenever $\|\bbeta\|=O(\sqrt{s\log{p}/n})$. Therefore, for any such $\bbeta$ we have
\be 
\ & \P_{\bbeta}(T(\sqrt{2\log{p}})\leq
t)\\
&\geq \E_{\bbeta}\left[\mathbf{1}_{\mathcal{G}_2}\P_{\bbeta}(\sum_{j\in S}\mathbf{1}(|z_j|>\sqrt{2\log{p}})\leq t|\bX)\P_{\bbeta}(\sum_{j\in S^c}\mathbf{1}(|z_j|>\sqrt{2\log{p}})=0|\bX)\right]
\ee

Now note that 
\be 
\mathbf{1}_{\mathcal{G}_2}\P_{\bbeta}\Big(\sum_{j\in S^c}\mathbf{1}(|z_j|>\sqrt{2\log{p}})=0|\bX \Big)&=\mathbf{1}_{\mathcal{G}_2}\P_{\bbeta}\left(\max_{j\in S^c}|\eta_j+\delta_{j,p}|\leq \sqrt{2\log{p}}\right),
\ee
where on $\mathcal{G}_2$ we have $\delta_{j,p}\leq C'\sqrt{sp\log{p}/n}$ for some large absolute constant $C'>0$. Consequently, \textcolor{black}{since $p^{3/2}\ll n$ and $s$ is polynomially smaller in order than $\sqrt{p}$ for $\alpha>\frac{1}{2}$} it is easy to check that for any $\epsilon>0$ one has for large enough $n,p$ that
\be
\mathbf{1}_{\mathcal{G}_2}\P_{\bbeta}\Big(\sum_{j\in S^c}\mathbf{1}(|z_j|>\sqrt{2\log{p}})=0|\bX \Big)\geq 1-\epsilon.
\ee

Also for $t\geq cs$ we have by Markov's Inequality
\be 
\mathbf{1}_{\mathcal{G}_2}\P_{\bbeta}\Big(\sum_{j\in S}\mathbf{1}(|z_j|>\sqrt{2\log{p}})\leq t|\bX \Big)&\geq \mathbf{1}_{\mathcal{G}_2}\left[1-\frac{\sum_{j\in S}\P_{\bbeta}\left(|z_j|>\sqrt{2\log{p}}\right|\bX)}{cs}\right]
\ee
Now for any $j\in S$
\be 
\mathbf{1}_{\mathcal{G}_2}\P_{\bbeta}\left(|z_j|>\sqrt{2\log{p}}|\bX \right )&\leq 2\P(\eta_j>\sqrt{2\log{p}}(1-\sqrt{r})-\delta_{j,p})
\ee
where on $\mathcal{G}_2$ we have $\delta_{j,p}\leq C'\sqrt{sp\log{p}/n}$ for some large absolute constant $C'>0$. Consequently, \textcolor{black}{if $p^{3/2}\ll n$}  it is easy to check that for any $\epsilon>0$ one has for large enough $n,p$ that
\be 
\mathbf{1}_{\mathcal{G}_2}\frac{\sum_{j\in S}\P_{\bbeta}\left(|z_j|>\sqrt{2\log{p}}\right|\bX)}{cs}&\geq \mathbf{1}_{\mathcal{G}_2}(1-\epsilon s/cs).
\ee
Therefore for $t\geq cs$ and $\|\bbeta\|=O(\sqrt{s\log{p}/n})$ we have for sufficiently large $n,p$ that
\be 
\P_{\bbeta}\left(T(\sqrt{2\log{p}}\leq t)\right)&\geq (1-\epsilon/c)(1-\epsilon)\P(\mathcal{G}_2)\\
&\geq (1-\epsilon/c)(1-\epsilon)(1-2e^{-Cp}),
\ee
for some absolute constant $C>0$. The last display can be made larger than any $1-\delta$ by choosing $\epsilon>0$ small enough and $n,p$ large enough.

\subsection{Proof of Theorem \ref{theorem:max_vs_scan}}
Recall that under $H_0$, we have $\mathbf{z}\sim N(\mathbf{0},I_p)$ and therefore, using \cite[Theorem 1]{deo1972some}, we have, for any $x>0$ and $t_p=\sqrt{2\log{p}}(1-\frac{\log{\log{p}}+4\pi-4}{8\log{p}})$ one has 

$$\P_{0}\left(\max_{j\in [p]}|z|_j\leq t_p+\frac{x}{\sqrt{2\log{p}}}\right)\rightarrow e^{-e^{-x}}. $$
Therefore,  in order to have Type I error of a Max Type test that rejects for large values of $\max_{j\in [p]}|z|_j$ one needs a cut-off of the form $t_p(x)=\sqrt{2\log{p}}(1-\frac{\log{\log{p}}+4\pi-4}{8\log{p}}+\frac{x}{2\log{p}})$ with $x\rightarrow \infty$. 
Moreover, by standard Mill's ratio bound we know, for $x\rightarrow \infty$
\be
\P_{0}\left(\max_{j\in [p]}|z|_j>t_p(x)\right)&\geq 1-\left\{1-\frac{\exp\left(-\log{p}\left[1-\frac{\log{\log{p}}+4\pi-4}{8\log{p}}+\frac{x}{2\log{p}}\right]^2\right)}{\sqrt{2\pi}(t_p(x)+t_p^{-1}(x))}\right\}^p.
\ee
\textcolor{black}{It is easy to see from the last display that if we need to have  $\P_{0}\left(\max_{j\in [p]}|z|_j>t_p(x)\right)\leq \exp(-Cs\log{p})$ for some $C>0$ then we need $x\geq c'\sqrt{s}\log{p}$ for some $c'>0$ . } However, for any $x\geq c'\sqrt{s}\log{p}$ there exists $c>0$ such that $t_p(x)\geq \sqrt{2cs\log{p}}$.
\textcolor{black}{Since the risk function involves a supremum over all $\bbeta\in \Xi(s,A)$, it is enough to prove the result for $\|\bbeta\|=O(\sqrt{s\log{p}/n})$ and $\|\bbeta\|_{\infty}=O(\sqrt{\log{p}/n})$ -- which we shall assume for our subsequent analyses.}

Now let $S$ be such that under $H_1$ one has $\mathrm{supp}(\bbeta)=S$ with $|S|=s$. Note that we can write as usual that $\mathbf{z}=(\bX^T\bX)^{1/2}\bbeta+\mathbf{\eta}$ where $\mathbf{\eta}\sim N(0,I_p)$.  Hence, 
\begin{align}
    z_i = \sqrt{n} \beta_i +  e_i^{T}\Big( \Big( \frac{\bX^{T} \bX}{n}  \Big)^{\frac{1}{2}}-I \Big)\sqrt{n} \bbeta  + \eta_i. \nonumber 
\end{align}
Recall the event $\mathcal{G}_2$, introduced in Lemma \ref{lem:good_event}. On the event $\mathcal{G}_2$, we have, 
\begin{align}
    | e_i^{T}\Big( \Big( \frac{\bX^{T} \bX}{n}  \Big)^{\frac{1}{2}}-I \Big)\sqrt{n} \bbeta | \leq \sqrt{n}\| \Big(\frac{\bX^T \bX}{n}\Big)^{\frac{1}{2}} - I\|  \|\bbeta \|_2\lesssim \sqrt{\frac{sp \log p}{n} }, \nonumber 
\end{align}
whenever $\|\bbeta\|=O(\sqrt{s\log{p}/n})$.
\be 
\P_{\bbeta}\left(\max_{j\in [p]}|z|_j>t_p(x)\right)&=\P_{\bbeta}(\max_{j\in [p]}|\eta_j+\beta_j+\delta_{j,p}|>t_p(x))
\ee
where on $\mathcal{G}_2$ we have $\delta_{j,p}\leq C'\sqrt{sp\log{p}/n}$ for some large absolute constant $C'>0$. Now $\sqrt{sp\log{p}/n}\ll \sqrt{s\log{p}}$ whenever  \textcolor{black}{$p\ll n$}. 
Therefore 
\be 
\P_{\bbeta}\left(\max_{j\in [p]}|z|_j>t_p(x)\right)&\leq \P_{\bbeta}\left(\max_{j\in [p]}|z|_j>\sqrt{2cs\log{p}}\right)\\
&\leq \P_{\bbeta}(\max_{j\in [p]}|\eta_j+\beta_j+\delta_{j,p}|>\sqrt{2cs\log{p}},\mathcal{G}_2)+\P_{\bbeta}(\mathcal{G}_2^c)\\
&\leq \P\left(\max_{j\in [p]}|\eta_j|>\sqrt{2cs\log{p}}(1-o(1))\right)+2e^{-c^*n}\\
&\leq e^{-cs\log{p}(1-o(1))}+2e^{-c^*n}
\ee
for some constant $c^*>0$. Above, the second to last inequality uses Lemma \ref{lem:good_event} and the fact that $\|\bbeta\|_{\infty}=O(\sqrt{\log{p}/n})$. This proves the desired result.


\noindent 

\textbf{Acknowledgments: }  RM thanks Alexandre Tsybakov for a stimulating discussion, which motivated this research. The authors would like to thank Sumit Mukherjee for many helpful discussions during the preliminary part of this project. 

\bibliographystyle{imsart-nameyear}
\bibliography{biblio_testing_rates}

\section{Proofs}
\begin{proof}[Proof of Lemma \ref{lemma:sparse_below_boundary_bonferroni}]
Let $\tilde{\mathbf{z}}=(\bX^T\bX)^{-1}\bX^T\by$ and consider the test $T_{3p}(\tau)$ defined as
\be 
T_{3p}(\tau):=\mathbf{1}\left(\|\tilde{\mathbf{z}}\|_{\infty}>\sqrt{2\tau\log{p}}\right).
\ee
To analyze $T_{3p}(\tau)$ we define the good event
\be 
\mathcal{G}(\delta_p)=\left\{\max_{j=1}^p|\hat{\omega}_j-1|\leq \delta_p\right\},
\ee
where $\delta_p>0$ is some sequence to be decided later and $\hat{\omega}_j=(\bX^T\bX/n)^{-1}_{jj}$. Note that for $\delta_p=C\sqrt{p/n}$ with large enough $C>0$ (depending only on the subgaussian norm of the rows of $\bX$) one has (using Lemma \ref{lem:good_event})
\be 
\P(\mathcal{G}(\delta_p)^c)\leq 2e^{-nc}.
\ee
Subsequently by union bound and Mill's Ratio bound we have for this choice of $\delta_p$
\be 
\P_{0}(T_{3p}(\tau)=1)&\leq 2p^{1-\tau/(1+\delta_p)}/\sqrt{2(1-\delta_p)\tau\log{p}}+2e^{-cn}.
\ee
Subsequently, we shall choose $\tau\geq 1+\delta_p$ to achieve Type I error of the test converging to $0$. 

Now take any $j^*\in \mathrm{argmax}\{j:|\beta_j|\}$. Then $|\beta_{j^*}|\geq \sqrt{C^*\log{p}/n}$ where $C^*$ will be chosen large enough depending on the desired $C>0$ in the statement of the lemma. Then 
\be 
\P_{0}(T_{3p}(\tau)=0)&\leq \P_{\bbeta}\left(|\tilde{z_{j^*}}|\leq \sqrt{2\tau\log{p}},\mathcal{G}(\delta_p)\right)+2e^{-cn}.
\ee
Note that $\tilde{z_{j^*}}=\beta_{j^*}+\hat{\omega}_j^{1/2}\eta_j$ where $\eta_j\sim \mathcal{N}(0,1)$. Therefore by choosing $C^*\geq \tau$
\be 
\P_{\bbeta}\left(|\tilde{z_{j^*}}|\leq \sqrt{2\tau\log{p}},\mathcal{G}(\delta_p)\right)&=\E_{\bbeta}\left[\mathbf{1}_{\mathcal{G}(\delta_p)}\P_{\bbeta}\left(\eta_j\leq -\frac{\sqrt{2\log{p}}(\sqrt{C^*}-\sqrt{\tau})}{\hat{\omega}_{j^*}^{1/2}}\right)\right]\\
&\leq \E_{\bbeta}\left[\mathbf{1}_{\mathcal{G}(\delta_p)}\bar{\Phi}\left(\frac{\sqrt{2\log{p}}(\sqrt{C^*}-\sqrt{\tau})}{\sqrt{1+\delta_p}}\right)\right]\\
&\leq \exp\left(-\log{p}(\sqrt{C^*}-\sqrt{\tau})^2/(1+\delta_p)\right).
\ee
Consequently
\be 
\ & \P_{0}\left(T_{3p}(\tau)=1\right)+\sup_{\bbeta\in \Xi(s,A): \|\bbeta\|_{\infty} \geq \sqrt{C^*\log{p}/n}}\P_{\bbeta}\left(T_{3p}(\tau)=0\right)\\
&\leq 4e^{-cn}+2p^{-1+\tau/(1+\delta_p)}/\sqrt{2(1-\delta_p)\tau\log{p}}+\exp\left(-\log{p}(\sqrt{C^*}-\sqrt{\tau})^2/(1+\delta_p)\right).
\ee
The proof follows by choosing $\tau$ large enough and subsequently choosing $C^*>4\tau$.
\end{proof}

\begin{proof}[Proof of Lemma \ref{lemma:sparsepoly_below_hc}]
We consider the HC test $T_{2p}(\tau_p)$ which rejects $H_0$ whenever $\sum_{i=1}^{p} \mathbf{1}(|z_i| > 2A) > 2p \bar{\Phi}(2A) + \tau_p \sqrt{2 p \bar{\Phi}(2A) (1-2\bar{\Phi}(2A))}$. The sequence $\tau_p$ will be suitably chosen as a part of this analysis. For convenience of notation, we set $\pi = 2\bar{\Phi}(2A)$. 
Consider first the Type I error of this test. Indeed, 
\be
\P_0 [ T_p ( \tau_p ) =1 ] = \P [ Z_1 + Z_2 > p \pi + \tau_p \sqrt{p \pi (1-\pi)} ], 
\ee
where $Z_1 \sim \textrm{Bin}(p-s , \pi)$, $Z_2 \sim \textrm{Bin}(s,\pi )$, and the random variables are independent. Next, we look at the Type II error of this test. 
Recall the good event $\mathcal{G}_2$ introduced in Lemma \ref{lem:good_event}. 
Observe that 
\begin{align}
    \Pbeta{}[T_p(\tau_p)=0]\leq \Pbeta{}[T_p(\tau_p)=0, \mathcal{G}_2] + 2\exp(-cp). \nonumber
\end{align}
Note that given $X$, $\mathbf{1}(|z_i|>2A) \succeq \mathbf{1}(|\eta| > 2A)$, where $\eta \sim \mathcal{N}(0,1)$. Further, given $X$, on the event $\mathcal{G}_2$, we have, 
\begin{align}
    | e_i^{T}\Big( \Big( \frac{\bX^{T} \bX}{n}  \Big)^{\frac{1}{2}}-I \Big)\sqrt{n} \bbeta | \leq \sqrt{n}\| \Big(\frac{\bX^T \bX}{n}\Big)^{\frac{1}{2}} - I\|  \|\bbeta \|_2\lesssim \sqrt{\frac{sp \log p}{n} }. \nonumber 
\end{align}
Thus we have, 
\begin{align}
    \Pbeta{}[T_p(\tau_p)=0, \mathcal{G}_2]\leq \P[Z_1 + Z_2' \leq p \pi + \tau_p \sqrt{p \pi (1-\pi)}], \nonumber 
\end{align}
where $Z_1\sim \mathrm{Bin}(p-s, \pi)$, $Z_2' \sim \mathrm{Bin}(s, \pi')$, 
\begin{align}
    \pi'= \P\Big[|A - O\Big(\sqrt{\frac{sp\log p}{n}}\Big) + \eta| > 2A \Big], \nonumber 
\end{align}
and $\eta \sim \mathcal{N}(0,1)$. Moreover, $Z_1,Z_2'$ are independent random variables. In the subsequent proof, we will often couple the Binomial random variables $Z_2, Z_2'$ optimally. 

First, observe that $s\pi' = p^{1 - \alpha - r + o(1)}$. Consider first the case $1 - \alpha - r >0$. In this case, using the optimal coupling between $Z_2, Z_2'$, we have,  for any $\gamma>0$,  
\begin{align}
&\P [ Z_1 + Z_2' \leq   p \pi  + \tau_p \sqrt{p \pi (1-\pi)} ] \nonumber \\
&\leq \P[ Z_1 + Z_2' \leq  p \pi + \tau_p \sqrt{ p \pi (1-\pi) } , Z_2' - Z_2  \geq  s(\pi' - \pi) - \sqrt{ s (\pi' - \pi) } (\log p)^ \gamma]  \nonumber \\
&+ \P[ Z_2' - Z_2 \leq  s(\pi' - \pi) - \sqrt{ s (\pi' - \pi) } (\log p)^ {\gamma}] \nonumber \\ 
&\leq \P [Z_1 + Z_2 \leq  p\pi  - s(\pi'- \pi) +\sqrt{ s (\pi' - \pi) } (\log p)^ \gamma  + \tau_p \sqrt{ p \pi (1-\pi)} ] + 2 \exp \Big(  - c (\log p )^{\gamma} \Big)\nonumber 
\end{align}
for some universal constant $c>0$, where the last inequality follows using Chernoff bound for the $Z_2 - Z_2' \sim \textrm{Bin}(s, p' - p)$. Thus we have,
\begin{align}
\mathrm{Risk}(T_{2p}, s, A) &\leq 
\P[ Z_1 + Z_2 > p \pi  + \tau_p \sqrt{p \pi (1-\pi)} ]\nonumber \\
& + \P[ Z_1 + Z_2 \leq p \pi  - s (\pi'-\pi) ( 1+ o(1)) +\tau_p \sqrt{p \pi (1-\pi)} ]   \nonumber \\
& + \exp( - c (\log p)^{\gamma}) + 2 \exp(-cp). \nonumber\\
& = 1 - \P[ p\pi - s (\pi'-\pi) ( 1+ o(1)) + \tau_p \sqrt{p \pi (1-\pi)} \leq Z_1 + Z_2 \leq  p\pi + \tau_p \sqrt{p \pi (1-\pi) } ]\nonumber \\&+ \exp( - c (\log p)^{\gamma}) + 2 \exp(-cp).\nonumber
\end{align}
Note that $Z_1 + Z_2 \sim \textrm{Bin} (p , \pi)$ and $p \min\{ \pi, 1- \pi \} \gg \log p$. We choose $\tau_p = \log \log p$ in this case, and note that using \cite[Theorem 1.2, 1.5]{bollobas2001random}, we have, 
\be
\ & \P\Big[ p\pi - s (\pi'-\pi) ( 1+ o(1)) + \tau_p \sqrt{p \pi (1-\pi) } \leq Z_1 + Z_2 \leq  p\pi  + \tau_p \sqrt{p \pi (1-\pi)} \Big] \\
&= \frac{s(\pi' - \pi)}{\sqrt{ p \pi (1-\pi)}} \exp\Big( - \frac{\tau_p^2}{2} (1 + o(1))\Big)
\ee
The desired result follows in this case, upon plugging this estimate back in the risk estimate, and using $\tau_n = \log \log p $. 

Next, consider the case $ 1 - \alpha - r <0$. In this case, $s\pi, s\pi' \to 0$ polynomially in $p$. We find $k \geq 1$ such that $k (r - (1- \alpha)) > ( \alpha - 1/2 - r )$. Thus we have, 
\begin{align}
& \ \P[ Z_1 + Z_2 > p\pi + \tau_p \sqrt{p \pi (1-\pi)}] \nonumber\\
&\leq  \sum_{l=0}^{k} \P[ Z_2 = l] \P[ Z_1 > p\pi - l + \tau_p \sqrt{p \pi (1-\pi)} ]  + O( (s\pi)^{k} ). \nonumber\\
&\leq \P[ Z_1 > p \pi  + \tau_p \sqrt{p \pi (1-\pi) }] \nonumber\\
&+ s\pi  (1 + o(1)) \P[ p\pi - k + \tau_p  \sqrt{p \pi (1-\pi) } < Z_1 < p \pi  + \tau_p \sqrt{p \pi (1-\pi)}] + O( ( s\pi)^k ). \nonumber\\
&= \P[ Z_1 >  p\pi  + \tau_p \sqrt{p \pi (1-\pi) }]  \nonumber\\
&+ \frac{s\pi }{\sqrt{p \pi (1-\pi)}} \exp\Big( - \frac{\tau_p^2}{2} (1 + o(1)) \Big) + O((s\pi)^k), \nonumber  
\end{align}
where the last inequality follows using \cite[Theorem 1.2,1.5]{bollobas2001random}. For the Type II error, we similarly have, 
\begin{align}
&\P[ Z_1 + Z_2' \leq  p \pi  + \tau_p \sqrt{p \pi (1-\pi)} ] \leq  \sum_{l=0 }^{k} \P[ Z_2' = l] \P[ Z_1 \leq p\pi  - l + \tau_p \sqrt{p \pi (1-\pi)}] + O( (s\pi')^k ). \nonumber\\
&= \P[ Z_1 \leq p \pi   + \tau_p \sqrt{p \pi (1-\pi)}] - \frac{s\pi'}{\sqrt{s \pi (1-\pi)}} \exp\Big( - \frac{\tau_p^2}{2} (1+ o(1))\Big) + O((s\pi')^k). \nonumber
\end{align}
The proof follows on combining the Type I and Type II estimates, by choosing $\tau_p = O(\log \log p)$. Finally, it remains to analyze the case $ 1 - \alpha -r =0$. In this case, $s\pi' \to 0$, but only logarithmically. The proof goes through along the lines of the case $1 - \alpha -r <0$, by choosing $k= O( \log p)$. We omit this proof to avoid repetition. 

\end{proof}

\begin{proof}[Proof of Lemma \ref{lemma:dense_above_upperbound}]
The proof is similar to the proof of Lemma \ref{lemma:above_boundary_sparse_upper}. We provide the details again for keeping track of the main changes.

First consider a test $T$ that rejects which rejects when $\max_{|S|=s}|\bz|_S>s\sqrt{2\tau^*\log{p}}$ with $\tau^*=\frac{(r+\alpha)^2}{4r}$ and $\sqrt{r}=A/\sqrt{2\log{p}}=\sqrt{p^{\alpha-\frac{1}{2}+\delta}/2\log{p}}$. This implies that $r\rightarrow \infty$ as $p\rightarrow \infty$ and $\tau^*\sim r/4$. Subsequently, by union bound and Lemma \ref{lemma:folded_normal_exp} we have the Type I error of the test $T$ as
\be 
\P_{0}(T=1)\leq {p \choose s}\times 3^s\times \exp(-\tau^*s\log{p})=\exp\left(-\frac{(r-\alpha)^2}{4r}s\log{p}(1+o(1))\right)
\ee

Next, we derive an upper bound to the Type II error. To this end, note that under $\mathbb{P}_{\bbeta}$, 
\begin{align}
    \mathbf{z} = \Big( \frac{\bX^{{\sf T}}\bX}{n} \Big)^{\frac{1}{2}} \sqrt{n} \bbeta + \tilde{\varepsilon}, \nonumber 
\end{align}
where $\tilde{\varepsilon}\sim \mathcal{N}_p(0,I_p)$. Let $S(\bbeta) =\mathrm{supp}(\bbeta)$, and note that 
\be
    \Pbeta{\Big[\max_{|S|= s} |\bz|_S < s \sqrt{2 \tau^* \log p} \Big] } \leq \Pbeta{\Big[ |\bz|_{S(\bbeta)} < s \sqrt{2 \tau^* \log p}  \Big]}.\label{eq:type2_bound}
\ee 
\noindent

Recall the event $\mathcal{G}_2$ introduced in Lemma \ref{lem:good_event}. Thus we have, 
\begin{align}
    |\mathbb{E}_{\bbeta}[z_i |\bX]- \sqrt{n}\beta_i| \leq C \sqrt{p} \|\bbeta\|_2. \nonumber 
\end{align}

Thus as long as $\|\bbeta\|_2 = o\Big(\sqrt{\frac{\tau^*\log p}{p}}\Big)$, we have, 
\be
    \Pbeta{\Big[ |\bz|_{S(\bbeta)} < s \sqrt{2 \tau^* \log p}  \Big]} &\leq \Pbeta{\Big[ \sum_{i \in S(\bbeta)}|\sqrt{n}\beta_i + \tilde{\varepsilon}_i | < s \sqrt{2 \tau^* \log p}\,\, (1 +o(1))   \Big]} + 2 \exp{(-cp)} \nonumber \\
 &\leq \P\Big[ \sum_{i=1}^{s} |\sqrt{2r \log p} + \tilde{\varepsilon}_i| < s \sqrt{2 \tau^* \log p}(1+o(1)) \Big]  + 2 \exp{(-cp)}, \nonumber \ee
 Finally, Lemma \ref{lemma:folded_normal_exp} and \eqref{eq:type2_bound} implies that 
\be
    \Pbeta{\Big[\max_{|S|= s} |\mathbf{z}|_S < s \sqrt{2 \tau^* \log p}(1+o(1)) \Big]} \leq 2 \exp(- (\sqrt{r} - \sqrt{\tau^*})^2 (1+o(1))s\log p) + 2 \exp{(-cp)}. \\ \label{eq:type2_final}
\ee
Finally, combining \eqref{eq:type1_final} and \eqref{eq:type2_final} yields the required upper bound in case $\|\bbeta\|_2 = o\Big(\sqrt{\frac{\tau^*\log p}{p}}\Big)$ i.e. for any sequence $\xi_p\rightarrow 0$
\be 
\ & \P_{0}(T(\tau^*)=1)+\sup_{\bbeta\in \Xi(s,A): \|\bbeta\|^2\leq \xi_p\tau^*\frac{\log{p}}{p}}\P_{\bbeta}(T(\tau^*)=0)\\
&\leq \exp\left(-\frac{(r-\alpha)^2}{4r}s\log{p}(1+o(1))\right)+2 \exp(- (\sqrt{r} - \sqrt{\tau^*})^2 s\log p) + 2 \exp{(-cp)}\\
&\leq 5\exp\left(-\frac{(r-\alpha)^2}{4r}s\log{p}(1+o(1))\right)=5\exp\left(-\frac{p^{1/2+\delta}}{8}(1+o(1))\right).
\ee

\begin{lemma}\label{lemma:dense_above_boundary_bonferroni}
Assume $p^{2}/n\rightarrow 0$. Then for any constant $C>0$, there exists a sequence $\xi_p\rightarrow 0$ and a sequence of tests $T_{p}$ such that
\be 
\P_{0}(T_p=1)+\sup_{\bbeta\in \Xi(s,A): \|\bbeta\|^2\geq \xi_p\tau^*\frac{\log{p}}{p}}\P_{\bbeta}(T_p=0)\leq e^{-C\frac{p^{1/2+\delta}}{8}}.
\ee
\end{lemma}

The requisite upper bound can now be completed by appealing to  Lemma \ref{lemma:sparse_above_boundary_bonferroni} with $C>1/8$ (by considering a Bonferroni correction between $T(\tau^*)$ and $T_{p}$).

\end{proof}

\begin{proof}[Proof of Lemma \ref{lemma:dense_above_lowerbound}] The Lemma follows verbatim from the proof of Lemma \ref{lemma:above_boundary_sparse_lower} by taking $\tau^*=\frac{(r+\alpha)^2}{4r}\sim r/4$ with $r=\frac{p^{\alpha-1/2+\delta}}{2\log{p}}$. The details are omitted for the sake of avoiding repetition. 

\end{proof}

\begin{proof}[Proof of Lemma \ref{lemma:folded_normal_exp}]
\begin{enumerate}
    \item 
     Chernoff bound and Lemma \ref{lemma:folded_normal_expression} imply that  
\begin{align}
    \mathbb{P}\Big[ \sum_{i=1}^s |Z_i| > s \sqrt{2 \tau \log p} \Big] \leq 2^s \exp{\Big( s \frac{\lambda^2}{2}\Big)} \cdot \exp{ \Big(-\lambda s \sqrt{2 \tau \log p}\Big)}, \label{eq:int1}
\end{align}
where $\lambda >0$ is arbitrary. Setting $\lambda = \sqrt{2 \tau \log p}$, and plugging it back into \eqref{eq:int1}, we obtain the upper bound 
\begin{align}
    \mathbb{P}\Big[ \sum_{i=1}^s |Z_i| > s \sqrt{2 \tau \log p} \Big]\leq 3^s \exp{\Big(-\tau s \log p  \Big)}.\nonumber 
\end{align}
\noindent
This derives the required bound in this case. 
    \item Let $\varepsilon_1, \cdots, \varepsilon_s$ be i.i.d. $\mathcal{N}(0,1)$ random variables. Note that  
    \begin{align}
        \P\Big[ \sum_{i=1}^{s} | \mu_i + \varepsilon_i | < s \sqrt{2 \tau \log p} \Big]\leq \P\Big[ \sum_{i=1}^{s} | \sqrt{2r \log p} + \varepsilon_i | < s \sqrt{2 \tau \log p} \Big]. \nonumber
    \end{align}
    Chernoff's inequality implies that 
    \begin{align}
    \P\Big[ \sum_{i=1}^{s} | \sqrt{2r \log p} + \varepsilon_i | < s \sqrt{2 \tau \log p} \Big]\leq \exp{(\lambda s \sqrt{2\tau\log p} )}\,\, \Big(\E\Big[ \exp{\big(- \lambda |\sqrt{2r \log p} + \varepsilon_1 | \big)} \Big] \Big)^s. \nonumber 
    \end{align}
    Setting $\lambda = (\sqrt{r} - \sqrt{\tau}) \sqrt{2\log p}$, and simplification using Lemma \ref{lemma:folded_normal_expression} and the Mills ratio bound \cite{williams1991probability} yields the desired proof. 
    
\end{enumerate}

\end{proof}

\begin{proof}[Proof of Lemma \ref{lemma:sparse_above_boundary_bonferroni}]

For any $\tau>0$, consider first a test $T_{\tau}$ that rejects when $\frac{\|(\bX^T\bX)^{-\frac{1}{2}}\bX^T\by\|_2^2-p}{\sqrt{2p}}>\tau$.
Then as in the proof of Theorem \ref{theorem:above_boundary_dense} we have
\be 
\P_{0}\left(T_{\tau}=1\right)&=\P\left(\frac{\chi^2_p-p}{\sqrt{2p}}>\tau\right)&\leq e^{-\frac{\tau^2}{2}(1+o(1))}, \quad \text{if} \quad \tau=o(\sqrt{p}),
\ee
and 
under any $\bbeta \in \Xi(s,A)$ 
\be 
\P_{\bbeta}\left(T_{\tau}=0\right)
&\leq \P_{\bbeta}\left(\frac{\chi^2_p\left(\|(\bX^T\bX)^{\frac{1}{2}}\bbeta\|_2^2\right)-p}{\sqrt{2p}}\leq \tau,\mathcal{A}_{\kappa}\right)+\P_{\bbeta}(\mathcal{A}_{\kappa}^c),
\ee 
where for any $\kappa \in (0,1)$ we define
\be 
\mathcal{A}_{\kappa}=\left\{\|(\bX^T\bX)^{\frac{1}{2}}\bbeta\|_2^2\geq (1-\kappa)n\|\bbeta\|_2^2\right\}.
\ee
 Now
 , by Lemma \ref{lemma:rip}, there exists a $C>0$ (depending only on the subgaussian norm of the rows of $\bX$) such that for any $\kappa \geq C\frac{s\log{(p/s)}}{n}$
 \be 
 \P_{\bbeta}(\mathcal{A}_{\kappa}^c)&\leq 2e^{-cn},
 \ee
 for a $c>0$ (once again only depending only on the subgaussian norm of the rows of $\bX$). Moreover, by stochastic monotonicity of non-central chi-squares random variables in terms of the non-centrality parameter, we have for any $\bbeta\in \Xi(s,A):\|\bbeta\|^2\geq \xi_p\frac{\log{p}}{p}$
 \be 
 \P_{\bbeta}\left(\frac{\chi^2_p\left(\|(\bX^T\bX)^{\frac{1}{2}}\bbeta\|_2^2\right)-p}{\sqrt{2p}}\leq \tau,\mathcal{A}_{\kappa}\right)&\leq \P_{\bbeta}\left(\frac{\chi^2_p\left((1-\kappa)n\|\bbeta\|_2^2\right)-p}{\sqrt{2p}}\leq \tau\right)\\ &\leq \P_{\bbeta}\left(\frac{\chi^2_p\left(n(1-\kappa)\xi_p\frac{\log{p}}{p}\right)-p}{\sqrt{2p}}\leq \tau\right).
 \ee 
Now first assume $p^{7/4}\sqrt{\log{p}}\ll n\lesssim p^2\log{p}$. Also let $\epsilon>0$ be such that $p^{1-\alpha}\ll p^{1/2-\epsilon}\log{p}$ (which is feasible whenever $\alpha>\frac{1}{2}$). Then we have for $\tau^2=p^{1/2-\epsilon}\log{p}$ that $\tau=o(\sqrt{p})$ and hence for any constant $C>0$ 
\be 
\P_{0}(T_{\tau}=1)\leq e^{-Cs\log{p}(1+o(1))},
\ee
by the choice of $\epsilon>0$. Further by applying Lemma \ref{lemma:noncchisq_conc}, whenever $\frac{n(1-\kappa)\xi_p\log p}{p}\gg \tau\sqrt{p}$ (which holds for $p^{7/4}\sqrt{\log{p}}\ll n\lesssim p^2\log{p}$ and $\tau^2=p^{1/2-\epsilon}\log{p}$ by choosing $\xi_p\rightarrow 0$ slow enough), we have with  $x=\left(\frac{n(1-\kappa)\xi_p\log p}{p}\right)^2/2(p+2\frac{n(1-\kappa)\xi_p\log p}{p})$ that
\be
\sup_{\bbeta\in \Xi(s,A):\|\bbeta\|_2^2\geq \xi_p\frac{\log{p}}{p}}\P_{\bbeta}(T_{\tau}=0)&\leq e^{-x}
\ee
No note that $x\gtrsim \frac{n^2\xi_p^2(\log{p})^2}{p^3}\gg s\log{p}$ whenever $p^{7/4}\sqrt{\log{p}}\ll n\lesssim p^2/\xi_p\log{p}$ 
$\xi_p\rightarrow 0$ slow enough and $s\ll\sqrt{p}$. Finally when $n\gg p^2/\xi_p\log{p}$ we have $x\gtrsim \frac{n\xi_p\log p}{p}\gg p\xi_p\gg s\log{p}$ by choosing $\xi_p\rightarrow 0$ slow enough. This completes the proof of Lemma \ref{lemma:sparse_above_boundary_bonferroni}.

\end{proof}

\begin{proof}[Proof of Lemma \ref{lemma:dense_above_boundary_bonferroni}]
For any $\tau>0$, consider first a test $T_{\tau}$ that rejects when $\frac{\|(\bX^T\bX)^{-\frac{1}{2}}\bX^T\by\|_2^2-p}{\sqrt{2p}}>\tau$.
Then as in the proof of Theorem \ref{theorem:above_boundary_dense} we have
\be 
\P_{0}\left(T_{\tau}=1\right)&=\P\left(\frac{\chi^2_p-p}{\sqrt{2p}}>\tau\right)&\leq e^{-\frac{\tau^2}{2}(1+o(1))}, \quad \text{if} \quad \tau=o(\sqrt{p}),
\ee
and 
under any $\bbeta \in \Xi(s,A)$ 
\be 
\P_{\bbeta}\left(T_{\tau}=0\right)
&\leq \P_{\bbeta}\left(\frac{\chi^2_p\left(\|(\bX^T\bX)^{\frac{1}{2}}\bbeta\|_2^2\right)-p}{\sqrt{2p}}\leq \tau,\mathcal{A}_{\kappa}\right)+\P_{\bbeta}(\mathcal{A}_{\kappa}^c),
\ee 
where for any $\kappa \in (0,1)$ we define
\be 
\mathcal{A}_{\kappa}=\left\{\|(\bX^T\bX)^{\frac{1}{2}}\bbeta\|_2^2\geq (1-\kappa)n\|\bbeta\|_2^2\right\}.
\ee
 Now
 , by Lemma \ref{lemma:rip}, there exists a $C>0$ (depending only on the subgaussian norm of the rows of $\bX$) such that for any $\kappa \geq C\frac{s\log{(p/s)}}{n}$
 \be 
 \P_{\bbeta}(\mathcal{A}_{\kappa}^c)&\leq 2e^{-cn},
 \ee
 for a $c>0$ (once again only depending only on the subgaussian norm of the rows of $\bX$). Moreover, by stochastic monotonicity of non-central chi-squares random variables in terms of the non-centrality parameter, we have for any $\bbeta\in \Xi(s,A):\|\bbeta\|^2\geq \tau^* \xi_p\frac{\log{p}}{p}$
 \be 
 \P_{\bbeta}\left(\frac{\chi^2_p\left(\|(\bX^T\bX)^{\frac{1}{2}}\bbeta\|_2^2\right)-p}{\sqrt{2p}}\leq \tau,\mathcal{A}_{\kappa}\right)&\leq \P_{\bbeta}\left(\frac{\chi^2_p\left((1-\kappa)n\|\bbeta\|_2^2\right)-p}{\sqrt{2p}}\leq \tau\right)\\ &\leq \P_{\bbeta}\left(\frac{\chi^2_p\left(n(1-\kappa)\xi_p\tau^*\frac{\log{p}}{p}\right)-p}{\sqrt{2p}}\leq \tau\right).
 \ee 
 Now recall that $\tau^*\sim r=O(p^{\alpha-1/2+\delta}/\log{p})$ and therefore $n(1-\kappa)\xi_p\tau^*\frac{\log{p}}{p}=O(np^{\alpha-3/2+\delta})$. Also note that it is enough to take $\tau=O(\sqrt{p^{1/2+\delta}})$ (which is allowed since $\delta<\frac{1}{2}$ implies $\tau=o(\sqrt{p})$) in order to prove the theorem. Therefore, with any such $\tau$ we have for some slowly decaying sequence $\xi_p$ that $n(1-\kappa)\xi_p\tau^*\frac{\log{p}}{p}\gg \tau\sqrt{p}$ provided $n\gg p^{2}$. Finally by Lemma \ref{lemma:noncchisq_conc} we have
 \be 
 \P_{\bbeta}\left(\frac{\chi^2_p\left(n(1-\kappa)\xi_p\tau^*\frac{\log{p}}{p}\right)-p}{\sqrt{2p}}\leq \tau\right)&\leq \exp\left[-\frac{\left(n(1-\kappa)\xi_p\tau^*\frac{\log{p}}{p}-\tau\sqrt{2p}\right)^2}{4(p+n(1-\kappa)\xi_p\tau^*\frac{\log{p}}{p})}\right]\\
 &\leq \exp\left[-\xi_p^2\frac{n^2p^{2\alpha-3+2\delta}}{4p}\right]+\exp\left[-\xi_p\frac{np^{\alpha-3/2+\delta}}{4p}\right]
 \ee
 Now note that since $n\gg p^2$, there exists $\xi_p$ slow enough such that $\min\{ \xi_p\frac{np^{\alpha-3/2+\delta}}{p}  \xi_p^2\frac{n^2p^{2\alpha-3+2\delta}}{p}\}\gg p^{1/2+\delta}$. This completes the proof of the lemma. 
 \end{proof}

\begin{proof}[Proof of Lemma \ref{lem:boundary_null}]
Recall $\mathbf{z} = (\bX^{T} \bX)^{-1/2} \bX^{T}\mathbf{y}$, and define 
\begin{align}
    \tilde{L}_{\pi} = \mathbb{E}_{\bbeta\sim \pi}[ L(\bbeta) \mathbf{1}\big(\max_{i\in \mathrm{supp}(\bbeta)} z_i < \sqrt{2\log p}\big)]. \nonumber 
\end{align}
 For the subsequent analysis, recall the good event $\mathcal{G}_2$ introduced in Lemma \ref{lem:good_event}. 

We have, 
\be
    \mathbb{P}_0[L_{\pi} >1] \leq \mathbb{E}_0\Big[ \mathbf{1}_{\mathcal{G}_2} \mathbb{P}_0[L_{\pi}>1| \bX] \Big] + 2 \exp(-cp). 
\ee
By the Bounded Convergence Theorem, it suffices to prove that $\mathbb{P}[L_{\pi} > 1|\bX] \to 0$ in probability on the good event $\mathcal{G}_2$. To this end, consider the decomposition 
\be
    L_{\pi} = \mathbb{E}_0[\tilde{L}_{\pi}|\bX] + L_{\pi}- \tilde{L}_\pi + \tilde{L}_\pi - \mathbb{E}_0[\tilde{L}_{\pi}|\bX]. \label{eq:boundary_int1} 
\ee

First, note that using Fubini's theorem, 
\begin{align}
    \mathbb{E}_0[\tilde{L}_{\pi}| \bX ] = \mathbb{E}_{\bbeta \sim \pi }[\mathbb{P}_{\bbeta}[\max_{i \in \mathrm{supp}(\bbeta)} z_i < 2\log p | \bX ] ]. \nonumber
\end{align}
Under $\mathbb{P}_{\bbeta}$, given $\bX$, $\mathbf{y} \sim \mathcal{N}(\bX \bbeta, I)$. Reasoning exactly as in the proof of the upper bound, we have, for $i \in \mathrm{supp}(\bbeta)$, on the good event, $z_i = \sqrt{2\log p} + O(\sqrt{\frac{p \log p }{n}} ) + \eta_i$, where $\eta_i \sim \mathcal{N}(0,1)$ are iid. Thus we see immediately that $\mathbb{E}_0[\tilde{L}_\pi] \to \Big( \frac{1}{2} \Big)^s$ as $n,p \to \infty$. This controls the first term in \eqref{eq:boundary_int1}. Next, note that 
\begin{align}
    \mathbb{P}_0[L_{\pi} - \tilde{L}_{\pi}>0| \bX] \leq \sum_{i=1}^{p} \mathbb{P}_0[z_i > \sqrt{2\log p}| \bX] = p \bar{\Phi}(\sqrt{2\log p}) \to 0. \nonumber 
\end{align}
Finally, we turn to the last term in \eqref{eq:boundary_int1}. Observe that it suffices to prove that for all $\varepsilon>0$, 
\begin{align}
    \mathbb{P}_0[|\tilde{L}_{\pi} - \mathbb{E}_0[\tilde{L}_{\pi}|\bX]| > \varepsilon| \bX] \to 0 \nonumber  
\end{align}
in probability on the event $\mathcal{G}_2$. By the conditional Chebychev inequality, we have, 
\begin{align}
   \mathbb{P}_0[|\tilde{L}_{\pi} - \mathbb{E}_0[\tilde{L}_{\pi}|\bX]| > \varepsilon| \bX] \leq \frac{\mathrm{Var}_0[\tilde{L}_\pi|\bX]}{\varepsilon^2}. \nonumber  
\end{align}
In turn, it suffices to prove that $\mathrm{Var}_0[\tilde{L}_\pi|\bX]= o(1)$ on the event $\mathcal{G}_2$. To this end, we observe that 
\begin{align}
    \tilde{L}_\pi = \frac{1}{{p \choose s}} \sum_{|S|=s} \exp\Big( \langle \mathbf{y}, \bX\bbeta_S \rangle - \frac{1}{2}\|\bX \bbeta_S \|^2  \Big) \mathbf{1}\Big(\max_{i \in S}z_i <\sqrt{2\log p}\Big):= \frac{1}{{p \choose s}} \sum_{|S|=s} \zeta_S, \nonumber 
\end{align}
where $\bbeta_S$ denotes the vector constructed by setting the entries in $S$ to $\sqrt{\frac{2\log p}{n}}$. 
Now, we have, 
\begin{align}
    \mathrm{Var}_0[\tilde{L}_\pi|\bX] = \frac{1}{{p \choose s}^2} \Big[ \sum_{|S|=s} \mathrm{Var}_0[\zeta_S|\bX] + \sum_{S\neq S': |S|=|S'|=s} \mathrm{Cov}_0[\zeta_S \zeta_{S'}|\bX] \Big]. \nonumber 
\end{align}
This implies 
\begin{align}
    \mathrm{Var}_0[\zeta_S|\bX] &\leq \mathbb{E}_0[\zeta^2_S |X] \nonumber \\
    &= \exp(\|\bX \bbeta_S \|_2^2) \mathbb{P}_0\Big[ \exp\Big( \langle \mathbf{y}, 2\bX\bbeta_S \rangle - \frac{1}{2}\|2\bX \bbeta_S \|_2^2  \Big) \mathbf{1}\Big(\max_{i \in S}z_i <\sqrt{2\log p}\Big) |\bX\Big]\nonumber \\
    &= \exp(\|\bX \bbeta_S \|_2^2)\mathbb{P}_{2\bbeta_S}\Big[ \max_{i \in S} z_i < \sqrt{2\log p}| \bX \Big], \nonumber 
\end{align}
so that under $\mathbb{P}_{2\bbeta_S}$, given $X$, 
$y\sim \mathcal{N}(2\bX \bbeta_S, I)$. On the event $\mathcal{G}_2$, 
\begin{align}
    \|\bX \bbeta_S \|_2^2 \leq \Big(1 + C \sqrt{\frac{p}{n}}\Big) n \|\bbeta\|_2^2 = \Big(1 + C \sqrt{\frac{p}{n}}\Big) 2s \log p. \nonumber 
\end{align}
Further, under $\mathbb{P}_{2\bbeta_S}$, given $\bX$,  $z= 2 \Big(\frac{\bX^{T}\bX}{n} \Big)^{1/2} \sqrt{n} \bbeta_S + \eta$, where $\eta \sim \mathcal{N}(0,I)$. For $i \in \mathrm{supp}(\bbeta)$, on the event $\mathcal{G}_2$, $z_i = 2 \sqrt{2\log p} + O\Big(\sqrt{ \frac{p\log p}{n}} \Big) + \eta_i$. Thus 
\begin{align}
    \mathbb{P}_{2\bbeta_S}\Big[ \max_{i \in S} z_i < \sqrt{2\log p}| \bX \Big] \leq \Phi\Big( - \sqrt{2\log p} +O\Big( \sqrt{\frac{p\log p}{n}} \Big)   \Big)^s. \nonumber 
\end{align}
Combining, we have, 
\begin{align}
    \frac{1}{{p \choose s}^2}\sum_{|S|=s} \mathrm{Var}_0[\zeta_S | \bX] \leq \frac{1}{{p \choose s}} \exp\Big(2s \log p \Big( 1 + C \sqrt{\frac{p}{n}} \Big) \Big)\Phi\Big( - \sqrt{2\log p} +O\Big( \sqrt{\frac{p\log p}{n}} \Big)   \Big)^{s} \to 0 \nonumber 
\end{align}
as $n,p \to \infty$ with $n \gg p(\log p)^2$. Next, we turn to the covariance. For sets $S,S' \subset [p]$ with $|S|=|S'| =s$ and $|S\cap S'|=k >0$, we have, 
\begin{align}
  \mathrm{Cov}_0[\zeta_S \zeta_{S'}|\bX] &\leq \mathbb{E}_0[\zeta_S \zeta_{S'}|\bX]
  = \exp(\langle \bX \bbeta_S , \bX \bbeta_{S'}\rangle ) \times \nonumber\\ 
  &\mathbb{E}_0\Big[\exp\Big( \langle \mathbf{y}, \bX(\bbeta_{S} + \bbeta_{S'}) \rangle - \frac{1}{2} \| \bX(\bbeta_S + \bbeta_{S'}) \|_2^2 \Big) \mathbf{1}\Big(\max_{i \in S \cup S'}z_i <\sqrt{2\log p}\Big)|\bX\Big] \nonumber \\
  &= \exp\Big(\langle \bX \bbeta_S , \bX \bbeta_{S'}\rangle \Big) \mathbb{P}_{\bbeta_{S} +\bbeta_{S'}}\Big[ \max_{i \in S \cup S'}z_i <\sqrt{2\log p} |\bX \Big], \nonumber 
\end{align}
Observe that on the good event, 
\begin{align}
    \langle \bX \bbeta_S, \bX \bbeta_{S'} \rangle = 2k \log p + O\Big(\sqrt{\frac{p(\log p)^2}{n}} \Big). \nonumber 
\end{align}
This implies, on the good event, \textcolor{black}{since we have that $n \gg p (\log p)^2$} 
\begin{align}
    \mathrm{Cov}_0[\zeta_S \zeta_{S'}|\bX] \leq (1+o(1)) p^{2k} \mathbb{P}_{\bbeta_S + \bbeta_{S'}}\Big[ \max_{i \in S_1 \cap S_2} z_i < 2 \log p |\bX \Big]. \nonumber 
\end{align}
For $i \in S_1 \cap S_2$, under $\mathbb{P}_{\bbeta_{S} + \bbeta_{S'}}$, 
\begin{align}
    z_i = 2 \sqrt{2\log p} + O\Big(\sqrt{\frac{p\log p}{n}} \Big) + \eta_i, \nonumber 
\end{align}
where $\eta_i$ are iid $\mathcal{N}(0,1)$ random variables. Thus 
\begin{align}
    &\frac{1}{{p \choose s}^2} \sum_{|S|=|S'|=s, |S\cap S'|>0} \mathrm{Cov}_0(\zeta_{S}\zeta_{S'}) \nonumber \\
    &= (1+o(1)) \sum_{k=1}^{s-1} \mathbb{P}[\mathrm{Hyp}(p,s,s)=k] p^{2k} \Big( \Phi\Big(- \sqrt{2\log p} + O\Big(\sqrt{\frac{p\log p}{n}} \Big)\Big) \Big)^k = o(1), \nonumber 
\end{align}
where the last equality follows upon using well-known Mills ratio bounds. 

Finally, it remains to handle the case where $S,S'$ are disjoint samples. If $|S|=|S'|=s$ and $|S\cap S'|=0$, on the good event, we have, 
\begin{align}
   \mathrm{Cov}_0[\zeta_S \zeta_{S'}|\bX]= \mathbb{E}_0[\zeta_S \zeta_{S'}|\bX] - \mathbb{E}_0[\zeta_S|\bX] \mathbb{E}_0[\zeta_{S'}|\bX]. \nonumber 
\end{align}
Proceeding as before, we have, on the good event, $\mathbb{E}_0[\zeta_S|X] \to \Big(\frac{1}{2}\Big)^s$. Similarly, recall that, 
\begin{align}
    \mathbb{E}_0[\zeta_S \zeta_{S'}|\bX] &= \exp\Big(\langle \bX \bbeta_S , \bX \bbeta_{S'}\rangle \Big) \mathbb{P}_{\bbeta_{S} +\bbeta_{S'}}\Big[ \max_{i \in S \cup S'}z_i <\sqrt{2\log p} |\bX \Big] \nonumber \\
    &=(1+o(1)) \Big(\frac{1}{2} \Big)^{2s}. \nonumber 
\end{align}
This implies, 
\be
    \frac{1}{{p \choose s}^2}\sum_{|S|=|S'|=s, |S\cap S'|=0} \mathrm{Cov}_0[\zeta_S \zeta_{S'}|\bX] \to 0
\ee
as $n,p \to \infty$. 
This completes the proof. 
\end{proof}

\begin{proof}[Proof of Lemma \ref{lem:boundary_alternative}] 
First note that $\mathbb{P}_{\bbeta}[L_{\pi}\leq 1]$ is constant for all $\bbeta$ sampled from $\pi$, and thus it suffices to prove the thesis for any fixed $\bbeta$ in the support of the prior. To this end, for any subset $S\subset [p]$, denote by $\bbeta_S$ a vector where the entries corresponding to $S$ are set at $\sqrt{\frac{2\log p}{n}}$, and are zero otherwise. For the subsequent discussion, we will derive the required asymptotics under $\mathbb{P}_{\bbeta_{[s]}}$. Next, recall the good event $\mathcal{G}_2$ from Lemma \ref{lem:good_event}, and observe that \begin{align}
    \mathbb{P}_{\bbeta_{[s]}}(L_{\pi} \leq 1) = \mathbb{E}\Big[\mathbf{1}_{\mathcal{G}_2} \mathbb{P}_{\bbeta_{[s]}}\Big[ L_{\pi} \leq 1 | \bX \Big]\Big] + o(1). \nonumber 
\end{align}
Armed with the notation introduced above, we have, 
\begin{align}
    L_{\pi} = \frac{1}{{p \choose s}} \sum_{|T|=s} \exp\Big(\langle \by , \bX \bbeta_{T}\rangle - \frac{1}{2}\|\bX \bbeta_{T} \|_2^2 \Big), \nonumber 
\end{align}
and under $\mathbb{P}_{\bbeta_{[s]}}$, $\by = \bX \bbeta_{[s]} + \beps$. This implies $\langle \by, \bX \bbeta_{T} \rangle= \langle \bX \bbeta_{[s]}, \bX \bbeta_{T}\rangle+ \langle \beps, \bX \bbeta_{T}\rangle$. On the event $\mathcal{G}_2$, $\langle \bX \bbeta_{[s]}, \bX \bbeta_{T} \rangle = n \langle \bbeta_{[s]} , \bbeta_{T}\rangle + O\Big(\sqrt{\frac{p(\log p)^2}{n}} \Big) = 2 |[s]\cap T| \log p + O\Big(\sqrt{\frac{p(\log p)^2}{n}} \Big)$ since $s=O(1)$. Further, we have, 
\begin{align}
    \| \bX \bbeta_T \|_2^2 = \| \bX \bbeta_{T\cap [s]}\|_2^2 + \| \bX \bbeta_{T \cap [s]^c} \|_2^2 + 2 \langle \bX \bbeta_{T \cap [s]}, \bX \bbeta_{T\cap [s]^c}\rangle. \nonumber  
\end{align}
One the event $\mathcal{G}$, 
\begin{align}
    \| \bX \bbeta_{T\cap [s]} \|_2^2 &= 2 |[s]\cap T| \log p + O\Big(\sqrt{\frac{p(\log p)^2}{n}} \Big), \nonumber \\
    \Big| \langle \bX \bbeta_{T \cap [s]}, \bX \bbeta_{T\cap [s]^c}\rangle \Big| & = O\Big(\sqrt{\frac{p(\log p)^2}{n}} \Big). \nonumber 
\end{align}
Combining, we have, on the event $\mathcal{G}$, 
\begin{align}
    L_{\pi} &= \frac{(1+o(1))}{{p\choose s}} \sum_{|T|=s} p^{|[s]\cap T|} \exp\Big( \langle \beps, \bX \bbeta_T\rangle - \frac{1}{2} \| \bX \bbeta_{T \cap [s]^c} \|_2^2 \Big) \nonumber \\
    &= (1+o(1)) \sum_{k=0}^{s} \sum_{T_1\subseteq [s], |T_1|=k} \exp\Big(\langle \beps, \bX \bbeta_{T_1}\rangle\Big) \frac{1}{{p \choose s-k}} \sum_{T_2 \subseteq [s]^c, |T_2| = s-k} \exp\Big(\langle \beps, \bX \bbeta_{T_2} \rangle - \frac{1}{2} \|\bX \bbeta_{T_2} \|_2^2 \Big). \nonumber 
\end{align}
The proof of Lemma \ref{lem:boundary_null} establishes that there exists universal constants $0<c<C<\infty$ such that with high probability as $n,p \to \infty$, for all $0\leq k \leq s$
\begin{align}
  c < \frac{1}{{p \choose s-k}} \sum_{T_2 \subseteq [s]^c, |T_2| = s-k} \exp\Big(\langle \beps, \bX \bbeta_{T_2} \rangle - \frac{1}{2} \|\bX \bbeta_{T_2} \|_2^2 \Big) < C. \nonumber 
\end{align}
This implies 
\begin{align}
    \mathbb{P}_{\bbeta_{[s]}}\Big[L_{\pi} \leq 1 |\bX \Big] =  \mathbb{P}_{\bbeta_{[s]}}\Big[ \frac{1}{\sqrt{2\log p}}\log L_{\pi} \leq 0 \Big|\bX\Big]. \nonumber 
\end{align}
Finally, we use the following elementary lemma about real numbers. We defer its proof to the end of this section. 

\begin{lemma}\label{lemma:soft_max}
Fix $N \geq 1$ and $a_1, \cdots , a_N$ such that $c < a_ i < C$ for all $1  \leq i \leq N$ and universal constants $0< c < C < \infty$. Then we have, for any sequence $M_n \to \infty$, 
\begin{align}
\Big| \frac{1}{M_n} \log \Big( \sum_{i=1}^{N} a_i \exp(x_i M_n)  \Big) - \max_{1\leq i \leq N} x_i \Big| \lesssim \frac{1}{M_n}. \nonumber
\end{align}
\end{lemma}
An application of Lemma \ref{lemma:soft_max} with $M_n= \sqrt{2\log p}$ immediately yields that 
\begin{align}
    \mathbb{P}_{\bbeta_{[s]}}\Big[L_{\pi} \leq 1 |\bX \Big] = \mathbb{P}_{\bbeta_{[s]}} \Big[\Big\langle \frac{\bX^{T}\beps}{\sqrt{n}}, e_{T}\Big\rangle  \leq o(1) \,\,\,\forall\,\, T \subseteq [s] \Big] = \mathbb{P}_{\bbeta_{[s]}} \Big[\Big\langle \frac{\bX^{T}\beps}{\sqrt{n}}, e_{j}\Big\rangle  \leq o(1) \,\,\,\forall\,\, j \in [s] \Big]. \nonumber  
\end{align}
Now, observe that $\Big(\Big\langle \frac{\bX^{T}\beps}{\sqrt{n}}, e_{j}\Big\rangle: 1 \leq j \leq s\Big)\sim \mathcal{N}(0,\Sigma_s)$, where $\Sigma_s = e_{[s]}^{T} \frac{\bX^{T}\bX}{n} e_{[s]}$. As $s$ is fixed, under the event $\mathcal{G}$, $\Sigma_s\to I_{s\times s}$ under Frobenius norm, and thus the collection  $\Big(\Big\langle \frac{\bX^{T}\beps}{\sqrt{n}}, e_{j}\Big\rangle: 1 \leq j \leq s\Big)$ converges in distribution to $\mathcal{N}(0,I_{s\times s})$. Finally, this implies
\begin{align}
    \mathbb{P}_{\bbeta_{[s]}}[L_{\pi} \leq 1|\bX] \to \Big(\frac{1}{2} \Big)^s. \nonumber 
\end{align}
This completes the proof. 
\end{proof}

It remains to prove Lemma \ref{lemma:soft_max}. 
\begin{proof}[Proof of Lemma \ref{lemma:soft_max} ]
We have the inequalities 
\begin{align}
&c \exp\Big( M_n \max_{1 \leq i \leq N} x_i \Big) \leq \sum_{i=1}^{N} a_i \exp\Big( x_i M_n \Big) \leq NC\exp\Big( \max_{1 \leq i \leq N} x_i M_n \Big)   \nonumber \\
\equiv&\frac{1}{M_n} \log c + \max_{1 \leq i \leq N} x_i \leq \frac{1}{M_n} \log \Big(  \sum_{i=1}^{N} a_i \exp(M_n x_i )  \Big) \leq \frac{1}{M_n} \log (CN) + \max_{1 \leq i \leq N} x_i. \nonumber
\end{align}
and the proof follows immediately. 
\end{proof}

\end{document}